\documentclass[11pt]{article}
\usepackage[left=3cm, right=3cm, top=2cm, bottom=2cm]{geometry}
\usepackage{mathtools}
\usepackage{eucal}
\usepackage{hyperref}
\usepackage[utf8x]{inputenc}
\usepackage{amsthm, amsmath, amsfonts, amssymb, mathrsfs, mathabx}
\usepackage[framemethod=TikZ]{mdframed}
\usepackage{breqn}
\usepackage{tikzsymbols}
\usepackage{titling} %->Title page with an abstract
\usepackage{setspace}
\usepackage{bbm}
\usepackage{array}
\usepackage{bigints}
\usepackage{relsize}
\usepackage{xcolor} %changing color to white/or anthing else

\usepackage{multicol}
\usepackage{lipsum} %Making font size smaller
\usepackage{lmodern} %Making font size smaller

\usepackage[refpage]{nomencl}

\makenomenclature

\hypersetup{colorlinks=true, linkcolor=blue, filecolor=blue, urlcolor=blue,}
\newtheorem{definition}{Definition}[section]
\newtheorem{corollary}{Corollary}[section]
\newtheorem{lemma}{Lemma}[section]

\newtheorem{theorem}{Theorem}[section]

\newtheorem{lemma1}[definition]{Lemma}
\newtheorem{corollary1}[theorem]{Corollary}

\newtheorem*{Assumptions*}{Assumptions} % Unnumbered theorems

\newtheorem*{theorem*}{Theorem} % Unnumbered theorems

\theoremstyle{remark} % Another style for Remarks.
\newtheorem*{remark}{\textbf{Remark}} % Another style for Remarks.

{%For Cardinality
\newcolumntype{?}{!{\vrule width 2pt}}%For Cardinality

\newcommand{\reals}{\mathbb{R}}
\newcommand{\naturals}{\mathbb{N}}
\newcommand{\timeint}{\mathcal{T}}
\newcommand{\indexint}{\mathcal{A}}
\newcommand{\inta}{\mathfrak{a}}

\title{Row-finite systems of stochastic differential equations with dissipative drift}
\author{Georgy Chargaziya \\ Mathematics Department, The University of York}
\doublespacing
\setlength\parindent{0pt}
\numberwithin{equation}{section}
\begin{document}
\maketitle
\begin{center}
\textbf{Abstract} 
\end{center}
Motivated by studies of stochastic systems describing non-equilibrium dynamics of (real-valued) spins of an infinite particle system in $\reals^{n}$ we consider a row-finite system of stochastic differential equations with dissipative drift. The existence and uniqueness of infinite time solutions is proved via finite volume approximation and a version of the Ovsjannikov method. 
\newpage
\tableofcontents
\newpage
\begin{multicols}{3}
\printnomenclature
\addcontentsline{toc}{section}{List of Symbols}
\end{multicols}
\newpage
\textbf{Acknowledgment} \\[1em]
Much of this work is based on joint research with Alexei Daletskii. His advise and guidance is gratefully acknowledged. I am grateful to Zdzislaw Brzezniak for helpful and stimulating discussions.
\newpage

\section{Introduction}\label{Introduction}
The study of properties of various physical phenomena has led to consideration of systems of infinitely many coupled finite dimensional stochastic differential equations. Such systems are known as lattice models with certain conditions on the so-called
\textquotedblleft spin variables\textquotedblright, which are being modelled by the SDEs. Term \textquotedblleft stochastic dynamics\textquotedblright\ is also often used to describe such systems in general and in particular SDEs that model the time dependence of spin variables. Origins of this terminology can be found in \cite{GPYW} and additional mathematical framework can be found, for example, in \cite{SAMR} and \cite{GJLPKM}. Questions concerning existence and uniqueness of solutions of such systems have also been studied in \cite{RLD} and \cite{GR}. \\[1em]
In recent decades studies of physical phenomena pertaining to non-crystalline (amorphous) substances and ferrofluids and amorphous magnets has led to an increased interest in studying countable systems of particles randomly distributed in $\reals^{d}$. Characterisation of each particle in such a system by an internal real or vector valued \textquotedblleft spin\textquotedblright\ parameter naturally leads to the consideration of a lattice model based on a fixed configuration $\gamma\subset\reals^{d}$ of particle positions. Instances when $\gamma\equiv\mathbb{Z}^{d}$ are well studied and have an extensive literature, see for example \cite{DZ1,LLL} and \cite{INZ}. However, as described in \cite{AD&DF} there are instances when the configuration $\gamma$ of particle positions doesn't have a regular structure but instead lends itself as a locally finite subset of $\reals^{d}$ where the typical number of \textquotedblleft neighbour variables\textquotedblright\ of a particle located at $x\in\gamma$ is proportional to $\log|x|$ for large $|x|$. \\[1em]
In \cite{ADGC} we saw an extension of work by \cite{AD}. This extension showed, under a suitable choice of coefficients, how to construct a unique strong solution of a stochastic differential equation, driven by a cylinder Wiener process, in a separable Hilbert space 
\begin{equation}
d\xi (t)=F(\xi (t))dt+\Phi(\xi (t))dW(t),\ t\geq0,
\end{equation}%
using the method of Ovsjannikov. The end result was a strong solution that takes values in an intersection of a suitably chosen scale of Hilbert spaces. This general theory was subsequently used to extend the work of \cite{AD&DF} [\textit{in a sense of considering a stochastic version}] by considering a lattice system on a locally finite subset $\gamma\subset\reals^{d}$ such that the spins $q_{x}$ and $q_{y}$ are allowed to interact via a pair potential if the distance between $x,y\in\gamma$ is no more than a fixed and positive interaction radius $r$, that is, they are neighbours in the geometric graph defined by $\gamma$ and $r$. Precisely speaking we considered a system
\begin{align}
d\xi_{x}(t)=\phi_{x}(\Xi(t))dt+\Psi_{x}(\Xi(t))dW_{x}(t),\ x\in\gamma,\ t\in [0,T], \label{prevworksys}
\end{align}
where $\phi_{x}$ and $\Psi_{x}$ were required to satisfy the so-called \textquotedblleft finite range \textquotedblright\ and \textquotedblleft uniform Lipschitz continuity\textquotedblright\ conditions and showed that system ($\ref{prevworksys}$) can be realised in a suitable scale of separable Hilbert spaces and hence studied using the method of Ovsjannikov. \\[1em]
In this paper, we would like to further build upon results of \cite{AD, AD&DF} and \cite{AKTR,AKT} and consider a lattice system of the form 
\begin{align}
d\xi_{x}(t)=\Phi_{x}(\xi_{x}(t),\Xi(t))dt+\Psi_{x}(\xi_{x}(t),\Xi(t))dW_{x}(t),\ x\in\gamma,\ t\in [0,T], \label{prevworksys1}
\end{align}
where $\Phi_{x}(a,b)\equiv V(a) + \phi_{x}(b)$, where $V$ is a real valued one particle potential satisfying the dissipativity condition, and $\Psi_{x}$ is Lipschitz. In our approach we will assume, as in \cite{AD&DF}, that configuration of particles $\gamma\subset\reals^{d}$ is a locally finite subset of $\reals^{d}$ distributed according to a Poisson or, more generally, Gibbs measure with a superstable low regular interaction energy, so that for all $x\in\gamma$ then number of particle in a certain compact vicinity of $x$ is proportional to $\log|x|$ for large $|x|$.  \\[1em]
Unfortunately, system (\ref{prevworksys1}) doesn't lend itself for an immediate and straightforward application of an Ovsjannikov method. Hence in this part we opt for an approach that was used in \cite{AKTR} and consider a so-called sequence of \textquotedblleft finite volume approximations\textquotedblright\ of the system (\ref{prevworksys1}). Precisely speaking a sequence of finite volume approximations is a sequence of solutions of truncated systems of the following form
\begin{equation}
\begin{alignedat}{2} 
&\xi_{x,t}^{n} = \zeta_{x} + \int_{0}^{t}\Phi_{x}(\xi_{x,s}^{n} ,\Xi_{s}^{n} )ds + \int_{0}^{t}\Psi_{x}(\xi_{x,s}^{n} ,\Xi_{s}^{n})dW_{x}(s), && \quad \forall x\in\Lambda_{n}\land t\in  [0,T], \\
&\xi_{x,t}^{n}  = \zeta_{x}, &&\quad \forall x\not\in\Lambda_{n}\land t\in [0,T].
\end{alignedat}
\end{equation}
where $\gamma\supset\Lambda_{n}\uparrow\gamma$ are finite. Using a comparison Theorem \ref{OvsMapTheorem}, which builds upon the method of Ovsjannikov, we ultimately show that the sequence of finite volume approximations converges to a unique strong solution of the system (\ref{prevworksys1})  in a certain scale of Banach spaces. 

\section{Main Framework}\label{MF}
\subsection{General Notation}\label{notation}

\nomenclature[$T$]{$T$}{}
\nomenclature[$\rho$]{$\rho$}{}
\nomenclature[$d$]{$d$}{}
\nomenclature[$\reals^{+}$]{$\reals^{+}$}{}
\nomenclature[$\reals^{++}$]{$\reals^{++}$}{}
\nomenclature[$\reals^{0}$]{$\reals^{0}$}{}
\nomenclature[$\naturals^{0}$]{$\naturals^{0}$}{}

\nomenclature[$\gamma$]{$\gamma$}{}
\nomenclature[$\mathcal{T}$]{$\mathcal{T}$}{}
\nomenclature[$\mathcal{A}$]{$\mathcal{A}$}{}
\nomenclature[$\overline{B(\cdot,\cdot)}$]{$\overline{B(\cdot,\cdot)}$}{}
\nomenclature[$B(\cdot,\cdot)$]{$B(\cdot,\cdot)$}{}

\nomenclature[$B_{x}$]{$B_{x}$}{}
\nomenclature[$n_{x}$]{$n_{x}$}{}
\nomenclature[$a$]{$a$}{}
\nomenclature[$\bar{a}$]{$\bar{a}$}{}
\nomenclature[$|\cdot|_{E}$]{$|\cdot|_{E}$}{}
\nomenclature[$|\cdot|_{S}$]{$|\cdot|_{S}$}{}

In our framework all vector spaces will be over $\reals$ and the cardinal number of any given set $A$ will always be denoted by $\#A$. We now start this subsection by introducing the following sets
\begin{align}
\reals^{0}\coloneqq (0,\infty), \quad \reals^{+}\coloneqq [0,\infty), \quad \reals^{++}\coloneqq [1,\infty), \quad \naturals^{0}\coloneqq\naturals\cup\{0\},
\end{align}
We also introduce constants $T, \ \rho, \ \underline{\inta}, \ \overline{\inta} \in \reals^{0}$, $d\in\reals^{+}$ and a special notation for the following closed intervals;
\begin{align}
\mathcal{A}\coloneqq[\underline{\inta},\overline{\inta}], \\[1em]
\mathcal{T}\coloneqq [0,T].
\end{align}
We let $\gamma$ be a locally finite subset of $\reals^{d}$ and $|\cdot|,\ |\cdot|_{S}$ be respectively the Euclidean and supremum norm in $\reals^{d}$. Moreover we 
agree to use the following abbreviations;
\begin{equation}
\begin{alignedat}{2}
B(x,\rho)&\coloneqq\{y\in\reals^{d} \ | \ |x-y&&|\ <\rho\}, \\[1em]
\overline{B(x,\rho)}&\coloneqq\{y\in\reals^{d} \ | \ |x-y&&|\ \leq\rho\}, \\[1em]
B_{x}&\coloneq\gamma\cap \overline{B(x,\rho)}, &&\!  \forall x\in\gamma, \\[1em]
n_{x}&\coloneq \#B_{x}, &&\!  \forall x\in\gamma.
\end{alignedat}
\end{equation}
\begin{mdframed}
	\begin{remark}
	The fact that $\gamma$ is a locally finite subset of $\reals^{d}$ means that $\gamma\cap X$ is finite if $X\subset\reals^{d}$ is compact and also implies that $\gamma$ is a countable subset of $\reals^{d}$. 
	\end{remark}
\end{mdframed}
Next, we fix in place a real valued function $a:\reals^{d}\to\reals^{+}$ and make the following assumptions;
\begin{enumerate}
	\item[(\textbf{A})] $a(x)\leq \bar{a}\quad\quad\quad\quad$\ \ \ for some constant $\bar{a}\in\reals^{+}$, \label{boundfora}
	\item[(\textbf{B})] $n_{x}\leq \mathcal{N}\log(1+|x|)$ for some constant $\mathcal{N}\in\reals^{+}$ and all $x\in\gamma$. \label{EstimateOnN}
\end{enumerate}

Given two vector spaes $A$ and $B$ let us now also introduce the following notation 
\begin{align}
A \prec B \iff \text{A is a subspace of B,}
\end{align}
and agree that $A^{B}$ will be understood as the infinite cartesian product, that is
\begin{align}
A^{B} = \bigtimes_{b\in B}A = \bigg{\{}\{z_{b}\}_{b\in B} \ \bigg{|} \ z_{b}\in A \ \text{for all} \ b\in B\bigg{\}}.
\end{align}
Suppose now that $\mathbf{X}\coloneqq\{X_{\inta}\}_{\inta \in\indexint}$ is a family of sets. We define for convenience and later use the following notation; 
\begin{align}
&\mathbf{X}(\cup)\coloneqq\bigcup_{\inta\in(\underline{\inta},\overline{\inta})}X_{\inta}, \\[1em]
&\mathbf{X}(\cap)\coloneqq\bigcap_{\inta\in(\underline{\inta},\overline{\inta})}X_{\inta},
\end{align}

\nomenclature[$A^{B}$]{$A^{B}$}{}
\nomenclature[$\mathbf{X}(\cup)$]{$\mathbf{X}(\cup)$}{}
\nomenclature[$\mathbf{X}(\cap)$]{$\mathbf{X}(\cap)$}{}
\nomenclature[$\prec$]{$\prec$}{}

\subsection{Scale and Ovsjannikov Map}\label{SandOM}
We now proceed to introduce several important definitions.
\begin{definition}\label{Defscale}
	A family $\mathbf{X}\coloneqq\{X_{\inta}\}_{\inta \in\indexint}$ of Banach spaces is called a scale if for all $\alpha < \beta \in \indexint$ and all $x\in X_{\alpha}$
	\begin{enumerate}
		\item $X_{\alpha}  \prec X_{\beta}$, \label{SpaceInclusion}
		\item $||x||_{X_{\beta}}\ \leq \ ||x||_{X_{\alpha}}$. \label{1stNormInequality}
	\end{enumerate}
\end{definition} 
\begin{definition} \label{Defovsop}
Let $\mathbf{X^{1}}$ be a scale and $\mathbf{X^{2}}\coloneqq\{X^{2}_{\inta}\}_{\inta \in\indexint}$ be a family of Banach spaces. Then
\[
G:\mathbf{X^{1}}(\cup)\to X_{\overline{\inta}}^{2}
\]
is called an Ovsjannikov map of order $q$ and $L$ from $\mathbf{X^{1}}$ to $\mathbf{X^{2}}$ if there exist $q,L\in\reals^{0}$ such that for all $\alpha < \beta \in \indexint$ and all $x,y\in X_{\alpha}$
\begin{enumerate}
	\item $G|_{X_{\alpha}^{1}}:X_{\alpha}^{1} \to X_{\beta}^{2}$, \label{0LCondition}
	\item $||G(x) - G(y)||_{X_{\beta}^{2}} \leq \frac{L}{(\beta -\alpha)^{q}}||x-y||_{X_{\alpha}^{1}}$.  \label{1stLCondition}
\end{enumerate}
\end{definition}
\begin{definition}
Suppose $\mathbf{X^{1}}$ is a scale and $\mathbf{X^{2}}\coloneqq\{X^{2}_{\inta}\}_{\inta \in\indexint}$ is a family of Banach spaces.
\begin{align}
\mathcal{O}(\mathbf{X^{1}},\mathbf{X^{2}},L,q)&\coloneqq\{\text{space of Ovsjannikov maps of order $q$ and $L$ from $\mathbf{X^{1}}$ to $\mathbf{X^{2}}$}\},\label{ovsspace1} \\[1em]
\mathcal{O}(\mathbf{X^{1}},L,q)&\coloneqq\{\text{space of Ovsjannikov maps of order $q$ and $L$ from $\mathbf{X^{1}}$ to $\mathbf{X^{1}}$}\}.\label{ovsspace12}
\end{align}
\end{definition}
\begin{definition}\label{DefSequenceSpaceScale}
For all $p\in\reals^{++}$ and all $\inta \in \indexint$ let
\begin{align}
l^{p}_{\inta} &\coloneqq \bigg{\{}z \in \reals^{\gamma} \ \bigg{|} \ \|z\|_{l^{p}_{\inta}} \coloneqq \bigg{(}\sum_{x\in\gamma}e^{- \inta|x|}|z_{x}|^{p}\bigg{)}^{\frac{1}{p}} < \infty\bigg{\}}, \\ \label{SequenceSpace}
\mathscr{L}^{p} &\coloneqq\{l^{p}_{\inta}\}_{\inta\in\indexint}, 
\end{align}
be, respectively, a normed linear space of weighted real sequences and a family of such spaces. 
\end{definition}

\nomenclature[$\reals^{\gamma}$]{$\reals^{\gamma}$}{}
\nomenclature[$l^{p}_{\inta}$]{$l^{p}_{\inta}$}{}
\nomenclature[$\mathscr{L}^{p}$]{$\mathscr{L}^{p}$}{}
\nomenclature[$\mathcal{O}(\mathbf{X^{1}},\mathbf{X^{2}},L,q)$]{$\mathcal{O}(\mathbf{X^{1}},\mathbf{X^{2}},L,q)$}{}
\nomenclature[$\mathcal{O}(\mathbf{X^{1}},L,q)$]{$\mathcal{O}(\mathbf{X^{1}},L,q)$}{}

\begin{theorem} \label{Ninlp}
Let $n\coloneqq\{n_{x}\}_{x\in\gamma}$. Then $n\in l_{\underline{\inta}}^{1}$.
\end{theorem}
\begin{proof}
Observe that assumption (\hyperref[EstimateOnN]{B}) implies that
\begin{align}
\sum_{x\in\gamma}e^{-\underline{\inta}|x|}|n_{x}| \ \leq \ \mathcal{N}\sum_{x\in\gamma}e^{-\underline{\inta}|x|}\log(1+|x|) \ \leq \ \mathcal{N}\sum_{x\in\gamma}e^{-\underline{\inta}|x|}|x|. 
\end{align}
Hence to conclude that proof we show that
\begin{align}
\sum_{x\in\gamma}e^{-\underline{\inta}|x|}|x|\ < \infty.  \label{NinplEqn1}
\end{align}
We now make a couple of preliminary observations and additional definitions. \\[1em] Fix a suitable $k \in\naturals$ such that $\sqrt{d}\frac{1}{2^{k}}<\rho$ and consider the following $k^{\text{th}}$ grid-partition or $\reals^{d}$
\begin{align}
\mathcal{R}^{k}&\coloneqq\{\mathcal{R}^{k}_{z}\}_{z\in\mathbb{Z}^{d}}, \\[1em]
\mathcal{R}^{k}_{z}&\coloneqq\bigg{\{}x\in\reals^{d}\ \bigg{|}\ \frac{z_{i}-1}{2^{k}}\leq x_{i}\leq \frac{z_{i}}{2^{k}}\bigg{\}}.
\end{align} 
We shall call members of the family $\mathcal{R}^{k}$, $k^{\text{th}}-$rectangles. Observe that for all $z\in\mathbb{Z}^{d}$
\begin{align}
Diam(\mathcal{R}^{k}_{z})\coloneqq\sup\{|x-y|_{S}\ |\ x,y\in\mathcal{R}^{k}_{z}\} = \frac{1}{2^{k}}.
\end{align}
Now introduce the following sets
\begin{equation}
\begin{alignedat}{2}
I_{n}&\coloneqq\bigg{\{}x\in\reals^{d}\ \bigg{|}\ |x|_{S}\ \leq\frac{1}{2}n\bigg{\}},\ &&\forall n\in\naturals^{0},\\[1em]
J_{n}&\coloneqq I_{n}-I_{n-1}, &&\forall n\in\naturals.
\end{alignedat}
\end{equation}
Consider now the real function $e^{-\underline{\inta}x}x:[0,\infty)\to\reals$. We see that $\frac{d}{dx}e^{-\underline{\inta}x}x = e^{-\underline{\inta}x}(1-\underline{\inta}x)$ and so it follows that $\frac{d}{dx}e^{-\underline{\inta}x}x<0$ if $x>\frac{1}{\underline{\inta}}$. Therfore letting $m\in\naturals$ be the smallest natural number such that $\max\{\frac{1}{\underline{\inta}},2\}\leq m$ we see that $e^{-\underline{\inta}x}x:(m,\infty)\to\reals$ is a decreasing function.
Finally observe that the following statements are true
\begin{enumerate}
	\item $I_{1}$ contains exactly $2^{k+1}$ of $k^{\text{th}}-$rectangles.
	\item $J_{n}$ contains fewer then $n2^{k+2}$ of $k^{\text{th}}-$rectangles.
	\item For all $n\in\naturals$, if $x\in\gamma\cap J_{n}$ then $|x|\ \geq n-1$.
	\item Suppose that $n\in\naturals$ and $z\in\mathbb{Z}^{d}$. Consider $x,y\in\gamma\cap\mathcal{R}^{k}_{z}\subset J_{n}$. It follows that 
	\begin{align}
	|x-y|&\leq\sqrt{d}|x-y|_{S}, \\
	&\leq\sqrt{d}Diam(\mathcal{R}^{k}_{z}), \\
	&\leq \sqrt{d}\frac{1}{2^{k}}, \\
	&\leq\rho. 
	\end{align}
	Hence we see that $y\in B_{x}$ and so from asumption \hyperref[EstimateOnN]{B} we see that
	\begin{align}
	\#\gamma\cap\mathcal{R}^{k}_{z} &\leq n_{x}, \\
	&\leq \mathcal{N}\log(1+|x|), \\
	&\leq \mathcal{N} |x|, \\
	&\leq \mathcal{N} n.
	\end{align}
	Therefore we conclude that for all $n\in\naturals$, $\#\gamma\cap J_{n}\leq \mathcal{N}n^{2}2^{k+2}$.
\end{enumerate}
\quad \\ \\ Returning now to the inequality (\ref{NinplEqn1}) we see that because $J_{m}$ is compact and $\gamma$ is locally finite we can let 
\begin{align}
B\coloneqq\sum_{x\in\gamma\cap J_{m}}e^{-\underline{\inta}|x|}|x|,
\end{align}
and observe that 
\begin{align}
\sum_{x\in\gamma}e^{-\underline{\inta}|x|}|x|\ &\leq B+ \sum_{n\in\naturals \atop n>m}\sum_{x\in\gamma\cap J_{n}}e^{-\underline{\inta}|x|}|x|, \\
&\leq B + \mathcal{N}2^{k+2}\sum_{n\in\naturals \atop n>m}e^{-\underline{\inta}(n-1)}(n-1)n^{2}.
\end{align}
Hence letting $\mathcal{K}\coloneqq\frac{m-1}{m}$ we see that
\begin{align}
\sum_{x\in\gamma}e^{-\underline{\inta}|x|}|x|\ \leq B + \mathcal{N}2^{k+2}\sum_{n\in\naturals \atop n>m}e^{-\mathcal{K}\underline{\inta}n}n^{3}.  \label{NinplEqn2}
\end{align}
Now, one can show via a simple calculation involving the integral test (for details see \cite{Nahin}) that the right hand side of the inequality (\ref{NinplEqn2}) above is finite hence the proof is complete. 
\end{proof}
\begin{mdframed}
	\begin{remark}
	From Theorem \ref{Ninlp} above it is clear that
	\begin{align}
	\sum_{x\in\gamma}e^{-\underline{\inta}|x|} < \infty. \label{NinplEqn3}
	\end{align}
	Moreover if $\|\cdot\|$ is any norm in $\reals^{d}$ equivalent to the Euclidean norm then, given a suitable modification of the asumption \hyperref[EstimateOnN]{B}, Theorem \ref{Ninlp} also implies that
	\begin{align}
	\sum_{x\in\gamma}e^{-\underline{\inta}\|x\|}\|x\|\ < \infty. \label{NinplEqn4}
	\end{align}
	\end{remark}
\end{mdframed}
\begin{theorem} \label{SequenceSpaceScale}
Suppose that $p\in\reals^{++}$. Then $\mathscr{L}^{p}$ is the scale.
\end{theorem}
\begin{proof}
It is clear from the Definition \ref{DefSequenceSpaceScale} that $\mathscr{L}^{p}$ is a family of normed linear spaces. Moreover 
conditions (\hyperref[SpaceInclusion]{1}) and (\hyperref[1stNormInequality]{2}) of the Definition \ref{Defscale} follow immediately from the simple fact that if $\alpha < \beta \in \indexint$ then $e^{-\alpha}>e^{-\beta}$. Hence to conclude the proof we fix $\inta\in\indexint$ and show that $l^{p}_{\inta}$ is a Banach space. \\ 

Let us begin by assuming that $\{z^{n}\}_{n\in\naturals}$ is a Cauchy sequence in $l^{p}_{\inta}$. Now fix an arbitrary $\epsilon>0$ and a suitable constant $N_{\epsilon}\in\naturals$ such that for all $n,m > N_{\epsilon}$ we have
\begin{align}
\bigg{(}\sum_{x\in\gamma}e^{- \inta|x|}|z^{n}_{x}-z^{m}_{x}|^{p}\bigg{)}^{\frac{1}{p}} < \epsilon.  \label{SequenceSpaceScaleeqn1}
\end{align}
Because $\epsilon$ is arbitrary we see from inequality (\ref{SequenceSpaceScaleeqn1}) above that for all $x\in\gamma$ sequence $\{z^{n}_{x}\}_{n\in\naturals}$ is Cauchy in $\reals$. Hence, it follows that we can define a new sequence $\mathbf{z}\coloneqq \{\mathbf{z}_{x}\}_{x\in\gamma}$ in $\reals^{\gamma}$ as follows 
\begin{align}
\mathbf{z}_{x} \coloneqq \lim_{n\to\infty}z^{n}_{x},\ \forall x\in\gamma.
\end{align}
Now we complete the proof by showing that $\mathbf{z}\in l^{p}_{\inta}$ and $\overbrace{\lim_{n\to\infty}z^{n}}^{\text{in} \ l^{p}_{\inta}} = \mathbf{z}$. To begin, we fix an arbitrary finite subset $A$ of $\gamma$. Now for all $n,m>N_{\epsilon}$ we see from inequality (\ref{SequenceSpaceScaleeqn1}) that
\begin{align}
\sum_{x\in A}e^{- \inta|x|}|z^{n}_{x}-z^{m}_{x}|^{p} < \epsilon^{p}.  \label{SequenceSpaceScaleeqn2}
\end{align}
Hence we see that for all $n>N_{\epsilon}$
\begin{equation}\label{SequenceSpaceScaleeqn3}
\begin{alignedat}{2}
\lim_{m\to\infty}\sum_{x\in A}e^{- \inta|x|}|z^{n}_{x}-z^{m}_{x}|^{p} &=  \mathbin{\textcolor{white}{\sum_{x\in A}e^{- \inta|x|}\lim_{m\to\infty}|z^{n}_{x}-z^{m}_{x}|^{p}}}&& \\
&= \sum_{x\in A}e^{- \inta|x|}|z^{n}_{x}-\lim_{m\to\infty}z^{m}_{x}|^{p} = \\
&&\!\!\!\!\!\!= \sum_{x \in A}e^{- \inta|x|}|z^{n}_{x}-\mathbf{z}_{x}|^{p}\ \leq\ \epsilon^{p}.
\end{alignedat}
\end{equation}
Since $A\subset\gamma$ is arbitrary we see from inequality (\ref{SequenceSpaceScaleeqn3}) above that for all $n>N_{\epsilon}$
\begin{align}
\sum_{x \in \gamma}e^{- \inta|x|}|z^{n}_{x}-\mathbf{z}_{x}|^{p} \leq \epsilon^{p}.
\end{align}
Because $\epsilon$ is also arbitrary we conclude that $\overbrace{\lim_{n\to\infty}z^{n}}^{\text{in} \ l^{p}_{\inta}} = \mathbf{z}$. Moreover we see that if $n>N_{\epsilon}$ then 
$z^{n}-\mathbf{z}\in l^{p}_{\inta}$. Since $l^{p}_{\inta}$ is a vector space we conclude that $\mathbf{z}\in l^{p}_{\inta}$ hence the proof is complete. 
\end{proof}
\subsection{Probability and Measure Spaces}\label{ProbabilitySpace}

We shall now proceed to describe the probability space and also a couple of important spaces of measurable maps and stochastic processes, that will become important in the main body of this text. \\[1em]
Hence let us assume the following. 

\nomenclature[$\mathbf{P}$]{$\mathbf{P}$}{}
\nomenclature[$\Omega$]{$\Omega$}{}
\nomenclature[$\mathbb{P}$]{$\mathbb{P}$}{}
\nomenclature[$\mathbb{F}$]{$\mathbb{F}$}{}
\nomenclature[$\mathcal{F}_{t}$]{$\mathcal{F}_{t}$}{}
\nomenclature[$\mathbf{P}_{t}$]{$\mathbf{P}_{t}$}{}
\nomenclature[$\mathbf{M}$]{$\mathbf{M}$}{}
\nomenclature[$\mathscr{B}(\mathcal{T})$]{$\mathscr{B}(\mathcal{T})$}{}
\nomenclature[$\mu$]{$\mu$}{}
\nomenclature[$\mathbf{M}\mathbf{P}$]{$\mathbf{M}\mathbf{P}$}{}
\nomenclature[$\mathscr{B}(l_{\inta}^{p})$]{$\mathscr{B}(l_{\inta}^{p})$}{}
\nomenclature[$\mathbf{M}^{\reals}$]{$\mathbf{M}^{\reals}$}{}
\nomenclature[$\mathbf{M}^{p}_{\inta}$]{$\mathbf{M}^{p}_{\inta}$}{}
\nomenclature[$\mathcal{M}(\cdot,\cdot)$]{$\mathcal{M}(\cdot,\cdot)$}{}

\begin{enumerate}
	\item Let us agree in the first palce that all probability and measure spaces in our discussion are completed. 
	\item Now we fix a filtered probability space
	\begin{align}
	\mathbf{P}\coloneqq (\Omega,\mathcal{F},\mathbb{P},\mathbb{F}),
	\end{align}
	on which of our subsequent work will be based. Moreover: 
	\begin{enumerate}
		\item For all $t\in\timeint$ we let $\mathbf{P}_{t}\coloneqq (\Omega,\mathcal{F}_{t},\mathbb{P})$.
		\item Filtration $\mathbb{F} \coloneqq \{\mathcal{F}_{t}\}_{t\in\mathcal{T}}$ is assumed to be right continuous. That is for all $t\in\timeint$,
		\begin{align}
		\mathcal{F}_{t} = \bigcap_{n\in\naturals}\mathcal{F}_{t+\frac{1}{n}}.
		\end{align}
	\end{enumerate}
	\item We fix a measure space $\mathbf{M}\coloneqq (\mathcal{T},\mathscr{B}(\mathcal{T}),\mu)$, where $\mu$ is a Lebesgue measure and $\mathscr{B}(\mathcal{T})$ is a Borel $\sigma-$algebra.
	\item We now agree to work on a fixed product measure space
	\begin{align}
	\mathbf{M}\mathbf{P}\coloneqq (\overline{\Omega}\coloneqq\mathcal{T}\times\Omega,\overline{\mathcal{F}}\coloneqq\mathscr{B}(\mathcal{T})\times\mathcal{F},\overline{\mathbb{P}}\coloneqq\mu\times\mathbb{P}).
	\end{align}
	\item Given two measurable spaces $\mathbf{A}$ and $\mathbf{B}$ we denote by $\mathcal{M}(\mathbf{A}, \mathbf{B})$ the space of all measurable maps from $\mathbf{A}$ to $\mathbf{B}$. In particular, the following measurable spaces will be frequently mentioned
	\begin{enumerate}
		\item $\mathbf{M}^{p}_{\inta}\coloneqq (l_{\inta}^{p},\mathscr{B}(l_{\inta}^{p}))$, 
		\item $\mathbf{M}^{\reals}\coloneqq (\reals,\mathscr{B}(\reals))$,
	\end{enumerate}
\end{enumerate}

Following definition fixes how we understand and denote stochastic processes in this text.
\begin{definition} \label{StochProcess}
	Let $Y$ be a normed linear space and $\mathbf{Y}\coloneqq(Y,\mathcal{B})$ be a measurable space. \textbf{Stochastic process} is an element of $\mathcal{M}(\mathbf{M}\mathbf{P}, \mathbf{Y})$. In particular for all $t\in\timeint$ and all $\omega\in\Omega$ 
	\[\mathcal{M}(\mathbf{P}, \mathbf{Y}) \ni \xi_{t}(\cdot):\Omega\to Y,\]
	\[\mathcal{M}(\mathbf{M}, \mathbf{Y}) \ni \xi_{\cdot}(\omega):\timeint\to Y.\]
	For brevity we shall denote by $\mathcal{S}(\mathbf{Y})$ the set of all stochastic processes from  $\mathbf{MP}$ to  $\mathbf{Y}$.
\end{definition}

\nomenclature[$W$]{$W$}{}
\nomenclature[$W_{x}$]{$W_{x}$}{}
\nomenclature[$\mathcal{S}(\cdot)$]{$\mathcal{S}(\cdot)$}{}
\nomenclature[$\mathcal{L}^{p}(\cdot,\cdot)$]{$\mathcal{L}^{p}(\cdot,\cdot)$}{}
\nomenclature[$\mathbf{L}$]{$\mathbf{L}$}{}

Following Bnach spaces will be frequently used. 
\begin{definition} \label{LPSpaces}
	Let $\mathscr{X}\coloneqq(X,\mathcal{A},\eta)$ be a measure space, $Y$ be a normed linear space, with norm denoted by $\|\cdot\|_{Y}$, and $\mathscr{Y}\coloneqq(Y,\mathcal{B})$ be a measurable space. For all $p\in\reals^{++}$ we define the following Banach spaces.
	\begin{align}
	\mathcal{L}^{p}(\mathscr{X},\mathscr{Y}) \coloneqq \left\{f:X\to Y \  \begin{tabular}{|l}
	\ $\|f\|_{\mathcal{L}^{p}(\mathscr{X},\mathscr{Y})}\coloneqq\left( \bigint_{\!\!\!\!X}\|f\|^{p}_{Y}d\eta \right)^{\frac{1}{p}} < \infty$. \\
	\ $f\in\mathcal{M}(\mathscr{X},\mathscr{Y})$.
	\end{tabular}\right\}
	\end{align}
\end{definition}
\begin{mdframed}
	\begin{remark}
		As it is often done in the literature, we will not consider explicitly the dependence of $\mathcal{L}^{p}(\cdot,\cdot)$ spaces on equivalence classes. We will work directly with the Definition \ref{LPSpaces} and when necessary acknowledge any issues arising from such dependence.
	\end{remark}
\end{mdframed}
\begin{theorem}
Suppose that $p\in\reals^{++}$. Then $\mathbf{L}\coloneqq\{\mathcal{L}^{p}(\mathbf{P},\mathbf{M}^{p}_{\inta})\}_{\inta\in\indexint}$ is a scale. 
\end{theorem}
\begin{proof}
The fact that $\mathbf{L}$ is a family of Banach spaces is a standart result from functional analysis, see \cite{RS} for example. Therefore it remains to verify that conditions (\hyperref[SpaceInclusion]{1}) and (\hyperref[1stNormInequality]{2}) of the Definition \ref{Defscale} hold. To this end let us start by fixing $\alpha < \beta \in \indexint$ and $f\in\mathcal{L}^{p}(\mathbf{P},\mathbf{M}^{p}_{\alpha})$. By Definition \ref{LPSpaces} it follows that $f \in\mathcal{M}(\mathbf{P},\mathbf{M}^{p}_{\alpha})$. Because $\mathscr{L}^{p}$ is the scale we conclude that $f \in\mathcal{M}(\mathbf{P},\mathbf{M}^{p}_{\beta})$ and $\|f\|_{l^{p}_{\beta}}^{p}\ \leq\ \|f\|_{l^{p}_{\alpha}}^{p}$. From Theorem \ref{LpComparison} we see that
\begin{align}
\int_{\Omega}\|f\|_{l^{p}_{\beta}}^{p}d\mathbb{P}\ \leq\ \int_{\Omega}\|f\|_{l^{p}_{\alpha}}^{p}d\mathbb{P}.
\end{align}
It follows that $f\in\mathcal{L}^{p}(\mathbf{P},\mathbf{M}^{p}_{\beta})$ and $\|f\|_{\mathcal{L}^{p}(\mathbf{P},\mathbf{M}^{p}_{\beta})}\ \leq\  \|f\|_{\mathcal{L}^{p}(\mathbf{P},\mathbf{M}^{p}_{\alpha})}$ hence the proof is complete. 
\end{proof}
\begin{definition}\label{ItoSpace}
	For all $p\in\reals^{++}$ we introduce the following spaces of stochastic processes. 
	\begin{align}
	L^{p}_{ad}\coloneqq \{ \xi\in \mathcal{L}^{p}(\mathbf{MP},\mathbf{M}^{\reals}) \ | \ \xi \ \text{is adapted to} \  \mathbb{F}\}.
	\end{align}
\end{definition}
\begin{mdframed}
	\begin{remark}
		Suppose that $p\geq2$ and $\xi\in L^{p}_{ad}$. Then $\xi\in L^{2}_{ad}$  by Theorem \ref{LPInclusion} and by Fubini Theorem \ref{FubiniTheorem} we also see that 
		\begin{align}
		\int_{0}^{T}\mathbb{E}\bigg{[}|\xi(t)|^{2}\bigg{]}dt < \infty. \label{SquareIntCondition}
		\end{align}
		This fact allows us to conclude that if $p\geq2$  then every process in $L^{p}_{ad}$ can be stochastically integrated with respect to the standard Wiener proces. See \cite{AR} and section \ref{WienerProcessinR} for more details. 
	\end{remark}
\end{mdframed}
\begin{enumerate}
	\setcounter{enumi}{4}
	\item Finally we fix a family of independent real valued Wiener processes $W\coloneqq \{W_{x}\}_{x\in\gamma}$ on $\mathbf{M}\mathbf{P}$ and require our filtration $\mathbb{F} \coloneqq \{\mathcal{F}_{t}\}_{t\in\mathcal{T}}$ to satisfy the following standart properties
	\begin{enumerate} \label{FiltrationProperties}
		\item For all $t\in\mathcal{T}$ and all $x\in\gamma$, $W_{x}(t)$ is $\mathcal{F}_{t}$ measurable, \label{FiltrationAss1}
		\item  For all $s \leq t\in\mathcal{T}$ and all $x\in\gamma$ $W_{x}(t) - W_{x}(s)$ is independent of $\mathcal{F}_{s}$.  \label{FiltrationAss2}
	\end{enumerate}
\end{enumerate}

\subsection{$\mathcal{Z}$ spaces}

\nomenclature[$L^{p}_{ad}$]{$L^{p}_{ad}$}{}
\nomenclature[$Z_{\inta}^{p}$]{$Z_{\inta}^{p}$}{}
\nomenclature[$\mathcal{Z}^{p}$]{$\mathcal{Z}^{p}$}{}

\begin{definition} \label{Zspaces}
For all $p\in\reals^{++}$ and all $\inta \in \indexint$ let
\begin{align}
Z_{\inta}^{p} &\coloneqq \left\{\xi\in\mathcal{S}(l_{\inta}^{p}) \  \begin{tabular}{|l}
\ $\|\xi\|_{Z^{p}_{\inta}} \coloneqq \bigg{(}\sup\bigg{\{}\mathbb{E}\bigg{[}\|\xi \|_{l^{p}_{\inta}}^{p}\bigg{]} \ \bigg{|} \ t\in\timeint\bigg{\}}\bigg{)}^{\frac{1}{p}} < \infty$. \\ 
\ $\xi_{x}$ is adapted to $\mathbb{F}$ for all $x\in\gamma$.
\end{tabular}\right\}, \\
\mathcal{Z}^{p} &\coloneqq\{Z^{p}_{\inta}\}_{\inta\in\indexint},
\end{align}
be, respectively, a normed linear space of $l_{\inta}^{p}$ valued processes and a family of such spaces. 
\end{definition}
\begin{mdframed}
\begin{remark}
Rigorously speaking $Z_{\inta}^{p}$ are a normed linear spaces only after we partitioned them into equivalence classes. However as with $\mathcal{L}^{p}(\cdot,\cdot)$ spaces
we will not consider explicitly such dependence at least untill the moment when we start addressing the question of uniqueness of certain processes belonging to these spaces. 
\end{remark}
\end{mdframed}
\begin{theorem}\label{aboutcomponents}
Let $p\in\reals^{++}$, $\inta \in \indexint$ and suppose that $\xi\in Z_{\inta}^{p}$. Then $\xi_{x}\in L^{p}_{ad}$ for all $x\in\gamma$.
\end{theorem}
\begin{proof}
From Definition \ref{Zspaces} we see that to complete the proof we need to show that for all $x\in\gamma$ we have $\xi_{x}\in\mathcal{L}^{p}(\mathbf{MP},\mathbf{M}^{\reals})$. Let us begin by establishing that $\xi_{x}\in\mathcal{M}(\mathbf{MP},\mathbf{M}^{\reals})$.\\[1em]
For each $x\in\gamma$ let us define maps
\[\mathscr{I}^{x}:l_{\inta}^{p}\to l_{\inta}^{p},\quad \mathscr{R}^{x}:\reals\to l_{\inta}^{p},\quad \xi|_{x}:\overline{\Omega}\to l_{\inta}^{p},\]
in the following way.
\begin{align} 
\mathscr{I}^{x}(\psi)_{y} &\coloneqq \begin{cases} 
\begin{tabular}{l|l}
$\psi_{y}$\ &\ $y\in\gamma\land y=x$, \\
$0$\ &\ $y\in\gamma\land y\not=x$.
\end{tabular}
\end{cases} \\[1em]
\mathscr{R}^{x}(z)_{y} &\coloneqq \begin{cases} 
\begin{tabular}{l|l}
$z$\ &\ $y\in\gamma\land y=x$, \\
$0$\ &\ $y\in\gamma\land y\not=x$.
\end{tabular}
\end{cases}\\[1em]
\xi|_{x} &\coloneqq \mathscr{I}^{x}(\xi).
\end{align}
Now observe that for each $x\in\gamma$ map $\mathscr{I}^{x}$ is continuous which implies that $\xi|_{x}\in\mathcal{M}(\mathbf{MP},\mathbf{M}^{p}_{\inta})$. Moreover observe that each $x\in\gamma$ map $\mathscr{R}^{x}$ is continuous and $\xi|_{x}=\mathscr{R}^{x}\circ\xi_{x}$. Consider now arbitary $A\coloneqq [a,b]\subset\reals$ and $x\in\gamma$. By continuity $B\coloneqq \mathscr{R}^{x}([a,b])$ is compact and so $B\in\mathscr{B}(l_{\inta}^{p})$. Since $\xi|_{x}\in\mathcal{M}(\mathbf{MP},\mathbf{M}^{p}_{\inta})$ it follows that $(\xi|_{x})^{-1}(B)\in\overline{\mathcal{F}}$. However
\[(\xi|_{x})^{-1}(B) = (\xi_{x})^{-1}\circ(\mathscr{R}^{x})^{-1}(B)=  (\xi_{x}^{-1})(A),\]
which establishes that $\xi_{x}\in\mathcal{M}(\mathbf{MP},\mathbf{M}^{\reals})$ for all $x\in\gamma$. \\[1em]
Finally since for all $x\in\gamma$ we have $|\xi_{x}|\leq e^{\frac{\inta}{p} |x|}\|\xi\|_{l^{p}_{\inta}}$ we may now conclude using Theorem \ref{LpComparison} that $\xi_{x}\in\mathcal{L}^{p}(\mathbf{MP},\mathbf{M}^{\reals})$ for all $x\in\gamma$ and the proof is complete. 
\end{proof}
\begin{mdframed}
	\begin{remark}
	In simple terms, Theorem \ref{aboutcomponents} above shows that, for all  $p\in\reals^{++}$ and $\inta \in \indexint$, component processes of each $\xi\in Z_{\inta}^{p}$ can be stochastically integrated with respect to the standard Wiener proces. 
	\end{remark}
\end{mdframed}

\begin{theorem} \label{ZspaceisBanach}
Let $p\in\reals^{++}$ and $\inta \in \indexint$. Then $Z^{p}_{\inta}$ is a Banach space.
\end{theorem}
\begin{proof}
According to Definition \ref{Zspaces} we need to show that $Z^{p}_{\inta}$ is complete. Therfore let us start by assuming that $\mathscr{X} \coloneqq \{\xi^{n}\}_{n\in\naturals}$ is a Cauchy sequence in $Z^{p}_{\inta}$. Moreover, let us also define $\mathscr{X}_{t} \coloneqq \{\xi^{n}_{t}\}_{n\in\naturals}$ and observe now the following
\begin{enumerate}
	\item From Definition \ref{Zspaces} we see that $\mathscr{X}$ is a Cauchy sequence in $\mathcal{L}^{p}(\mathbf{MP},\mathbf{M}^{p}_{\inta})$. Hence let us define $\overline{\xi}:\overline{\Omega}\to l^{p}_{\inta}$ in the following way
	\begin{align}
	\overline{\xi}\coloneqq \overbrace{\bigg{[}\ \lim_{n\to\infty}\xi^{n}\ \bigg{]}}^{\text{in} \ \mathcal{L}^{p}(\mathbf{MP},\mathbf{M}^{p}_{\inta})}\!\! .\label{ZspaceisBanachE1}
	\end{align}
	\item From Definition \ref{Zspaces} we see that for all $t\in\timeint$ sequence $\mathscr{X}_{t}$ is Cauchy in $\mathcal{L}^{p}(\mathbf{P},\mathbf{M}^{p}_{\inta})$. Moreover we see that  
	\begin{align}
	\lim_{n,m\to\infty} \|\xi^{n}_{t} - \xi^{m}_{t}\|_{\mathcal{L}^{p}(\mathbf{P},\mathbf{M}^{p}_{\inta})}\ = 0\ \text{uniformly on } \timeint. \label{ZspaceisBanachE2}
	\end{align}
	We define $\overline{\overline{\xi}}:\overline{\Omega}\to l^{p}_{\inta}$ in the following way
	\begin{align}
	\overline{\overline{\xi}}(t,\omega)\coloneqq \overbrace{\bigg{[}\lim_{n\to\infty}\xi^{n}_{t}\bigg{]}}^{\text{in} \ \mathcal{L}^{p}(\mathbf{P},\mathbf{M}^{p}_{\inta})}(\omega), \label{ZspaceisBanachE3}
	\end{align}
	and conclude from a convergence result (\ref{ZspaceisBanachE2}) above that
	\begin{align}
	\lim_{n\to\infty} \|\xi^{n}_{t} - \overline{\overline{\xi}}_{t}\|_{\mathcal{L}^{p}(\mathbf{P},\mathbf{M}^{p}_{\inta})}\ = 0\ \text{uniformly on } \timeint. \label{ZspaceisBanachE4}
	\end{align}
\end{enumerate}
Now we make an important observation that underpins the rest of this proof. From the fact that 
\begin{align}
\lim_{n\to\infty} \|\xi^{n}_{t} - \overline{\xi}_{t}\|_{\mathcal{L}^{p}(\mathbf{MP},\mathbf{M}^{p}_{\inta})}\ = 0, \label{ZspaceisBanachE5}
\end{align}
We can coclude using Fubini Theorem \ref{FubiniTheorem} and Theorem \ref{ASsubsequence} that there exist a subsequence $\sigma$ such that 
\begin{align}
\mu-a.s.,\ \mathbb{E}\bigg{[}\|\xi^{\sigma(n)}_{t}-\overline{\xi}_{t}\|_{l^{p}_{\inta}}^{p}\bigg{]} \to 0. \label{ZspaceisBanachE6}
\end{align}
Using Egoroff Theorem \ref{EgoroffT} and the fact that $\mu$ is a regular measure we find a sequence $\{A_{k}\}_{k\in\naturals}$ of subsets of $\timeint$ such that for all $k\in\naturals$
\begin{enumerate}
	\item $A_{k}\subset\timeint$ is compact and $\mu(A_{k})\leq \frac{1}{2^{k}}$,
	\item Moreover 
	\begin{align}
	\mathbb{E}\bigg{[}\|\xi^{\sigma(n)}_{t}-\overline{\xi}_{t}\|_{l^{p}_{\inta}}^{p}\bigg{]} \to 0\ \text{uniformly on } \timeint-A_{k}. \label{ZspaceisBanachE7}
	\end{align}
\end{enumerate}
Let us now define a null set
\begin{align}
\widetilde{\timeint}\coloneqq \bigg{\{}t\in\timeint\ \big{|}\ t\in\bigcap_{k\in\naturals}A_{k}\bigg{\}}, \label{ZspaceisBanachE8}
\end{align}
a sequence $\{B_{k}\}_{k\in\naturals}$ of subsets of $\timeint$ where for all $k\in\naturals$
\begin{align}
B_{k} \coloneqq A_{k} - \widetilde{\timeint}, \label{ZspaceisBanachE9}
\end{align}
and a map $\xi:\overline{\Omega}\to l^{p}_{\inta}$ in the following way
\begin{align} \label{ZspaceisBanachE10}
\xi(t,\omega)&\coloneqq \begin{cases} 
\begin{tabular}{l|l}
$\overline{\xi}(t,\omega)$\ &\ $\omega\in\Omega\land t\in(\timeint-\widetilde{\timeint})$, \\
$\overline{\overline{\xi}}(t,\omega)$\ &\ $\omega\in\Omega\land t\in\widetilde{\timeint}$.
\end{tabular}
\end{cases}
\end{align}
We conclude immediately tha $\xi\in\mathcal{S}(l_{\inta}^{p})$ because $\mathbf{MP}$ is complete, $\widetilde{\timeint}$ a null set and so $\xi=\overline{\xi}$ almost everywhere on $\overline{\Omega}$. Now, to conclude this proof it remains to show the following:
\begin{enumerate}
	\item[\textbf{I}.] $\lim_{n\to\infty} \|\xi^{n} - \xi\|_{Z^{p}_{\inta}} \to 0$ as $n\to\infty$,
	\item[\textbf{II}.] $\xi_{x}$ is adapted to $\mathbb{F}$ for all $x\in\gamma$.
\end{enumerate}
\underline{Proof of \textbf{I}}. \\
Fix $k\in\naturals$ and using composite definition of $\xi$ together with estimates (\ref{ZspaceisBanachE4}) and (\ref{ZspaceisBanachE7}) observe that
\begin{align}
\mathbb{E}\bigg{[}\|\xi^{\sigma(n)}_{t}-\xi_{t}\|_{l^{p}_{\inta}}^{p}\bigg{]} \to 0\ \text{uniformly on } (\timeint-B_{k})\cup\widetilde{\timeint}. \label{ZspaceisBanachE11}
\end{align}
Therefore we fix an arbitrary $\epsilon > 0$ and a suitable $N_{\epsilon}\in\naturals$ such that for all $n>N_{\epsilon}$ and all $k\in\naturals$ we have
\begin{align}
\mathbb{E}\bigg{[}\|\xi^{\sigma(n)}_{t}-\xi_{t}\|_{l^{p}_{\inta}}^{p}\bigg{]} <\epsilon\ \text{uniformly on } (\timeint-B_{k})\cup\widetilde{\timeint}. \label{ZspaceisBanachE12}
\end{align}
Hence for all $n>N_{\epsilon}$ and all $k\in\naturals$ we see that
\begin{align}
\sup\bigg{\{}\mathbb{E}\bigg{[}\|\xi^{\sigma(n)}_{t}-\xi_{t}\|_{l^{p}_{\inta}}^{p}\bigg{]} \ \bigg{|} \ t\in(\timeint-B_{k})\cup\widetilde{\timeint}\bigg{\}} \leq \epsilon. \label{ZspaceisBanachE13}
\end{align} 
Moreover, inequality (\ref{ZspaceisBanachE13}) above shows that for all $n>N_{\epsilon}$ 
\begin{align}
\sup\bigg{\{}\mathbb{E}\bigg{[}\|\xi^{\sigma(n)}_{t}-\xi_{t}\|_{l^{p}_{\inta}}^{p}\bigg{]} \ \bigg{|} \ t\in\timeint\bigg{\}} \leq \epsilon, \label{ZspaceisBanachE14}
\end{align} 
for otherwise there exists $\bar{n}>N_{\epsilon}$ and $\bar{t}\in\timeint$ such that 
\begin{align}
\mathbb{E}\bigg{[}\|\xi^{\sigma(\bar{n})}_{\bar{t}}-\xi_{\bar{t}}\|_{l^{p}_{\inta}}^{p}\bigg{]} > \epsilon. \label{ZspaceisBanachE15}
\end{align}
Moreover $\bar{t}\not\in\widetilde{\timeint}$ for otherwise we will contradict inequality (\ref{ZspaceisBanachE13}). Therefore, by definition of $\widetilde{\timeint}$ it follows that there exists $\bar{k}\in\naturals$ such that $t\not\in A_{\bar{k}}$ hence $t\not\in B_{\bar{k}}$ and so $t\in(\timeint-B_{\bar{k}})$. Therefore from inequality (\ref{ZspaceisBanachE13}) we get a contradiction
\begin{align}
\mathbb{E}\bigg{[}\|\xi^{\sigma(\bar{n})}_{\bar{t}}-\xi_{\bar{t}}\|_{l^{p}_{\inta}}^{p}\bigg{]} \leq \epsilon, \label{ZspaceisBanachE16}
\end{align}
hence inequality (\ref{ZspaceisBanachE14}) holds and we conclude that $\lim_{n\to\infty} \|\xi^{\sigma(n)} - \xi\|_{Z^{p}_{\inta}} \to 0$ as $n\to\infty$. Finally because $\mathscr{X}$ is a Cauchy sequence we conclude that $\lim_{n\to\infty} \|\xi^{n} - \xi\|_{Z^{p}_{\inta}} \to 0$ as $n\to\infty$. \\[1em]
\underline{Proof of \textbf{II}}. \\
We start by fixing $x\in\gamma$ and $t\in\timeint$. Now from the previous part (i.e. \underline{Proof of \textbf{I}}) we can deduce that
\begin{align}
\lim_{n\to\infty} \|\xi^{n}_{x,t} - \xi_{x,t}\|_{\mathcal{L}^{p}(\mathbf{P},\mathbf{M}^{\reals})}\ = 0. \label{ZspaceisBanachE17}
\end{align}
Therefore using Theorem \ref{ASsubsequence} we find a subsequence $\sigma$ such that $\mathbb{P}-a.s.$ we have $\xi^{\sigma(n)}_{x,t} \to \xi_{x,t}$. Since for all $n\in\naturals$ processes $\xi^{\sigma(n)}_{x,t}$ is $\mathcal{F}_{t}$ measurable we conclude by Theorem \ref{measurablelimit} that $\xi_{x,t}$ is $\mathcal{F}_{t}$ measurable and the proof is complete. 
\end{proof}
\begin{theorem} \label{Zscaleproof}
Suppose that $p \in\reals^{++}$. Then $\mathcal{Z}^{p}$ is the scale.
\end{theorem}
\begin{proof}
From Theorem \ref{ZspaceisBanach} we alredy know that $\mathcal{Z}^{p}$ is a family of Banach spaces so to conclude the proof it only remains to show that conditions (\ref{SpaceInclusion}) and (\ref{1stNormInequality}) of the Definition \ref{Defscale} are satisfied.\\[1em] 
Let us begin by fixing $\alpha < \beta \in \indexint$ and observing that condition (\ref{SpaceInclusion}) of the Definition \ref{Defscale} can be verified by showing that  $Z^{p}_{\alpha} \subset Z^{p}_{\beta}$. To see that this is true we now fix $\xi\in Z^{p}_{\alpha}$. Hence, $\xi\in\mathcal{S}(l^{p}_{\alpha})$ and components of $\xi$ are adapted to $\mathbb{F}$ by Definition \ref{Zspaces}. Since $\mathscr{L}^{p}$ is a scale we also see that $\xi\in\mathcal{S}(l^{p}_{\beta})$. Moreover for all $t\in\timeint$ we observe that $\|\xi_{t}\|_{l_{\alpha}^{p}}\in \mathcal{L}^{p}(\mathbf{P},\mathbf{M}^{\reals})$, $\|\xi_{t}\|_{l_{\beta}^{p}}\in \mathcal{M}(\mathbf{P},\mathbf{M}^{\reals})$ and  $\|\xi_{t}\|_{l_{\beta}^{p}} \leq \|\xi_{t}\|_{l_{\alpha}^{p}}$. Now from Theorem \ref{LpComparison} we see that 
\begin{align}
\mathbb{E}\bigg{[} \|\xi_{t}\|_{l_{\beta}^{p}}^{p}\bigg{]} \leq \mathbb{E}\bigg{[} \|\xi_{t}\|_{l_{\alpha}^{p}}^{p}\bigg{]},
\end{align}
which establishes that $\|\xi\|_{Z_{\beta}^{p}}\ \leq \|\xi\|_{Z_{\alpha}^{p}}$ hence proving that both conditions (\ref{SpaceInclusion}) and (\ref{1stNormInequality}) of the Definition \ref{Defscale} are satisfied.
\end{proof}
\subsection{Stochastic System}\label{systemdescription}

\nomenclature[$\zeta$]{$\zeta$}{}
\nomenclature[$\Phi_{x}$]{$\Phi_{x}$}{}
\nomenclature[$\Psi_{x}$]{$\Psi_{x}$}{}
\nomenclature[$V$]{$V$}{}
\nomenclature[$c$]{$c$}{}
\nomenclature[$b$]{$b$}{}
\nomenclature[$R$]{$R$}{}
\nomenclature[$M_{1}$]{$M_{1}$}{}
\nomenclature[$M_{2}$]{$M_{2}$}{}

Throughout this section let us assume that $\reals\ni p\geq 2$. We now wish to introduce and study the following stochastic system, which we will denote by $\mathscr{O}^{p}$. 
\begin{align}\label{MainSystem}
\xi_{x,t} &= \zeta_{x} + \int_{0}^{t}\Phi_{x}(\xi_{x,s},\Xi_{s})ds + \int_{0}^{t}\Psi_{x}(\xi_{x,s},\Xi_{s})dW_{x}(s), \quad x\in\gamma, \quad t\in\timeint, \tag{$\mathscr{O}^{p}$}
\end{align}
where:
\begin{enumerate}
	\item we assume that $\zeta\in l^{p}_{\underline{\inta}}$.
	\item we let $V$ in $C(\reals)$ and assume that for all $x\in\gamma$ maps $\Phi_{x}:\reals\times\reals^{\gamma}\to \reals$ are measurable defined in the following way
	\begin{align} 
	\Phi_{x}(q,\{z_{y}\}_{y\in\gamma}) \coloneqq V(q) + \sum_{y\in B_{x}}a(x-y)z_{y}, \label{Eq1ForExistence}
	\end{align}
	for all $q\in\reals$ and all $\{z_{y}\}_{y\in\gamma} \in \reals^{\gamma}$, where function $a$ was defined previously in (\hyperref[boundfora]{A}).
\end{enumerate}
For all $x\in\gamma$ the following conditions are placed on maps $\Phi_{x}$.
\begin{enumerate}
	\item[(\textbf{C})] There exists $c\in\reals^{0}$ and $\reals^{++}\ni R\leq p$ such that for all $q\in\reals$ and all $x\in\gamma$
	\begin{align}
	|\Phi_{x}(q,0)|\leq c(1+|q|^{R})_{.} \label{conditionA}
	\end{align}
	\item[(\textbf{D})] There exists $b\in\reals$ such that for all $q_{1},q_{2}\in\reals$  and all $x\in\gamma$
	\begin{align}
	(q_{1} - q_{2})(\Phi_{x}(q_{1},0) - \Phi_{x}(q_{2},0))\leq b(q_{1} - q_{2})^{2}_{.} \label{conditionB}
	\end{align} 
	\end{enumerate}
\begin{enumerate}
\setcounter{enumi}{2}
	\item For all $x\in\gamma$ we assume that maps $\Psi_{x}:\reals\times\reals^{\gamma}\to \reals$ are measurable.
\end{enumerate}
Moreover for all $x\in\gamma$ the following conditions are also placed on maps $\Psi_{x}$.
\begin{enumerate}
	\item[(\textbf{E})] There exists $M_{1},M_{2}\in\reals$ such that for all $q_{1},q_{2}\in\reals$, $Z_{1},Z_{2}\in\reals^{\gamma}$ and all $x\in\gamma$
	\begin{align}
	|\Psi_{x}(q_{1},Z_{1}) - \Psi_{x}(q_{2},Z_{2})| &\leq M_{1}|q_{1} - q_{2}| + M_{2}n_{x}\sum_{y\in B_{x}}|z_{1,y} - z_{2,y}|, \label{conditionC}\\[1em]
	|\Psi_{x}(0,0)|&\leq c.
	\end{align} 
\end{enumerate}
\begin{definition}\label{StrongSol}
Suppose that $\reals\ni p\geq 2$. A stochastic process $\Xi$ is called a strong solution of the system (\ref{MainSystem}) if $\Xi\in\mathcal{Z}^{p}(\cap)$ and for all $x\in\gamma$ and all $t\in\timeint$ we have
\begin{align}
\xi_{x,t} &= \zeta_{x} + \int_{0}^{t}\Phi_{x}(\xi_{x,s},\Xi_{s})ds + \int_{0}^{t}\Psi_{x}(\xi_{x,s},\Xi_{s})dW_{x}(s),\ \mathbb{P}-a.s.
\end{align}
\end{definition}
The main goal of this document is to show that for all $\reals\ni p\geq 2$ the stochastic system (\ref{MainSystem}) admits a unique strong solution. We now conclude this subsection with the following Lemma. 

\nomenclature[$\Tilde{a}$]{$\Tilde{a}$}{}

\begin{lemma1} \label{DriftLemma}
Suppose that $q_{1},q_{2}\in\reals$ and let $Z_{1},Z_{2}\in\reals^{\gamma}$. Then for all $x\in\gamma$ we have 
\begin{align}
|\Phi_{x}(q_{1},Z_{1})| &\leq c(1+|q_{1}|^{R}) + \Tilde{a}_{x}\bigg{(}\sum_{y\in B_{x}}z_{1,y}^{2}\bigg{)}^{\frac{1}{2}}_{,} \label{DriftLemmaIneq1}\\[1em]
(q_{1} - q_{2})(\Phi_{x}(q_{1},Z_{1}) - \Phi_{x}(q_{2},Z_{2})) &\leq (b+\frac{1}{2})(q_{1} - q_{2})^{2} + \frac{1}{2}\Tilde{a}_{x}^{2}\sum_{y\in B_{x}}(z_{1,y} - z_{2,y})^{2}_{,} \label{DriftLemmaIneq2}
\end{align}
where $\Tilde{a}_{x} = \bigg{(}\sum_{y\in B_{x}}a^{2}(x-y)\bigg{)}^{\frac{1}{2}}_{.}$
\end{lemma1}

\begin{proof}
\noindent First we prove inequality (\ref{DriftLemmaIneq1}). We begin by considering the following chain of calculations
\begin{align}
|\Phi_{x}(q_{1},Z_{1})| &= |\frac{1}{2}V(q_{1}) - \sum_{y\in B_{x}}a(x-y)z_{1,y}|, \\[1em]
& = |\Phi_{x}(q_{1},0) - \sum_{y\in B_{x}}a(x-y)z_{1,y}|, \label{DriftLemmaIneq3} \\[1em]
& \leq |\Phi_{x}(q_{1},0)| + |\sum_{y\in B_{x}}a(x-y)z_{1,y}|.
\end{align}
\noindent Therefore using assumption (\hyperref[conditionA]{C}) we see that 
\begin{align}
|\Phi_{x}(q_{1},Z_{1})| &\leq c(1+|q_{1}|^{R}) + \sum_{y\in B_{x}}|a(x-y)z_{1,y}|, \\[1em]
&\leq c(1+|q_{1}|^{R}) + \bigg{(}\sum_{y\in B_{x}}a^{2}(x-y)\bigg{)}^{\frac{1}{2}}\bigg{(}\sum_{y\in B_{x}}z_{1,y}^{2}\bigg{)}^{\frac{1}{2}}_{.}
\end{align}
Hence using the definition of $\Tilde{a}_{x}$ above we see that
\begin{align}
|\Phi_{x}(q_{1},Z_{1})| \leq c(1+|q_{1}|^{R}) + \Tilde{a}_{x}\bigg{(}\sum_{y\in B_{x}}z_{1,y}^{2}\bigg{)}^{\frac{1}{2}}_{,}
\end{align}
which establishes that inequality (\ref{DriftLemmaIneq1}) is true. Now we show that inequality (\ref{DriftLemmaIneq2}) above is also true. \\[1em]
We start by observing from equations (\ref{DriftLemmaIneq3}) and (\ref{Eq1ForExistence}) that
\begin{equation}
\begin{alignedat}{1}
(q_{1} - q_{2})(\Phi_{x}(q_{1},Z_{1}) - &\Phi_{x}(q_{2},Z_{2})) = \\[1em]
& = (q_{1} - q_{2})(\Phi_{x}(q_{1},0) - \Phi_{x}(q_{2},0)) + \\[1em]
&\quad\quad\quad\quad\quad\quad\quad\quad\quad\quad\ \ +(q_{1} - q_{2})\sum_{y\in B_{x}}a(x-y)(z_{1,y} - z_{2,y}).
\end{alignedat}
\end{equation}
Hence using assumption (\hyperref[conditionB]{D}) we see that 
\begin{align}
(q_{1} - q_{2})(\Phi_{x}(q_{1},Z_{1}) - &\Phi_{x}(q_{2},Z_{2})) \leq \\[1em]  
&\leq b(q_{1} - q_{2})^{2} + \frac{1}{2}(q_{1} - q_{2})^{2} + \frac{1}{2}\bigg{(}\sum_{y\in B_{x}}a(x-y)(z_{1,y} - z_{2,y})\bigg{)}^{2}_{,} \\[1em]
&\leq (b+\frac{1}{2})(q_{1} - q_{2})^{2} + \frac{1}{2}\sum_{y\in B_{x}}a^{2}(x-y)\sum_{y\in B_{x}}(z_{1,y} - z_{2,y})^{2}_{.}
\end{align}
Finally using, once again, the definition of $\Tilde{a}_{x}$ above we see that
\begin{align}
(q_{1} - q_{2})&(\Phi_{x}(q_{1},Z_{1}) - \Phi_{x}(q_{2},Z_{2})) \leq (b+\frac{1}{2})(q_{1} - q_{2})^{2} + \frac{1}{2}\Tilde{a}_{x}^{2}\sum_{y\in B_{x}}(z_{1,y} - z_{2,y})^{2}_{,}
\end{align}
and the proof is complete.
\end{proof}
\newpage 
%ALL THE STARTUP STUFF IS ABOVE.

\section{Auxiliary Results} \label{AR}
In this section we prove two results that will be used later on to show that the stochastic system (\ref{MainSystem}) admits a unique strong solution. \\[1em]
Throughout this section let us assume that $\reals\ni p\geq 2$.

\begin{theorem} \label{ItoComposition}
Suppose that $\inta\in\indexint$ and let $\Xi\coloneqq\{\xi_{x}\}_{x\in\gamma}$ be an element in $Z^{p}_{\inta}$. Then for all $x\in\gamma$  we have $\Phi_{x}(\xi_{x},\Xi)\in L^{1}_{ad}$ and $\Psi_{x}(\xi_{x},\Xi)\in L^{2}_{ad}$.
\end{theorem}

\begin{proof}
We combine Theorems \ref{aboutcomponents} and \ref{LPInclusion} to conclude that for all $x\in\gamma$ we have 
\begin{align}
\xi_{x}\in L^{p}_{ad}\subset L^{2}_{ad}\subset L^{1}_{ad}.
\end{align}
Since composition of mesurable maps is measurable we conclude that $x\in\gamma$ we have 
\begin{align}
\Phi_{x}(\xi_{x},\Xi), \Psi_{x}(\xi_{x},\Xi)\in \mathcal{M}(\mathbf{MP},\mathbf{M}^{\reals}),
\end{align}
and adapted to $\mathbb{F}$. Now according to the definition (\ref{Eq1ForExistence}) and the assumption (\hyperref[conditionA]{C}) we have for all $x\in\gamma$ the following inequality
\begin{align}
|\Phi_{x}(\xi_{x},\Xi)|\ \leq |c|(1+|\xi_{x}|^{R}) + \sum_{y\in B_{x}}a(x-y)|\xi_{y}|. \label{ItoCompositionEqn2}
\end{align}
Moreover, because $R\leq p$ we can use Theorem \ref{LpComparison} to conclude that $\Phi_{x}(\xi_{x},\Xi)\in L^{1}_{ad}$. Finally we combine Theorem \ref{powerinequalitytheorem} with the assumption (\hyperref[conditionC]{E}) to conclude that for all $x\in\gamma$ we have
\begin{align}
|\Psi_{x}(\xi_{x},\Xi)|^{2} \leq 4|\Psi_{x}(0,0)|^{2}+ 4M_{1}^{2}|\xi_{x}|^{2} + 4M_{2}^{2}n_{x}^{3}\sum_{y\in B_{x}}|\xi_{x}|^{2}. \label{ItoCompositionEqn3}
\end{align} 
Once again applying Theorem \ref{LpComparison} to the inequality (\ref{ItoCompositionEqn3}) above we conclude that $\Psi_{x}(\xi_{x},\Xi)\in L^{2}_{ad}$ hence the proof is complete.
\end{proof} 

\begin{theorem} \label{OvsMapTheorem}
Let $L\coloneqq 4e^{\underline{\inta}\rho}C\mathcal{N}^{q+1}(1+\rho)^{\frac{1}{2}}$ and let $Q\coloneqq\{Q_{x,y}\}_{x,y\in \gamma}$ be an infinite real matrix such that for all $x,y\in\gamma$ we have
\begin{align}
x\not\in B_{y} \iff Q_{x,y} = 0 \iff y\not\in B_{x}. \label{MatrixOvsCond2}
\end{align}
Moreover assume that for all $x,y\in\gamma$ there exist $C\in\reals^{0}$ and $q\in\reals^{++}$ such that 
\begin{align}
|Q_{x,y}|\leq C n_{x}^{q}. \label{MatrixOvsCond1}
\end{align}
Then $Q\in\mathcal{O}(\mathscr{L}^{1},L,\frac{1}{2})$. That is, $Q$ is the Ovsjannikov map of order $L$ and $\frac{1}{2}$ on $\mathscr{L}^{1}$.
\end{theorem}
\begin{proof}
Consider arbitrary $\alpha<\beta\in\indexint$ and fix $z\in l_{\alpha}^{1}$. We will complete this proof by showing that
\begin{align}
\|Qz\|_{\beta} \leq \frac{L}{(\beta-\alpha)^{\frac{1}{2}}}\|z\|_{\alpha}, \label{OvsMapTheoreme1}
\end{align}
which establishes that $Q$ is a linear operator from $l_{\alpha}^{1}$ to $l_{\beta}^{1}$ hence verifing conditions (\ref{0LCondition}) and (\ref{1stLCondition}) of the Definition \ref{Defovsop}. \\[1em]
Consider now the following equation
\begin{align}
\|Qz\|_{\beta} &= \sum_{x\in\gamma}e^{-\beta |x|}\bigg{|}\sum_{y\in\gamma}Q_{x,y}z_{y}\bigg{|}. \label{OvsMapTheoremEqn1}
\end{align}
Moreover, for all $x\in\gamma$ we will make use of the following facts
\begin{equation}
\begin{alignedat}{3}
&\mathbf{I}.\  &&x\not\in B_{y}\lor y\not\in B_{x} &&\implies Q_{x,y} = 0. \\
&\mathbf{II}.\ &&y\in B_{x} &&\implies -|x| \leq -|y| + \rho. \\
&\mathbf{III}.\ &&x\in B_{y} &&\implies \sqrt{|x|} \leq \sqrt{|y|} + \sqrt{\rho}. 
\end{alignedat}
\end{equation}
Now, using equation (\ref{OvsMapTheoremEqn1}) together with the facts $\mathbf{I}$ and $\mathbf{II}$ we see that
\begin{align}
\|Qz\|_{\beta} &\leq \sum_{x\in\gamma}\sum_{y\in\gamma}|Q_{x,y}|e^{-\beta |x|}|z_{y}|, \label{OvsMapTheoremEqn2} \\[1em] 
&\leq e^{\beta\rho}\sum_{x\in\gamma}\sum_{y\in B_{x}}|Q_{x,y}|e^{-\beta |y|}|z_{y}|, \\[1em]  
&\leq e^{\beta\rho}\sum_{x\in\gamma}\sum_{y\in B_{x}}|Q_{x,y}|e^{-(\beta-\alpha) |y|}e^{-\alpha |y|}|z_{y}|. \label{OvsMapTheoreme2}
\end{align}
Hence from inequality (\ref{OvsMapTheoreme2}) we see that 
\begin{align}
\|Qz\|_{\beta} &\leq e^{\beta\rho}\sum_{x\in\gamma}\sum_{y\in\gamma}|Q_{x,y}|e^{-(\beta-\alpha) |y|}e^{-\alpha |y|}|z_{y}|, \\[1em]  
&= e^{\beta\rho}\sum_{y\in\gamma}\sum_{x\in\gamma}|Q_{x,y}|e^{-(\beta-\alpha) |y|}e^{-\alpha |y|}|z_{y}|, \\[1em] 
&\leq e^{\overline{\inta}\rho}K\|z\|_{\alpha}, \label{OvsMapTheoremEqn3}
\end{align}
where
\begin{align}
K \coloneqq \sup\bigg{\{}\sum_{x\in \gamma}|Q_{x,y}|e^{-(\beta-\alpha)|y|} \ \bigg{|} \ y\in\gamma \bigg{\}}. \label{OvsMapTheoremEqn4}
\end{align}
We now estimate the value of supremum in the definition (\ref{OvsMapTheoremEqn4}) above. Hence using condition (\ref{MatrixOvsCond1}) together with the fact $\mathbf{I}$ we see that for all $y\in\gamma$
\begin{align}
\sum_{x\in\gamma}|Q_{x,y}|e^{-(\beta-\alpha)|y|} &= \sum_{x\in B_{y}}|Q_{x,y}|e^{-(\beta-\alpha) |y|}, \\[1em]
&\leq  C\sum_{x\in B_{y}}n_{x}^{q}e^{-(\beta-\alpha)|y|}.
\end{align}
Using now assumption (\hyperref[EstimateOnN]{B}) together with the fact $\mathbf{III}$ we see that for all $y\in\gamma$
\begin{align}
\sum_{x\in\gamma}|Q_{x,y}|e^{-(\beta-\alpha)|y|} & \leq  C\sum_{x\in B_{y}}\mathcal{N}^{q}|x|^{\frac{1}{2}}e^{-(\beta-\alpha)|y|}, \\[1em] 
&\leq C\mathcal{N}^{q}\sum_{x\in B_{y}}(|y|^{\frac{1}{2}}+\rho^{\frac{1}{2}})e^{-(\beta-\alpha)|y|}, \\[1em] 
&\leq C\mathcal{N}^{q}n_{y}(|y|^{\frac{1}{2}}+\rho^{\frac{1}{2}})e^{-(\beta-\alpha)|y|}, \\[1em] 
&\leq C\mathcal{N}^{q+1}(|y|^{\frac{1}{4}}+|y|^{\frac{1}{2}}\rho^{\frac{1}{2}})e^{-(\beta-\alpha)|y|}, \\[1em] 
&\leq B(|y|^{\frac{1}{4}}+|y|^{\frac{1}{2}})e^{-(\beta-\alpha)|y|},
\end{align}
where $B\coloneqq C\mathcal{N}^{q+1}(1+\rho)^{\frac{1}{2}}$. \\[1em]
Now returning to equation (\ref{OvsMapTheoremEqn4}) we see that
\begin{align}
K &\leq B\sup\bigg{\{}(|y|^{\frac{1}{4}}+|y|^{\frac{1}{2}})e^{-(\beta-\alpha)|y|} \ \bigg{|} \ y\in\gamma \bigg{\}}, \\[1em] 
&\leq B\sup\bigg{\{}(h^{\frac{1}{4}}+h^{\frac{1}{2}})e^{-(\beta-\alpha)h} \ \bigg{|} \ h>0 \bigg{\}}, \\[1em] 
&\leq 4B\sup\bigg{\{}h^{\frac{1}{2}}e^{-(\beta-\alpha)h} \ \bigg{|} \ h>0 \bigg{\}}, \\[1em] 
&\leq 4B\sup\bigg{\{}\bigg{(}he^{-2(\beta-\alpha)h}\bigg{)}^{\frac{1}{2}} \ \bigg{|} \ h>0 \bigg{\}}, \\[1em] 
&\leq 4B\bigg{(}\sup\bigg{\{}he^{-2(\beta-\alpha)h}\ \bigg{|} \ h>0 \bigg{\}}\bigg{)}^{\frac{1}{2}}. \label{OvsMapTheoreme3}
\end{align}
Now, we can deduce that function $he^{-2(\beta-\alpha)h}:(0,\infty)\to\reals$  attains its supremum when $\frac{d}{dh}he^{-2(\beta-\alpha)h} = 0$ that is when $h=\frac{1}{2(\beta-\alpha)}$. Hence it follows from inequality (\ref{OvsMapTheoreme3}) that
\begin{align}
K &\leq \frac{4B}{(\beta-\alpha)^{\frac{1}{2}}}\frac{1}{e\sqrt{2}}.
\end{align}
Now, continuing from equation (\ref{OvsMapTheoremEqn3}) we finally see that  
\begin{align}
\|Qz\|_{\beta} &\leq e^{\underline{\inta}\rho}K\|z\|_{\alpha}, \\
&\leq \frac{4e^{\underline{\inta}\rho}C\mathcal{N}^{q+1}(1+\rho)^{\frac{1}{2}}}{(\beta-\alpha)^{\frac{1}{2}}}\|z\|_{\alpha}, \label{OvsMapTheoremEqn5}
\end{align}
hence the proof is complete. 
\end{proof}
\begin{mdframed}
\begin{remark}
In the following Theorem we will describe an equation of the form
\begin{align}
f(t) = z_{\underline{\inta}}+ \int_{0}^{t}Q(f(s))ds, \ t\in\timeint, \label{ProblemToCite}
\end{align}	 
and rely on our work in subsection \ref{DeterministicEqn} to conclude, with the choice 
\[\mathbf{X}\equiv\mathscr{L}^{1} \quad  \text{and} \quad F\equiv Q,\]
that equation (\ref{ProblemToCite}) has a unique continuous solution, in the context of Theorem (\ref{ABEUtheorem}).
\end{remark}
\end{mdframed}
\begin{theorem}[Comparison Theorem] \label{GronTheorem} \quad\\
Suppose $z_{\underline{\inta}}\in l^{1}_{\underline{\inta}}$, $q<1$ and $Q\coloneqq\{Q_{x,y}\}_{x,y\in \gamma}$ is an element of $\mathcal{O}(\mathscr{L}^{1},L,q)$. Moreover suppose that $Q_{x,y}\geq 0$ for all $x,y\in\gamma$ and, in the context of Theorem (\ref{ABEUtheorem}), let $f$ be the unique continuous solution of the integral equation
\begin{align}
f(t) = z_{\underline{\inta}}+ \int_{0}^{t}Q(f(s))ds, \ t\in\timeint. \label{OldProblem}
\end{align}
Finally, suppose that $g:[0,T]\to l^{1}_{\underline{\inta}}$ is a continuous map such that for all $x\in\gamma$
\begin{align}
g_{x}(t) \leq z_{\underline{\inta},x} + \bigg{[}\int_{0}^{t}Q(g(s))ds\bigg{]}_{x}, \ t\in\timeint. \label{Groncondition}
\end{align}
Then for all $t\in\timeint$ and all $x\in\gamma$
\begin{align}
g_{x}(t) \leq f_{x}(t). \label{GronResult}
\end{align}
\end{theorem}
\begin{proof}
For all $\inta\in\indexint$ let $H_{\inta}=\mathcal{C}([0,T],l^{1}_{\inta})$ and define a family $\mathbf{H}\coloneqq\{H_{\inta}\}_{\inta\in\indexint}$. It follows from subsection \ref{DeterministicEqn} that $\mathbf{H}$ is a scale. Moreover from Theorem \ref{ABIntmap} we know that map $\mathcal{I}:\mathbf{H}(\cup)\to H_{\overline{\inta}}$ defined for all $t\in [0,T]$ and all $\kappa\in H_{\alpha}$ via formula
\begin{align}
\mathcal{I}(\kappa)(t) \coloneqq z_{\underline{\inta}}+ \int_{0}^{t}Q(\kappa(s))ds, \label{IntegralMap}
\end{align}
is an Ovsjannikov map of order $TL$ and $q$ on $\mathbf{H}$. That is $\mathcal{I}\in\mathcal{O}(\mathbf{H},TL,q)$.\\[1em] 
Therefore, using Theorem \ref{ABfixpointtheorem}, we see that if $\underline{\inta}<\beta\in\indexint$ then the sequence $\{\mathcal{I}^{n}(g)\}_{n\in\naturals}$ where
\begin{align}
\begin{rcases} 
\mathcal{I}^{1}(g)(t) &\coloneqq z_{\underline{\inta}}+ \int_{0}^{t}Q(g(s))ds, \\
&\vdots \\
\mathcal{I}^{n+1}(g)(t) &\coloneqq \mathcal{I}(\mathcal{I}^{n}(g))(t),
\end{rcases} \ \forall \ t\in\timeint.
\end{align}
is such that
\begin{align}
\overbrace{\bigg{[}\lim_{n\to\infty}\mathcal{I}^{n}(g)\bigg{]}}^{\text{in} \ \mathcal{C}([0,T],l^{1}_{\beta})} = f.
\end{align}
Therefore it is also true that $\lim_{n\to\infty}\mathcal{I}^{n}_{x}(g)(t)= f_{x}(t)$ for all $x\in\gamma$ and all $t\in\timeint$. Hence to conclude the proof it is sufficient to fix $x\in\gamma$ and $t\in\timeint$ and prove by induction that
\begin{align}
g_{x}(t)\leq \mathcal{I}^{n}_{x}(g)(t), \ \forall \ n\in\naturals. \label{InduationH1}
\end{align}
Case $n=1$ is satisfied by the initial assumption on $g$, so let us now assume that the induction hypothesis (\ref{InduationH1}) is true for some $n\geq 1$ and proceed by considering the following chain of inequalities
\begin{align}
\mathcal{I}^{n+1}_{x}(g)(t) &= \mathcal{I}_{x}(\mathcal{I}^{n}(g))(t)_{,} \label{AS} \\[1em]  
&= z_{\underline{\inta},x} + \bigg{[}\int_{0}^{t}Q(\mathcal{I}^{n}(g)(s))ds\bigg{]}_{x,} \\[1em] 
&= z_{\underline{\inta},x} + \sum_{y\in \gamma}Q_{x,y}\int_{0}^{t}\mathcal{I}^{n}_{y}(g)(s)ds, \\[1em] 
&\geq z_{\underline{\inta},x} + \sum_{y\in p}Q_{x,y}\int_{0}^{t}g_{y}(s)ds, \\[1em] 
&= z_{\underline{\inta},x} + \bigg{[}\int_{0}^{t}Q(g(s))ds\bigg{]}_{x,} \\
&\geq g_{x}(t). \label{AF}
\end{align}
Finally from inequalities (\ref{AS}) - (\ref{AF}) we conclude that ineauality (\ref{InduationH1}) holds hence the proof is complete. 
\end{proof}
\begin{corollary1} \label{GronCorollary}
Suppose that $z_{\underline{\inta},x}\geq 0$ $x\in\gamma$. Moreover assume that components of $g$ are non-negative functions, that is $g_{x}(t)\geq 0$ for all $x\in\gamma$ and all $t\in\timeint$. Then for all $\beta > \alpha \in\indexint$ there exists a constant $K(\alpha,\beta)\in\reals$ such that 
\begin{align}
\sum_{x\in\gamma}e^{-\beta|x|}\sup_{t\in\timeint}g_{x}(t) \leq K(\alpha,\beta)\sum_{x\in\gamma}e^{-\alpha|x|}z_{\underline{\inta},x}. \label{Corrolary1EqnMain}
\end{align}
\end{corollary1}
\begin{proof}
Using Theorem \ref{GronTheorem}, we start by making an observation that for all $x\in\gamma$ and all $t\in\timeint$
\begin{align}
g_{x}(t) & \leq z_{\underline{\inta},x} + \bigg{[}\int_{0}^{t}Q(g(s))ds\bigg{]}_{x,} \\[1em] 
&\leq z_{\underline{\inta},x} + \bigg{[}\int_{0}^{t}Q(f(s))ds\bigg{]}_{x.}
\end{align}
Therefore we see that for all $x\in\gamma$
\begin{align}
\sup_{t\in\timeint}g_{x}(t) & \leq z_{\underline{\inta},x} + \bigg{[}\int_{0}^{T}Q(f(s))ds\bigg{]}_{x,} \\
&= f_{x}(T). \label{Corrolary1Eqn1}
\end{align}
Hence it follows that
\begin{align}
\sum_{x\in\gamma}e^{-\beta|x|}\sup_{t\in\timeint}g_{x}(t) & \leq \sum_{x\in\gamma}e^{-\beta|x|}f_{x}(T), \\
&\leq \|f(T)\|_{l_{\beta}^{1}}. \label{GronCorollaryEqn2}
\end{align}
Norm in the inequality (\ref{GronCorollaryEqn2}) above can be estimated using Theorem \ref{ABnormesttheorem} and remark that proceeds it. In particular we get
\begin{align}
\|f(T)\|_{l_{\beta}^{1}} \leq \sum_{n=0}^{\infty}\frac{L^{n}T^{n}}{(\beta-\alpha)^{q}}\frac{n^{q}}{n!} \|z_{\underline{\inta}}\|_{l_{\alpha}^{1}}.
\end{align}
Finally letting $K(\alpha,\beta) = \sum_{n=0}^{\infty}\frac{L^{n}T^{n}}{(\beta-\alpha)^{q}}\frac{n^{q}}{n!}$ we see that
\begin{align}
\sum_{x\in\gamma}e^{-\beta|x|}\sup_{t\in\timeint}g_{x}(t) \leq K(\alpha,\beta) \|z_{\underline{\inta}}\|_{l_{\alpha}^{1}},
\end{align}
hence the proof is complete.
\end{proof}

\newpage
\section{Truncated Systems} \label{TS}
Throughout this section let us assume that $\reals\ni p\geq 2$. \\[1em]
We now start working with a sequence $\{\Lambda_{n}\}_{n\in\naturals}$ of finite subsets of $\gamma$ such that $\Lambda_{n}\uparrow\gamma$ as $n\to\infty$. Moreover for each $n\in\naturals$ we now wish to introduce and study the following stochastic system, which we will denote by $\mathscr{O}^{p}_{n}$.  

\nomenclature[$\Lambda_{n}$]{$\Lambda_{n}$}{}
\nomenclature[$\Xi^{n}$]{$\Xi^{n}$}{}

\begin{equation} \label{FinVolSystem}
\begin{alignedat}{2} 
&\xi_{x,t}^{n} = \zeta_{x} + \int_{0}^{t}\Phi_{x}(\xi_{x,s}^{n} ,\Xi_{s}^{n} )ds + \int_{0}^{t}\Psi_{x}(\xi_{x,s}^{n} ,\Xi_{s}^{n})dW_{x}(s), && \quad \forall x\in\Lambda_{n}\land t\in\timeint, \\
&\xi_{x,t}^{n}  = \zeta_{x}, &&\quad \forall x\not\in\Lambda_{n}\land t\in\timeint.
\end{alignedat}
\tag{$\mathscr{O}^{p}_{n}$}
\end{equation}
In simple words, for each $n\in\naturals$ system (\ref{FinVolSystem}) is a stoped/truncated version of the system (\ref{MainSystem}), which was described in subsection \ref{systemdescription}.\\[1em]
In this section our goal is to prove two important results concerning systems (\ref{FinVolSystem}). In the subsequent sections these results will help us to establish that system (\ref{MainSystem}) admits a unique strong solution.\\[1em] 
We shall now rely on \cite{ABW,DF} and state the next result without a proof. 
\begin{theorem}  \label{FiniteVolumeLemma}
For all $n\in\naturals$ and $\zeta\in l^{p}_{\underline{\inta}}$ system (\ref{FinVolSystem}) has a continuous solution $\Xi^{n}\in Z_{\underline{\inta}}^{p}$.
\end{theorem}
\begin{mdframed}
	\begin{remark}
	For all $n\in\naturals$, a term solution in the Theorem \ref{FiniteVolumeLemma} above is to be understood in the same sence as explained in the Definition \ref{StrongSol} except we do not require $\Xi^{n}$ to be a map from $\overline{\Omega}$ to $\mathcal{Z}^{p}(\cap)$.
	\end{remark}
\end{mdframed}
\begin{mdframed}
\begin{remark}
Combining Theorems (\ref{FiniteVolumeLemma}) and (\ref{ItoComposition}) with the Definition (\ref{ItoProcess1}) we see that $\xi_{x}^{n}$ in an Itô process for all $n\in\naturals$ and $x\in\gamma$.
\end{remark}
\end{mdframed}
In the next two sections of this document it will be shown that the sequence $\{\Xi^{n}\}_{n\in\naturals}$ converges to the unique strong solution of the system (\ref{MainSystem}). However before this can be achieved we need to establish the following two theorems. \\ \\
\begin{theorem}\label{TailTheorem1}
Suppose that $n\in\naturals$ and $\reals\ni p\geq2$. Moreover let $\Xi^{n}$ be the process defined in the Theorem \ref{FiniteVolumeLemma} and for all $x\in\gamma$ let $\xi^{n}_{x}$ be components of $\Xi^{n}$. Then for all $\underline{\inta}<\alpha\in\indexint$ we have
\begin{align}
\sum_{x\in\gamma}e^{-\alpha|x|}\sup_{n\in\naturals}\sup_{t\in\timeint}\mathbb{E}\bigg{[}|\xi^{n}_{x,t}|^{p}\bigg{]} < \infty. \label{Theorem4MainEqn}
\end{align}
\end{theorem}
\begin{proof}
Let us start by recalling that 	
\begin{equation}
\begin{alignedat}{2} 
&\xi_{x,t}^{n} = \zeta_{x} + \int_{0}^{t}\Phi_{x}(\xi_{x,s}^{n} ,\Xi_{s}^{n} )ds + \int_{0}^{t}\Psi_{x}(\xi_{x,s}^{n} ,\Xi_{s}^{n})dW_{x}(s), && \quad \forall x\in\Lambda_{n}\land t\in\timeint, \\
&\xi_{x,t}^{n}  = \zeta_{x}, &&\quad \forall x\not\in\Lambda_{n}\land t\in\timeint.
\end{alignedat}
\end{equation}
Hence using Itô Lemma \ref{ItoLemma1} we see that if $x\in\Lambda_{n}$ then for all $t\in\timeint$
\begin{equation}
\begin{alignedat}{2}
|\xi_{x,t}^{n}|^{p} = |\zeta_{x}|^{p} + \int_{0}^{t}p(\xi_{x,s}^{n})^{p-1}&\Phi_{x}(\xi_{x,s}^{n},\Xi_{s}^{n})ds +\mathbin{\textcolor{white}{\int_{0}^{t}\frac{(p-1)p}{2}}} && \\[1em]
&+ \int_{0}^{t}\frac{(p-1)p}{2}(\xi_{x,s}^{n})^{p-2}(\Psi_{x}(\xi_{x,s}^{n},\Xi_{s}^{n}))^{2}ds + \\[1em]
&&\!\!\!\!\!\!\!\!\!\!\!\!\!\!\!\!\!\!\!\!\!\!\!\!\!\!\!\!\!\!\!\!\!\!\!\!\!\!\!\!\!\!\!\!\!\!\!\!\!\!\!\!\!\! + \int_{0}^{t}p(\xi_{x,s}^{n})^{p-1}\Psi_{x}(\xi_{x,s}^{n},\Xi_{s}^{n})dW_{x}(s).
\end{alignedat}
\end{equation}
Now from assumptions (\hyperref[conditionA]{C}), (\hyperref[conditionB]{D}) and Lemma \ref{DriftLemma} we can deduce that for all $t\in\timeint$
\begin{align}
(\xi_{x,t}^{n})^{p-1}\Phi_{x}(\xi_{x,t}^{n},\Xi_{t}^{n}) &=(\xi_{x,t}^{n})^{p-2}(\xi_{x,t}^{n})\Phi_{x}(\xi_{x,t}^{n},\Xi_{t}^{n}), \\[1em]
& \leq (\xi_{x,t}^{n})^{p-2}\bigg{[} (b+\frac{1}{2})|\xi_{x,t}^{n}|^{2} + \frac{1}{2}\Tilde{a}_{x}^{2}\sum_{y\in B_{x}}|\xi_{y,t}^{n}|^{2} + \xi_{x,t}^{n}\Phi_{x}(0,0) \bigg{]}, \\[1em]
&\leq  (\xi_{x,t}^{n})^{p-2}\bigg{[} (b+1)|\xi_{x,t}^{n}|^{2} + \Tilde{a}_{x}^{2}\sum_{y\in B_{x}}|\xi_{y,t}^{n}|^{2} + c^{2} \bigg{]}, \\[1em]
& \leq (b+1)|\xi_{x,t}^{n}|^{p} + \Tilde{a}_{x}^{2}|\xi_{x,t}^{n}|^{p-2}\sum_{y\in B_{x}}|\xi_{y,t}^{n}|^{2} + |\xi_{x,t}^{n}|^{p-2}c^{2}, \\[1em]
&\leq (b+1)|\xi_{x,t}^{n}|^{p} + \Tilde{a}_{x}^{2}n_{x}\max_{y\in B_{x}}|\xi_{y,t}^{n}|^{p-2}\max_{y\in B_{x}}|\xi_{y,t}^{n}|^{2} + |\xi_{x,t}^{n}|^{p-2}c^{2}, \\[1em]
&\leq (b+1)|\xi_{x,t}^{n}|^{p} + \Tilde{a}_{x}^{2}n_{x}\max_{y\in B_{x}}|\xi_{y,t}^{n}|^{p} + (1+ |\xi_{x,t}^{n}|)^{p}c^{2}.
\end{align}
Now using in addition Theorem \ref{powerinequalitytheorem} we see that for all $t\in\timeint$ we have
\begin{align}
(\xi_{x,t}^{n})^{p-1}\Phi_{x}(\xi_{x,t}^{n},\Xi_{t}^{n}) &\leq (b+1)|\xi_{x,t}^{n}|^{p} +\Tilde{a}_{x}^{2}n_{x}\sum_{y\in B_{x}}|\xi_{y,t}^{n}|^{p} + 2^{p-1}c^{2} + 2^{p-1}c^{2}|\xi_{x,t}^{n}|^{p}, \\[1em]
&\leq (b+1+2^{p-1}c^{2})|\xi_{x,t}^{n}|^{p} +\Tilde{a}_{x}^{2}n_{x}\sum_{y\in B_{x}}|\xi_{y,t}^{n}|^{p} + 2^{p-1}c^{2},\\[1em]
&\leq (b+1+2^{p-1}c^{2})|\xi_{x,t}^{n}|^{p} +\bar{a}^{2}n_{x}^{3}\sum_{y\in B_{x}}|\xi_{y,t}^{n}|^{p} + 2^{p-1}c^{2}, \label{ConditionBConsequence}
\end{align}

Moreover from assumption (\hyperref[conditionC]{E}) we know that for all $t\in\timeint$
\begin{align}
(\xi_{x,s}^{n})^{p-2}(\Psi_{x}(\xi_{x,s}^{n},\Xi_{s}^{n}))^{2} &\leq (\xi_{x,s}^{n})^{p-2}\bigg{[} 4M_{1}^{2}|\xi_{x,t}^{n}|^{2} +4M_{2}^{2}n_{x}^{2}\bigg{(} \sum_{y\in B_{x}}|\xi_{y,t}^{n}|\bigg{)}^{2} + 4|\Psi_{x}(0,0)|^{2}\bigg{]}, \\[1em]
& \leq (\xi_{x,s}^{n})^{p-2}\bigg{[} 4M_{1}^{2}|\xi_{x,t}^{n}|^{2} +4M_{2}^{2}n_{x}^{3}\sum_{y\in B_{x}}|\xi_{y,t}^{n}|^{2} + 4c^{2}\bigg{]}, \\[1em]
&\leq 4M_{1}^{2}|\xi_{x,t}^{n}|^{p} + 4M_{2}^{2}n_{x}^{4}\max_{y\in B_{x}}|\xi_{y,t}^{n}|^{p-2}\max_{y\in B_{x}}|\xi_{y,t}^{n}|^{2} +4c^{2}|\xi_{x,s}^{n}|^{p-2}, \\[1em]
&\leq 4M_{1}^{2}|\xi_{x,t}^{n}|^{p} + 4M_{2}^{2}n_{x}^{4}\sum_{y\in B_{x}}|\xi_{y,t}^{n}|^{p} + 4c^{2}2^{p-1}(1+|\xi_{x,s}^{n}|^{p}), \\[1em]
&\leq (4M_{1}^{2}+4c^{2}2^{p-1})|\xi_{x,t}^{n}|^{p} + 4M_{2}^{2}n_{x}^{4}\sum_{y\in B_{x}}|\xi_{y,t}^{n}|^{p} + 4c^{2}2^{p-1}. \label{ConditionCConsequence}
\end{align}
Now letting
\begin{align}
&A_{1} \coloneqq (b+1+2^{p-1}c^{2}),  \label{ConditionCConsequencel1} \\
&A_{2} \coloneqq (4M_{1}^{2}+4c^{2}2^{p-1}),  \label{ConditionCConsequencel2}\\
&A_{3} \coloneqq (p\bar{a}^{2}+p^{2}4M_{2}^{2}),  \label{ConditionCConsequencel3}\\
&A_{4}\coloneqq 5p^{2}2^{p}c^{2}T.  \label{ConditionCConsequencel4}
\end{align}
we observe from inequalities (\ref{ConditionBConsequence}) and (\ref{ConditionCConsequence}) together with the system (\ref{FinVolSystem}) that for all $x\in\Lambda_{n}$ we have
\begin{equation}
\begin{alignedat}{2}
\mathbb{E}\bigg{[}|\xi_{x,t}^{n}|^{p}\bigg{]} \leq p^{2}(A_{1}+A_{2})\int_{0}^{t}\mathbb{E}\bigg{[}|\xi_{x,s}^{n}|^{p}\bigg{]}ds +
 &+ A_{3}n_{x}^{4}\sum_{y\in B_{x}}\int_{0}^{t}\mathbb{E}\bigg{[}|\xi_{y,s}^{n}|^{p}\bigg{]}ds + A_{4},\ t\in\timeint. \label{1stSolIneq}
\end{alignedat}
\end{equation}
Now we fix an arbitrary $n\in\naturals$ and also define a measurable map $\eta^{n}:\timeint\to l^{1}_{\underline{\inta}}$, that is a map $\eta^{n}\in\mathcal{M}(\mathbf{M},\mathbf{M}^{p}_{\underline{\inta}})$, via the following formula 
\begin{align}
\eta_{x}^{n}(t)\coloneqq\max_{m\leq n} \ \mathbb{E}\bigg{[}|\xi_{x,t}^{m}|^{p}\bigg{]},\ \forall t\in\timeint.
\end{align}
Hence we deduce from the inequality (\ref{1stSolIneq}) and from the system (\ref{FinVolSystem}) that for all $x\in\gamma$
\begin{align}
\eta_{x}^{n}(t) \leq \sum_{y\in\gamma}Q_{x,y}\int_{0}^{t}\eta_{y}^{n}(s)ds + A_{x},\ t\in\timeint.  \label{IneqA}
\end{align}
where
\begin{align}
Q_{x,y} = \begin{cases} 
p^{2}(A_{1}+A_{2}) + A_{3}n_{x}^{4}, &x=y,\\ 
A_{3}n_{x}^{4}, &0<|x-y| < \rho,\\
0, &|x-y|> \rho.
\end{cases} \label{Matrix1}
\end{align}
and
\begin{align}
A_{x} = |\zeta_{x}|^{p} + A_{4}.
\end{align}
Moreover the following facts can now also be deduced.
\begin{enumerate}
	\item $A\in l^{1}_{\underline{\inta}}$ as a result of Theorem \ref{Ninlp} and the choice $\zeta\in l^{p}_{\underline{\inta}}$.
	\item Using Theorem \ref{Continuity Lemma}, we see that $\eta^{n}\in C([0,T],l^{1}_{\underline{\inta}})$,
	\item From equation (\ref{Matrix1}) we see that there exists a constant $C$ such that $|Q_{x,y}|\leq Cn_{x}^{4}$. Therefore using Theorem \ref{OvsMapTheorem} we conclude that there exists some $L\in\reals^{0}$ such that $Q$ is the Ovsjannikov operator of order $L$ and $\frac{1}{2}$ on $\mathcal{L}^{1}$. 
\end{enumerate}
Now since $n\in\naturals$ was arbitrary, application of Theorem \ref{GronTheorem} and Corrolary \ref{GronCorollary} to the inequality (\ref{IneqA}) tells us that for all $n\in\naturals$ we have
\begin{align}
\sum_{x\in\gamma}e^{-\alpha|x|}\sup_{t\in\timeint}\eta_{x}^{n}(t) \leq K(\underline{\inta}, \alpha)\sum_{x\in\gamma}e^{-\underline{\inta}|x|}|A_{x}|.
\end{align}
Hence we see that
\begin{align}
\sum_{x\in\gamma}e^{-\alpha|x|}\sup_{t\in\timeint} \max_{m\leq n}\mathbb{E}\bigg{[}|\xi_{x,t}^{m}|^{p}\bigg{]} \leq K(\underline{\inta}, \alpha)\sum_{x\in\gamma}e^{-\underline{\inta}|x|}|A_{x}|. \label{EqnForFatoux1}
\end{align}
Therefore
\begin{align}
\sup_{n\in\naturals}\bigg{\{}\sum_{x\in\gamma}e^{-\alpha|x|}\sup_{t\in\timeint} \max_{m\leq n}\mathbb{E}\bigg{[}|\xi_{x,t}^{m}|^{p}\bigg{]}\bigg{\}} \leq K(\underline{\inta}, \alpha)\sum_{x\in\gamma}e^{-\underline{\inta}|x|}|A_{x}|. \label{Theorem5EqnAA}
\end{align}
\begin{mdframed}
	\begin{remark}
	Consider now arbitrary $x\in\gamma$. It is clear that 
	\[\sup_{n\in\naturals}\bigg{(}\max_{m\leq n}\sup_{t\in\timeint}\mathbb{E}\bigg{[}|\xi_{x,t}^{m}|^{p}\bigg{]}\bigg{)} \leq \sup_{n\in\naturals}\sup_{t\in\timeint}\mathbb{E}\bigg{[}|\xi_{x,t}^{n}|^{p}\bigg{]}. \] 	
	Moreover for all $\epsilon>0$ there exists $k\in\naturals$ such that 
	\begin{align}
	\sup_{n\in\naturals}\sup_{t\in\timeint}\mathbb{E}\bigg{[}|\xi_{x,t}^{n}|^{p}\bigg{]} - \epsilon &\leq \sup_{t\in\timeint}\mathbb{E}\bigg{[}|\xi_{x,t}^{k}|^{p}\bigg{]},\\[1em]
	&\leq \max_{m\leq k}\sup_{t\in\timeint}\mathbb{E}\bigg{[}|\xi_{x,t}^{m}|^{p}\bigg{]},\\[1em]
	&\leq \sup_{n\in\naturals}\bigg{(}\max_{m\leq n}\sup_{t\in\timeint}\mathbb{E}\bigg{[}|\xi_{x,t}^{m}|^{p}\bigg{]}\bigg{)}.
	\end{align}
	Since $\epsilon$ is arbitrary It follows that 
	\begin{align}
	\sup_{t\in\timeint}\sup_{n\in\naturals}\mathbb{E}\bigg{[}|\xi_{x,t}^{n}|^{p}\bigg{]} &= \sup_{n\in\naturals}\bigg{(}\max_{m\leq n}\sup_{t\in\timeint}\mathbb{E}\bigg{[}|\xi_{x,t}^{m}|^{p}\bigg{]}\bigg{)}, \\[1em]
	&=\sup_{n\in\naturals}\bigg{(}\sup_{t\in\timeint}\max_{m\leq n}\mathbb{E}\bigg{[}|\xi_{x,t}^{m}|^{p}\bigg{]}\bigg{)}.
	\end{align}
	\end{remark}
\end{mdframed}
Remark above shows that if an arbitrary set $A\subset\gamma$ is finite then
\begin{align}
\sup_{n\in\naturals}\bigg{\{}\sum_{x\in A}e^{-\alpha|x|}\sup_{t\in\timeint} \max_{m\leq n}\mathbb{E}\bigg{[}|\xi_{x,t}^{m}|^{p}\bigg{]}\bigg{\}} = \sum_{x\in A}e^{-\alpha|x|}\sup_{t\in\timeint} \sup_{n\in\naturals}\mathbb{E}\bigg{[}|\xi_{x,t}^{m}|^{p}\bigg{]}.
\end{align}
Hence from inequality (\ref{Theorem5EqnAA}) we finally learn that 
\begin{align}
\sum_{x\in\gamma}e^{-\alpha|x|}\sup_{t\in\timeint}\sup_{n\in\naturals}\mathbb{E}\bigg{[}|\xi^{n}_{x,t}|^{p}\bigg{]} \leq K(\underline{\inta}, \alpha)\sum_{x\in\gamma}e^{-\underline{\inta}|x|}|A_{x}|, \label{Theorem5EqnAAA}
\end{align}
and the proof is complete.
\end{proof}

\begin{theorem} \label{CauchySequenceTheorem}
Suppose that $\reals\ni p\geq2$ and for all $n\in\naturals$ process $\Xi^{n}$ is the solution of the truncated system (\ref{FinVolSystem}) as defined in the Theorem \ref{FiniteVolumeLemma}. Then for all $\underline{\inta}<\alpha\in\indexint$ sequence $\{\Xi^{n}\}_{n\in\naturals}$ is Cauchy in $Z^{p}_{\alpha}.$	
\end{theorem}

\begin{proof}
Fix $n,m\in\naturals$ and define
\begin{align}
\bar{\Xi}^{n,m} \coloneqq \Xi^{n}-\Xi^{m}.
\end{align}
In addition let us assume, without loss of generality, that $\Lambda_{n}\subset\Lambda_{m}$. For all $x\in\gamma$ we shall now estimate components $\bar{\xi}^{n,m}_{x}$ of  $\bar{\Xi}^{n,m}$ by considering three separate cases namely; $x\not\in\Lambda_{m}$, $x\in\Lambda_{n}$ and $x\in\Lambda_{m}-\Lambda_{n}$.\\[1em]
First of all, from the definition of the system (\ref{FinVolSystem}) we see that if $x\not\in\Lambda_{m}$ then we have
\begin{align}
\bar{\xi}^{n,m}_{x,t}=0,\ \forall t\in\timeint. \label{Theorem5Eqn1A}
\end{align}
Let us now define for all $x\in\gamma$ and all $t\in\timeint$ the following processes
\begin{align}
\Phi_{x}^{n,m}(t)&\coloneqq\Phi_{x}(\xi^{n}_{x,t},\Xi^{n}_{t}) - \Phi_{x}(\xi^{m}_{x,t},\Xi^{m}_{t}), \label{CauchySequenceTheorem121} \\[1em]
\Psi_{x}^{n,m}(t)&\coloneqq\Psi_{x}(\xi^{n}_{x,t},\Xi^{n}_{t}) - \Psi_{x}(\xi^{m}_{x,t},\Xi^{m}_{t}),  \label{CauchySequenceTheorem122}
\end{align}
and consider the situation when $x\in\Lambda_{n}$.  In this case we have
\begin{align}
\bar{\xi}^{n,m}_{x,t} = \int_{0}^{t}\Phi_{x}^{n,m}(s)ds + \int_{0}^{t}\Psi_{x}^{n,m}(s)dW_{x}(s),\ t\in\timeint.  \label{CauchySequenceTheorem123}
\end{align}
Hence using Itô Lemma \ref{ItoLemma1} we see that if $x\in\Lambda_{n}$ then for all $t\in\timeint$
\begin{equation}
\begin{alignedat}{2}
|\bar{\xi}^{n,m}_{x,t}|^{p} = \int_{0}^{t}p(\bar{\xi}^{n,m}_{x,s})^{p-1}&\Phi_{x}^{n,m}(s)ds + \mathbin{\textcolor{white}{\int_{0}^{t}\frac{p(p-1)}{2}(\bar{\xi}^{n,m}_{x,s})^{p-2}(\Psi_{x}^{n,m}(s))^{2}ds}}&& \\[1em]
&+ \int_{0}^{t}\frac{p(p-1)}{2}(\bar{\xi}^{n,m}_{x,s})^{p-2}(\Psi_{x}^{n,m}(s))^{2}ds + \\[1em]
&&\!\!\!\!\!\!\!\!\!\!\!\!\!\!\!\!\!\!\!\!\!\!\!\!\!\!\!\!\!\!\!\!\!\!\!\!\!\!\!\!\!\!\!\!\!\!\!\!\!\!\!\!\!\!\!\!\!\! + \int_{0}^{t}p(\bar{\xi}^{n,m}_{x,t})^{p-1}\Psi_{x}^{n,m}(s)dW_{x}(s).  \label{Theorem5EqnZero}
\end{alignedat}
\end{equation}
Now, from Lemma \ref{DriftLemma} we can see that for all $t\in\timeint$ we have
\begin{align}
(\bar{\xi}^{n,m}_{x,t})^{p-1}\Phi_{x}^{n,m}(t) &= (\bar{\xi}^{n,m}_{x,t})^{p-2}\bar{\xi}^{n,m}_{x,t}\bigg{(}\Phi_{x}(\xi^{n}_{x,t},\Xi^{n}_{t}) - \Phi_{x}(\xi^{m}_{x,t},\Xi^{m}_{t})\bigg{)}, \\[1em]
&\leq (\bar{\xi}^{n,m}_{x,t})^{p-2}\bigg{(} (b+\frac{1}{2})(\xi^{n}_{x,t} - \xi^{m}_{x,t})^{2} + \frac{1}{2}\Tilde{a}_{x}^{2}\sum_{y\in B_{x}}(\xi^{n}_{y,t} - \xi^{m}_{y,t})^{2} \bigg{)}, \\[1em]
&\leq (b+1)|\bar{\xi}^{n,m}_{x,t}|^{p} + \max_{y\in B_{x}}|\bar{\xi}^{n,m}_{x,t}|^{p-2}\bigg{(}\Tilde{a}_{x}^{2}n_{x}\max_{y\in B_{x}}|\bar{\xi}^{n,m}_{x,t}|^{2}\bigg{)}, \\[1em]
&\leq (b+1)|\bar{\xi}^{n,m}_{x,t}|^{p} + \Tilde{a}_{x}^{2}n_{x}\max_{y\in B_{x}}|\bar{\xi}^{n,m}_{x,t}|^{p}, \\[1em]
&\leq (b+1)|\bar{\xi}^{n,m}_{x,t}|^{p} +  \Tilde{a}_{x}^{2}n_{x}\sum_{y\in B_{x}}|\bar{\xi}^{n,m}_{x,t}|^{p}, \\[1em]
&\leq (b+1)|\bar{\xi}^{n,m}_{x,t}|^{p} +  \bar{a}^{2}n_{x}^{3}\sum_{y\in B_{x}}|\bar{\xi}^{n,m}_{x,t}|^{p}. \label{QQ1}
\end{align}
Moreover, using assumption (\hyperref[conditionC]{E}) we can see that for all $t\in\timeint$ we also have
\begin{align}
(\bar{\xi}^{n,m}_{x,t})^{p-2}(\Psi_{x}^{n,m}(t))^{2} &= (\bar{\xi}^{n,m}_{x,t})^{p-2}\bigg{(}\Phi_{x}(\xi^{n}_{x,t},\Xi^{n}_{t}) - \Phi_{x}(\xi^{m}_{x,t},\Xi^{m}_{t})\bigg{)}^{2}, \\[1em]
&\leq (\bar{\xi}^{n,m}_{x,t})^{p-2}\bigg{(} 2M_{1}^{2}(\xi^{n}_{x,t} - \xi^{m}_{x,t})^{2} + 2M_{2}^{2}n_{x}^{3}\sum_{y\in B_{x}}(\xi^{n}_{y,t} - \xi^{m}_{y,t})^{2}\bigg{)}, \\[1em]
&\leq 2M_{1}^{2} |\bar{\xi}^{n,m}_{x,t}|^{p} + \max_{y\in B_{x}}|\bar{\xi}^{n,m}_{x,t}|^{p-2}\bigg{(}2M_{2}^{2}n_{x}^{4}\max_{y\in B_{x}}|\bar{\xi}^{n,m}_{x,t}|^{2}\bigg{)}, \\[1em]
&\leq 2M_{1}^{2} |\bar{\xi}^{n,m}_{x,t}|^{p} + 2M_{2}^{2}n_{x}^{4}\max_{y\in B_{x}}|\bar{\xi}^{n,m}_{x,t}|^{p}, \\[1em]
&\leq 2M_{1}^{2} |\bar{\xi}^{n,m}_{x,t}|^{p} + 2M_{2}^{2}n_{x}^{4}\sum_{y\in B_{x}}|\bar{\xi}^{n,m}_{x,t}|^{p}. \label{QQ2}
\end{align}
Therefore letting
\begin{align}
B_{1}&\coloneqq (b+1 + 2M_{1}^{2}),\\
B_{2}&\coloneqq (p\bar{a}^{2} + 2p^{2}M_{2}^{2}),
\end{align}
we can deduce from equation \ref{Theorem5EqnZero} that if $x\in\Lambda_{n}$ then
\begin{equation}
\begin{alignedat}{2}
\mathbb{E}\bigg{[}|\bar{\xi}^{n,m}_{x,t}|^{p}\bigg{]} \leq p^{2}B_{1}\int_{0}^{t}\mathbb{E}\bigg{[}|\bar{\xi}^{n,m}_{x,s}|^{p}\bigg{]}ds + B_{2}n_{x}^{4}\sum_{y\in B_{x}}\int_{0}^{t}\mathbb{E}\bigg{[}|\bar{\xi}^{n,m}_{y,s}|^{p}\bigg{]}ds,\ t\in\timeint. \label{Theorem5EqnA}
\end{alignedat}
\end{equation}
Finally, when $x\in\Lambda_{m}-\Lambda_{n}$ we see using Theorem \ref{powerinequalitytheorem} that for all $t\in\timeint$
\begin{align}
|\bar{\xi}^{n,m}_{x,t}|^{p} &\leq (|\xi^{n}_{x,t}| + |\xi^{m}_{x,t}|)^{p}, \\[1em]
&\leq  2^{p-1}|\xi^{n}_{x,t}|^{p} + 2^{p-1}|\xi^{m}_{x,t}|^{p}.
\end{align}
Therefore, using Theorem \ref{TailTheorem1} and equation (\ref{Theorem5Eqn1A}), we see now that if $x\in\Lambda_{m}-\Lambda_{n}$ then for all $t\in\timeint$ we have
\begin{align}
\mathbb{E}\bigg{[}|\bar{\xi}^{n,m}_{x,t}|^{2}\bigg{]} &\leq 2^{p}\sup_{n\in\naturals}\mathbb{E}\bigg{[}|\xi^{n}_{x,t}|^{p}\bigg{]}, \\[1em]
&\leq 2^{p}\mathbbm{1}_{\Lambda_{m}-\Lambda_{n}}(x)\sup_{n\in\naturals}\sup_{t\in\timeint}\mathbb{E}\bigg{[}|\xi^{n}_{x,t}|^{p}\bigg{]}. \label{Theorem5EqnB}
\end{align}
Therefore we can finally deduce, combining equations (\ref{Theorem5Eqn1A}), (\ref{Theorem5EqnA}) and (\ref{Theorem5EqnB}), that all $x\in\gamma$ and for all $t\in\timeint$ we have
\begin{equation}\label{Theorem5EqnC}
\begin{split}
\mathbb{E}\bigg{[}|\bar{\xi}^{n,m}_{x,t}|^{p}\bigg{]} \leq p^{2}B_{1}\int_{0}^{t}\mathbb{E}\bigg{[}|\bar{\xi}^{n,m}_{x,s}|^{p}\bigg{]}&ds +\\[1em]
&\!\!\!\!\!\!\!\! + B_{2}n_{x}^{4}\sum_{y\in B_{x}}\int_{0}^{t}\mathbb{E}\bigg{[}|\bar{\xi}^{n,m}_{y,s}|^{p}\bigg{]}ds +\\[1em]
&\quad\quad\quad\quad\quad\quad\quad +2^{p}\mathbbm{1}_{\Lambda_{m}-\Lambda_{n}}(x)\sup_{n\in\naturals}\sup_{t\in\timeint}\mathbb{E}\bigg{[}|\xi^{n}_{x,t}|^{p}\bigg{]}. 
\end{split}
\end{equation}
Now, as in the proof of Theorem \ref{TailTheorem1}, infinite system of inequalities (\ref{Theorem5EqnC}) can be rewritten in the following way.\\[1em]
Define, relying on the inequality (\ref{Theorem5EqnC}) a measurable map $\varrho^{n,m}:\timeint\to\reals^{\gamma}$, that is a map $\varrho^{n,m}\in\mathcal{M}(\mathbf{M},\mathbf{M}^{p}_{\underline{\inta}})$, via the following formula 
\begin{align}
\varrho_{x}^{n,m}(t)\coloneqq \mathbb{E}\bigg{[}|\bar{\xi}^{n,m}_{x,t}|^{p}\bigg{]},\ \forall t\in\timeint,
\end{align}
and deduce from inequalities (\ref{Theorem5EqnA}) - (\ref{Theorem5EqnB}) that 
\begin{align}
\varrho_{x}^{n,m}(t) \leq \sum_{y\in\gamma}Q_{x,y}\int_{0}^{t}\varrho_{y}^{n,m}(s)ds+A_{x},\ t\in\timeint,
\end{align}
where 
\begin{align}
Q_{x,y} = \begin{cases} 
p^{2}B_{1} + B_{2}n_{x}^{4}, &x=y,\\ 
B_{2}n_{x}^{4}, &0<|x-y| < p,\\
0, &|x-y|> p.
\end{cases} \label{MatrixQ2}
\end{align}
and
\begin{align}
A_{x} = 2^{p}\mathbbm{1}_{\Lambda_{m}-\Lambda_{n}}(x)\sup_{n\in\naturals}\sup_{t\in\timeint}\mathbb{E}\bigg{[}|\xi^{n}_{x,t}|^{p}\bigg{]}.
\end{align}
Now, fixing $\underline{\inta}<\tilde{\alpha}<\alpha\in\indexint$ we can also deduce the following facts.
\begin{enumerate}
	\item $A\in l^{1}_{\tilde{\alpha}}$ as a result of Theorem \ref{TailTheorem1}. 
	\item Using Theorem \ref{Continuity Lemma}, we see that $\varrho^{n,m}\in C([0,T], l^{1}_{\tilde{\alpha}})$,
	\item From equation (\ref{Matrix1}) we see that there exists a constant $D$ such that $|Q_{x,y}|\leq Dn_{x}^{4}$. Therefore using Theorem \ref{OvsMapTheorem} we conclude that there exists some $L\in\reals^{0}$ such that $Q$ is the Ovsjannikov operator of order $L$ and $\frac{1}{2}$ on $\mathcal{L}^{1}$. 
\end{enumerate}
Therefore we can now use Theorem \ref{GronTheorem} and Corollary \ref{GronCorollary} to conclude that
\begin{align}
\sum_{x\in\gamma}e^{-\alpha|x|}\sup_{t\in\timeint}\varrho_{x}^{n,m}(t) \leq K(\tilde{\alpha}, \alpha)\sum_{x\in\gamma}e^{-\tilde{\alpha}|x|}|A_{x}|. \label{Theorem5EqnD}
\end{align}
From equation (\ref{Theorem5EqnD}) and definition (\ref{Zspaces}) we therefore see that we have the following estimate
\begin{align}
\|\Xi^{n} - \Xi^{m}\|_{Z^{p}_{\alpha}}^{p} &= \sup_{t\in\timeint}\mathbb{E}\bigg{[}||\Xi^{n}_{t} - \Xi^{m}_{t}||^{p}_{l^{p}_{\alpha}}\bigg{]}, \\[1em]
&=\sup_{t\in\timeint} \mathbb{E}\bigg{[}\sum_{x\in\gamma}e^{-\alpha|x|}|\xi^{n}_{x,t} - \xi^{m}_{x,t}|^{p}\bigg{]}, \\[1em]
&\leq \sum_{x\in\gamma}e^{-\alpha|x|}\sup_{t\in\timeint}\varrho_{x}^{n,m}(t), \\[1em]
&\leq K(\tilde{\alpha}, \alpha)\sum_{x\in\gamma}e^{-\tilde{\alpha}|x|}|A_{x}|, \\[1em]
&\leq K(\tilde{\alpha}, \alpha) \sum_{x\in\gamma}e^{-\tilde{ \alpha}|x|}2^{p}\mathbbm{1}_{\Lambda_{m}-\Lambda_{n}}(x)\sup_{n\in\naturals}\sup_{t\in\timeint}\mathbb{E}\bigg{[}|\xi^{n}_{x,t}|^{p}\bigg{]}, \\[1em]
&\leq 2^{p}K(\tilde{\alpha}, \alpha) \sum_{x\in\Lambda_{m}-\Lambda_{n}}e^{-\tilde{\alpha}|x|}\sup_{n\in\naturals}\sup_{t\in\timeint}\mathbb{E}\bigg{[}|\xi^{n}_{x,t}|^{p}\bigg{]}. \label{CauchyEqn}
\end{align}
Estimate above implies that the right hand side of equation (\ref{CauchyEqn}) is the remainder of a convergent series hence the proof is complete. 
\end{proof}
\newpage

%EVERYTHING BEFORE EXISTENCE AND UNIQUENESS.
\section{One Dimesional Special Case} \label{ODSC}
Suppose that $\reals\ni p\geq2$ and $\underline{\inta}<\alpha\in\indexint$. For all $n\in\naturals$ let $\Xi^{n}$ be a solution of the truncated system (\ref{FinVolSystem}) and using Theorem \ref{CauchySequenceTheorem} let $\{\Xi^{n}\}_{n\in\naturals}$ be a Cauchy sequence in $Z^{p}_{\alpha}$. Since $Z^{p}_{\alpha}$ is a Banach space, by Theorem \ref{ZspaceisBanach}, we now define the following process
\begin{align}
\overbrace{\ \Xi \coloneqq \lim_{n\to\infty}\Xi^{n}\ }^{\text{in}\ Z_{\alpha}^{p}}. \label{mainprocess}
\end{align}
Consider now an arbitrary $x\in\gamma$. The main goal of this section is to prove that the following stochastic integral equation 
\begin{align}
\eta_{x,t} &= \zeta_{x} + \int_{0}^{t}\Phi_{x}(\eta_{x,s},\Xi_{s})ds + \int_{0}^{t}\Psi_{x}(\eta_{x,s},\Xi_{s})dW_{x}(s),\ t\in\timeint, \label{maineqn}
\end{align}
has a solution in $\mathcal{S}(\mathbf{M}^{\reals})$. We begin our work by proving an auxiliary result.
\begin{theorem}\label{TAB}
	Suppose that $x\in\gamma$ and $\Xi$ is a process defined by (\ref{mainprocess}). Then
	\begin{align}
	\mathbb{E}\bigg{[}
	\sup_{t\in\timeint}|\xi_{x,t}|^{p}\bigg{]} <\infty.
	\end{align}
\end{theorem}
\begin{proof}
We shall prove this theorem by showing that for all $\epsilon>0$ there exist $N\in\naturals$ such that for all $n,m\geq N$ we have 
\begin{align}
\mathbb{E}\bigg{[}
\sup_{t\in\timeint}|\xi^{n}_{x,t} - \xi^{m}_{x,t}|^{p}\bigg{]} <\epsilon,
\end{align}
where for all $n\in\naturals$ processes $\xi^{n}_{x}$ are components of $\Xi^{n}$.\\[1em]
Since $\Lambda_{n}\uparrow\gamma$ we begin by finding some $\bar{N}\in\naturals$ such that $x\in\Lambda_{\bar{N}}$ and temporary fixing some $n,m\geq \bar{N}$. Moreover let us assume, without loss of generality, that $n<m$ so that $x\in\Lambda_{n}\subset\Lambda_{m}$ and we define
\begin{align}
\bar{\xi}_{x,t}^{n,m}\coloneqq \xi^{n}_{x,t} - \xi^{m}_{x,t},\ \forall t\in\timeint.
\end{align}
Now we recal Theorem \ref{CauchySequenceTheorem}. In particular we are interested in using definitions (\ref{CauchySequenceTheorem121}) - (\ref{CauchySequenceTheorem122}) and an equation (\ref{CauchySequenceTheorem123}). \\
Hence an application of Itô Lemma shows that for all $t\in\timeint$
\begin{equation}
\begin{alignedat}{2}
|\bar{\xi}^{n,m}_{x,t}|^{p} = \int_{0}^{t}p(\bar{\xi}^{n,m}_{x,s})^{p-1}&\Phi_{x}^{n,m}(s)ds + \mathbin{\textcolor{white}{\int_{0}^{t}\frac{p(p-1)}{2}(\bar{\xi}^{n,m}_{x,s})^{p-2}(\Psi_{x}^{n,m}(s))^{2}ds}}&& \\[1em]
&+ \int_{0}^{t}\frac{p(p-1)}{2}(\bar{\xi}^{n,m}_{x,s})^{p-2}(\Psi_{x}^{n,m}(s))^{2}ds + \\[1em]
&&\!\!\!\!\!\!\!\!\!\!\!\!\!\!\!\!\!\!\!\!\!\!\!\!\!\!\!\!\!\!\!\!\!\!\!\!\!\!\!\!\!\!\!\!\!\!\!\!\!\!\!\!\!\!\!\!\!\! + \int_{0}^{t}p(\bar{\xi}^{n,m}_{x,t})^{p-1}\Psi_{x}^{n,m}(s)dW_{x}(s). \label{AA1}
\end{alignedat}
\end{equation}
Therefore we see from equation (\ref{AA1}) above that 
\begin{equation}
\begin{alignedat}{2}
\sup_{t\in\timeint}|\bar{\xi}^{n,m}_{x,t}|^{p} = \int_{0}^{T}p(\bar{\xi}^{n,m}_{x,s})^{p-1}&\Phi_{x}^{n,m}(s)ds + \mathbin{\textcolor{white}{\int_{0}^{T}\frac{p(p-1)}{2}(\bar{\xi}^{n,m}_{x,s})^{p-2}(\Psi_{x}^{n,m}(s))^{2}ds}}&& \\[1em]
&+ \int_{0}^{T}\frac{p(p-1)}{2}(\bar{\xi}^{n,m}_{x,s})^{p-2}(\Psi_{x}^{n,m}(s))^{2}ds + \\[1em]
&&\!\!\!\!\!\!\!\!\!\!\!\!\!\!\!\!\!\!\!\!\!\!\!\!\!\!\!\!\!\!\!\!\!\!\!\!\!\!\!\!\!\!\!\!\!\!\!\!\!\!\!\!\!\!\!\!\!\!\!\!\!\!\!\!  + \sup_{t\in\timeint}\int_{0}^{t}p(\bar{\xi}^{n,m}_{x,t})^{p-1}\Psi_{x}^{n,m}(s)dW_{x}(s). \label{AA2}
\end{alignedat}
\end{equation}
Moreover from inequalities (\ref{QQ1}) and (\ref{QQ2}) we see that
\begin{align}
(\bar{\xi}^{n,m}_{x,t})^{p-1}\Phi_{x}^{n,m}(t) \leq (b+1)|\bar{\xi}^{n,m}_{x,t}|^{p} +  \bar{a}^{2}n_{x}^{3}\sum_{y\in B_{x}}|\bar{\xi}^{n,m}_{x,t}|^{p}.. \label{AA3}
\end{align}
and
\begin{align}
(\bar{\xi}^{n,m}_{x,t})^{p-2}(\Psi_{x}^{n,m}(t))^{2} \leq 2M_{1}^{2} |\bar{\xi}^{n,m}_{x,t}|^{p} + 2M_{2}^{2}n_{x}^{4}\sum_{y\in B_{x}}|\bar{\xi}^{n,m}_{x,t}|^{p}. . \label{AA4}
\end{align}
Hence from inequality (\ref{AA2}) and inequalities (\ref{AA3}) -  (\ref{AA4}) above we see by letting 
\begin{align}
C_{1}&\coloneqq p^{2}(b+1 + 2M_{1}^{2}),\\[1em]
C_{2}&\coloneqq n_{x}^{4}(p\bar{a}^{2} + 2p^{2}M_{2}^{2}), \\[1em]
K&\coloneqq C_{1} \int_{0}^{T}\mathbb{E}\bigg{[}|\bar{\xi}^{n,m}_{x,s}|^{p}\bigg{]}ds + C_{2}\sum_{y\in B_{x}}\int_{0}^{T}\mathbb{E}\bigg{[}|\bar{\xi}^{n,m}_{y,s}|^{p}\bigg{]}ds. \label{AA41}
\end{align}
that we have the following inequality
\begin{equation}
\begin{split}
\mathbb{E}\bigg{[}\sup_{t\in\timeint}|\bar{\xi}^{n,m}_{x,t}|^{p}\bigg{]} \leq K + \mathbb{E}\bigg{[}\sup_{t\in\timeint}\int_{0}^{t}p(\bar{\xi}^{n,m}_{x,s})^{p-1}\Psi_{x}^{n,m}(s)dW_{x}(s)\bigg{]},  \label{AA5}
\end{split}
\end{equation}
Now using results from subsection \ref{WienerProcessinR}, in particular Burkholder-Davis-Gundy inequality \ref{BDGInequality} and also using Jensen inequality \ref{JensenInequality} we see that the following estimate on the stochastic term from the inequality (\ref{AA5}) above holds. 
\begin{align}
\mathbb{E}\bigg{[}\sup_{t\in\timeint}\int_{0}^{t}p(\bar{\xi}^{n,m}_{x,s})^{p-1}\Psi_{x}^{n,m}(s)dW_{x}(s)\bigg{]} & \leq \mathbb{E}\bigg{[}\bigg{(}\int_{0}^{t}\bigg{(}p(\bar{\xi}^{n,m}_{x,s})^{p-1}\Psi_{x}^{n,m}(s)\bigg{)}^{2}ds\bigg{)}^{\frac{1}{2}}\bigg{]}, \\[1em]
&\leq \bigg{(}\mathbb{E}\bigg{[}\int_{0}^{t}\bigg{(}p(\bar{\xi}^{n,m}_{x,s})^{p-1}\Psi_{x}^{n,m}(s)\bigg{)}^{2}ds\bigg{]}\bigg{)}^{\frac{1}{2}}. \label{AA6}
\end{align}
To simplify inequality (\ref{AA6}) we note that according to the (\hyperref[conditionC]{E}) for all $t\in\timeint$ the following estimate is true
\begin{align}
\bigg{(}(\bar{\xi}^{n,m}_{x,t})^{p-1}\Psi_{x}^{n,m}(t)\bigg{)}^{2} &= (\bar{\xi}^{n,m}_{x,t})^{2p-2}\bigg{(}M_{1}|\bar{\xi}^{n,m}_{x,t}| + M_{2}n_{x}\sum_{y\in B_{x}}|\bar{\xi}^{n,m}_{y,t}|\bigg{)}^{2}, \\[1em]
&\leq (\bar{\xi}^{n,m}_{x,t})^{2p-2}\bigg{(}2M_{1}^{2}|\bar{\xi}^{n,m}_{x,t}|^{2} + 2M_{2}^{2}n_{x}^{3}\sum_{y\in B_{x}}|\bar{\xi}^{n,m}_{y,t}|^{2}\bigg{)}, \\[1em]
&\leq 2M_{1}^{2}|\bar{\xi}^{n,m}_{x,t}|^{2p} + \max_{y\in B_{x}}|\bar{\xi}^{n,m}_{y,t}|^{2p-2}\bigg{(}2M_{2}^{2}n_{x}^{4}\max_{y\in B_{x}}|\bar{\xi}^{n,m}_{y,t}|^{2}\bigg{)}, \\[1em]
&\leq 2M_{1}^{2}|\bar{\xi}^{n,m}_{x,t}|^{2p} + 2M_{2}^{2}n_{x}^{4}\sum_{y\in B_{x}}|\bar{\xi}^{n,m}_{y,t}|^{2p}.
\end{align}
Now letting
\begin{align}
C_{3}&\coloneqq 2p^{2}M_{1}^{2}T, \\
C_{4}&\coloneqq 2p^{2}M_{2}^{2}n_{x}^{4}T,
\end{align}
It follows now that inequality (\ref{AA6}) can be written as follows
\begin{equation}
\begin{split}
\mathbb{E}\bigg{[}\sup_{t\in\timeint}\int_{0}^{t}p(\bar{\xi}^{n,m}_{x,s})^{p-1}\Psi_{x}^{n,m}(s)dW_{x}(s)\bigg{]} \leq C_{3}\sup_{t\in\timeint}\mathbb{E}\bigg{[}|\bar{\xi}^{n,m}_{x,t}|^{2p}\bigg{]} + C_{4}\sum_{y\in B_{x}}\sup_{t\in\timeint}\mathbb{E}\bigg{[}|\bar{\xi}^{n,m}_{x,t}|^{2p}\bigg{]},
\end{split}
\end{equation}
Therefore returning to the inequality (\ref{AA5}) we see that
\begin{equation}
\begin{alignedat}{2}
\mathbb{E}\bigg{[}\sup_{t\in\timeint}|\bar{\xi}^{n,m}_{x,t}|^{p}\bigg{]} \leq TC_{1} \sup_{t\in\timeint}\mathbb{E}\bigg{[}|\bar{\xi}^{n,m}_{x,t}|^{p}\bigg{]} &+ TC_{2}\sum_{y\in B_{x}}\sup_{t\in\timeint}\mathbb{E}\bigg{[}|\bar{\xi}^{n,m}_{y,t}|^{p}\bigg{]} + \\[1em]
& + C_{3}\sup_{t\in\timeint}\mathbb{E}\bigg{[}|\bar{\xi}^{n,m}_{x,t}|^{2p}\bigg{]} + C_{4}\sum_{y\in B_{x}}\sup_{t\in\timeint}\mathbb{E}\bigg{[}|\bar{\xi}^{n,m}_{x,t}|^{2p}\bigg{]}. \label{AA7}
\end{alignedat}
\end{equation}
Since $B_{x}$ is finite we can now use Theorem \ref{CauchySequenceTheorem} to conclude that, with a suitable choice of $n,m\in\naturals$, the right hand side of the inequality (\ref{AA7}) above can be made arbitrary small hence the proof is complete. 
\end{proof}

Relying on \cite{ABW} we now state without proof the following result.

\begin{theorem} \label{localtheorem}
There exists $\tau\in\reals$ such that the stochastic integral equation (\ref{maineqn}) admits a unique local maximal solution $\eta_{x}:[0,\tau]\times\Omega\to\reals$. 
\end{theorem} 
Now using Theorem \ref{localtheorem} we can establish the existence of a global solution. Precisely speaking we have the following result. 
\begin{theorem} \label{globaltheorem}
	There exists a solution $\eta_{x}:\overline{\Omega}\to\reals$ of the stochastic integral equation (\ref{maineqn}).
\end{theorem}
\begin{proof}
Clearly if $\tau\geq\timeint$ then there is nothing to prove so let us assume that $\tau < \timeint$. Now by Theorem \ref{localtheorem} there exists a unique, maximal local solution $\eta_{x}:[0,\tau]\times\Omega\to\reals$ to the stochastic equation (\ref{maineqn}). Hence to complete the proof we will show that this solution is also a global one. That is we will establishing that almost surely $\lim_{n\to\infty}\tau_{n} = \infty$, where $\tau_{n}$ is the first exit time of the maximal local solution from the interval (-n,\ n) for all $n\in\naturals$. \\[1em]
Now, since $\eta$ is a local solution to  (\ref{maineqn}) we conclude that for all $n\in\naturals$ and all $t\in[0,\infty)$ we have 
\begin{align}
\eta_{x,t\land\tau_{n}} &= \zeta_{x} + \int_{0}^{t\land\tau_{n}}\Phi_{x}(\eta_{x,s\land\tau_{n}},\Xi_{s\land\tau_{n}})ds + \int_{0}^{t\land\tau_{n}}\Psi_{x}(\eta_{x,s\land\tau_{n}},\Xi_{s\land\tau_{n}})dW_{x}(s). \label{globeqn1}
\end{align}
Hence using Itô Lemma \ref{ItoLemma1} we see that for all $t\in[0,\infty)$ we have the following
\begin{equation}
\begin{alignedat}{2}
|\eta_{x,t\land\tau_{n}}|^{p} = &\int_{0}^{t\land\tau_{n}}p(\eta_{x,s\land\tau_{n}})^{p-1}\Phi_{x}(\eta_{x,s\land\tau_{n}},\Xi_{s\land\tau_{n}})ds + \\[1em]
&\quad\quad\quad\quad\quad\quad + \int_{0}^{t\land\tau_{n}}\frac{p(p-1)}{2}(\eta_{x,s\land\tau_{n}})^{p-2}(\Psi_{x}(\eta_{x,s\land\tau_{n}},\Xi_{s\land\tau_{n}}))^{2}ds + \\[1em]
&\quad\quad\quad\quad \quad\quad\quad\quad\quad\quad\quad + \int_{0}^{t\land\tau_{n}}p(\eta_{x,s\land\tau_{n}})^{p-1}\Psi_{x}(\eta_{x,s\land\tau_{n}},\Xi_{s\land\tau_{n}})dW_{x}(s). \label{BB1}
\end{alignedat}
\end{equation}
Now by letting 
\begin{align}
\bar{\Phi}_{x}^{p}(\eta,t) \coloneqq (\eta_{x,t\land\tau_{n}})^{p-1}\Phi_{x}(\eta_{x,t\land\tau_{n}},\Xi_{t\land\tau_{n}}),\ \forall t\in[0,\infty),
\end{align}
we conclude from inequalities (\ref{ConditionBConsequence}) - (\ref{ConditionCConsequence}) and definitions (\ref{ConditionCConsequencel1}) - (\ref{ConditionCConsequencel4}) that for all $t\in[0,\infty)$
\begin{align}
\Phi_{x}^{p}(\eta,t) &\leq A_{1}|\eta_{x,t\land\tau_{n}}|^{p} + \bar{a}^{2}n_{x}^{3}\sum_{y\in B_{x}}|\xi_{y,t\land\tau_{n}}|^{p} + 2^{p-1}c^{2}, \\[1em]
&\leq A_{1}|\eta_{x,t\land\tau_{n}}|^{p} + \bar{a}^{2}n_{x}^{3}\sum_{y\in B_{x}}\sup_{t\in\timeint}|\xi_{y,t\land\tau_{n}}|^{p} + 2^{p-1}c^{2}, \\[1em]
&\leq A_{1}|\eta_{x,t\land\tau_{n}}|^{p} +\bar{a}^{2}n_{x}^{3}\sum_{y\in B_{x}}\sup_{t\in\timeint}|\xi_{y,t}|^{p} + 2^{p-1}c^{2}, \label{BB2}
\end{align}
and 
\begin{align}
(\eta_{x,s\land\tau_{n}})^{p-2}(\Psi_{x}(\eta_{x,s\land\tau_{n}},\Xi_{s\land\tau_{n}}))^{2} &\leq A_{2}|\eta_{x,t\land\tau_{n}}|^{p} + 4M_{2}^{2}n_{x}^{4}\sum_{y\in B_{x}}|\xi_{y,t\land\tau_{n}}|^{p} + 4c^{2}2^{p-1}, \\[1em]
&\leq A_{2}|\eta_{x,t\land\tau_{n}}|^{p} + 4M_{2}^{2}n_{x}^{4}\sum_{y\in B_{x}}\sup_{t\in\timeint}|\xi_{y,t}|^{p} + 4c^{2}2^{p-1}. \label{BB3}
\end{align}
Threfore, combining inequalities (\ref{BB2}) - (\ref{BB3}) with an inequality (\ref{BB1}) we see that for all $t\in[0,\infty)$ we have
\begin{align}
\mathbb{E}\bigg{[}|\eta_{x,t\land\tau_{n}}|^{p}\bigg{]} & \leq p^{2}(A_{1}+A_{2})\int_{0}^{t}\mathbb{E}\bigg{[}|\eta_{x,s\land\tau_{n}}|^{p}\bigg{]}ds + Tn^{4}_{x}A_{3}\sum_{y\in B_{x}}\mathbb{E}\bigg{[}\sup_{t\in\timeint}|\xi_{y,t}^{n}|^{p}\bigg{]} + A_{4}, \\[1em]
&\leq D\int_{0}^{t}\mathbb{E}\bigg{[}|\eta_{x,s\land\tau_{n}}|^{p}\bigg{]}ds +  K(x). \label{BB4}
\end{align}
Where
\begin{align}
D &\coloneqq p^{2}(A_{1}+A_{2}),\\[1em]
K(x) &\coloneqq Tn^{4}_{x}A_{3}\sum_{y\in B_{x}}\mathbb{E}\bigg{[}\sup_{t\in\timeint}|\xi_{y,t}^{n}|^{p}\bigg{]} + A_{4}. 
\end{align}
Now using Gronwall's inequality \ref{ClassicGronTheorem} together with the inequality (\ref{BB4}) above we see that for all $t\in[0,\infty)$ we have
\begin{align}
\mathbb{E}\bigg{[}|\eta_{x,t\land\tau_{n}}|^{p}\bigg{]} \leq K(x)e^{Dt}. \label{BB5}
\end{align}
However using the definition of $\tau_{n}$ we see that for all $n\in\naturals$ we have $|\eta_{x,\tau_{n}}|\geq n$. Moreover, because $\mathbb{P}(\tau_{n}<t)= \mathbb{E}\bigg{[}\mathbbm{1}_{\{\tau_{n}<t\}}\bigg{]}$ we also see that for all $t\in[0,\infty)$
\begin{align}
n^{p}\mathbb{P}(\tau_{n}<t) &\leq \mathbb{E}\bigg{[}|\eta_{x,\tau_{n}}|^{p}\mathbbm{1}_{\{\tau_{n}<t\}}\bigg{]}, \\[1em]
&\leq \mathbb{E}\bigg{[}|\eta_{x,\tau_{n}}|^{p}\mathbbm{1}_{\{\tau_{n}<t\}}\bigg{]} + \mathbb{E}\bigg{[}|\eta_{x,\tau_{n}}|^{p}\mathbbm{1}_{\{\tau_{n}\geq t\}}\bigg{]}, \\[1em]
&= \mathbb{E}\bigg{[}|\eta_{x,t\land\tau_{n}}|^{p}\mathbbm{1}_{\{\tau_{n}<t\}}\bigg{]} + \mathbb{E}\bigg{[}|\eta_{x,t\land\tau_{n}}|^{p}\mathbbm{1}_{\{\tau_{n}\geq t\}}\bigg{]}, \\[1em]
&= \mathbb{E}\bigg{[}|\eta_{x,t\land\tau_{n}}|^{p}\bigg{]}. \label{BB6}
\end{align}
Therefore using inequalities (\ref{BB5}) - (\ref{BB6}) above we see that for all $n\in\naturals$ and for all $t\in[0,\infty)$ we have
\begin{align}
\mathbb{P}(\tau_{n}<t) \leq \frac{1}{n^{p}}K(x)e^{Dt}.
\end{align}
Hence for all $t\in[0,\infty)$ we have
\begin{align}
\lim_{n\to\infty}\mathbb{P}(\tau_{n}<t) = 0.
\end{align}
Now convergence in probability and the fact that $\{\tau_{n}\}_{n\in\naturals}$ is an increasing sequence impliy that almost surely $\lim_{n\to\infty}\tau_{n} = \infty$ hence the proof is complete. 
\end{proof}

\newpage

\section{Existence and Uniqueness} \label{EU}
Throughout this section let us assume that $\reals\ni p\geq 2$. \\[1em]
In this section we will learn that system (\ref{MainSystem}) admits a unique strong solution. We shall start by showing existence.
\begin{theorem}\label{Existence}
Stochastic system (\ref{MainSystem}) admits a strong solution.
\end{theorem}
\begin{proof}
Let us start by fixing some $\underline{\inta}<\alpha\in\indexint$. Now, according to the Theorem \ref{CauchySequenceTheorem} sequence $\{\Xi^{n}\}_{n\in\naturals}$ converges in $Z^{p}_{\alpha}$. Therefore, this proof can be completed by letting 
\begin{align}
\overbrace{\ \Xi \coloneqq \lim_{n\to\infty}\Xi^{n}\ }^{\text{in}\ Z_{\alpha}^{p}},
\end{align}
and showing that $\Xi\equiv\{\xi_{x}\}_{x\in\gamma}$ is also a strong solution of the system (\ref{MainSystem}). However because $\Xi$ in $Z^{p}_{\alpha}$ we see from the Definition \ref{StrongSol} that to complete the proof it only remains to show that for all $x\in\gamma$ and all $t\in\timeint$ we have
\begin{align}
\xi_{x,t} &= \zeta_{x} + \int_{0}^{t}\Phi_{x}(\xi_{x,s},\Xi_{s})ds + \int_{0}^{t}\Psi_{x}(\xi_{x,s},\Xi_{s})dW_{x}(s), \quad \mathbb{P}-a.s. \label{ExEq1}
\end{align}
Using our work in the previous section \ref{ODSC}, in particular using Theorem \ref{globaltheorem} we begin by defining a family of processes $H\coloneqq\{\eta_{x}\}_{x\in\gamma}$ such that for all $x\in\gamma$ and all $t\in\timeint$ we have
\begin{align}
\eta_{x,t} &= \zeta_{x} + \int_{0}^{t}\Phi_{x}(\eta_{x,s},\Xi_{s})ds + \int_{0}^{t}\Psi_{x}(\eta_{x,s},\Xi_{s})dW_{x}(s), \quad \mathbb{P}-a.s.
\end{align}
 Now, if $n\in\naturals$ then we also recall from the Theorem \ref{FiniteVolumeLemma} and the Definition, of the truncated system, (\ref{FinVolSystem}) that for all $x\in\gamma$ and all $t\in\timeint$ we have
\begin{equation}
\begin{rcases} 
\begin{alignedat}{2}
&\xi_{x,t}^{n} = \zeta_{x} + \int_{0}^{t}\Phi_{x}(\xi_{x,s}^{n},\Xi_{s}^{n})ds + \int_{0}^{t}\Psi_{x}(\xi_{x,s}^{n},\Xi_{s}^{n})dW_{x}(s) && \quad \forall x\in\Lambda_{n} \\
&\xi_{x,t}^{n} = \zeta_{x} &&\quad \forall x\not\in\Lambda_{n}
\end{alignedat}
\end{rcases}, \  \mathbb{P}-a.s.
\end{equation}
 Moreover convergence $\overbrace{\ \Xi = \lim_{n\to\infty}\Xi^{n}\ }^{\text{in}\ Z_{\alpha}^{p}}$ in particular implies that
\begin{align}
\adjustlimits{lim}_{n\to\infty}{sup}_{t\in\timeint}\mathbb{E}\bigg{[}
\ \sum_{x\in\gamma}e^{-\alpha|x|}|\xi^{n}_{x,t} - \xi_{x,t}|^{p}\bigg{]} = 0. \label{ExEq2}
\end{align}
Now, from equation (\ref{ExEq2}) above and Theorem \ref{LPInclusion} it follows that for all $x\in\Lambda_{n}$ and uniformly on $\timeint$ we have 
\begin{align}
\lim_{n\to\infty}\mathbb{E}\bigg{[}|\xi^{n}_{x,t} - \xi_{x,t}|\bigg{]} = 0.
\end{align}
Therefore, observing that $\Lambda_{n}\uparrow\gamma$ as $n\to\infty$ we see that in order to establish the equation (\ref{ExEq1}), and hence conclude the proof, it remains to show that for all $x\in\gamma$ and uniformly on $\timeint$ we have 
\begin{align}
\lim_{n\to\infty}\mathbb{E}\bigg{[}|\xi^{n}_{x,t} - \eta_{x,t}|\bigg{]} = 0. \label{ExEq3}
\end{align}
\begin{mdframed}
	\begin{remark} 
	Indeed, this will show that for all $x\in\gamma$ and uniformly on $\timeint$ we have
	\begin{align}
	\mathbb{E}\bigg{[}|\xi_{x,t} - \eta_{x,t}|\bigg{]} = 0, 
	\end{align}
	which, because $\gamma$ is countable, will establish that for all $x\in\gamma$ and uniformly on $\timeint$ we have
	\begin{align}
	\xi_{x,t} = \eta_{x,t},\ \mathbb{P}-a.s.,
	\end{align}
	hence establishing equation (\ref{ExEq1}) and compleating the proof.
	\end{remark}
\end{mdframed}
Now, let us fix an arbitrary $n\in\naturals$ and define for all $x\in\gamma$ and all $t\in\timeint$ the following processes
\begin{align}
\Phi_{x}^{n}(t)&\coloneqq\Phi_{x}(\xi^{n}_{x,t},\Xi^{n}_{t}) - \Phi_{x}(\eta_{x,t},\Xi_{t}),\\[1em]
\Psi_{x}^{n}(t)&\coloneqq\Psi_{x}(\xi^{n}_{x,t},\Xi^{n}_{t}) - \Psi_{x}(\eta_{x,t},\Xi_{t}),\\[1em]
\mathscr{X}^{n}_{x,t}&\coloneqq\xi^{n}_{x,t} - \eta_{x,t}.
\end{align}
Hence using Itô Lemma we begin observing that for all $x\in\gamma$ and all $t\in\timeint$ we have 
\begin{equation}
\begin{alignedat}{2}
|\mathscr{X}^{n}_{x,t}|^{p} = \int_{0}^{t}p(\mathscr{X}^{n}_{x,t})^{p-1}&\Phi_{x}^{n,m}(s)ds + \mathbin{\textcolor{white}{\int_{0}^{t}\frac{p(p-1)}{2}(\bar{\xi}^{n,m}_{x,s})^{p-2}(\Psi_{x}^{n,m}(s))^{2}ds}}&& \\[1em]
&+ \int_{0}^{t}\frac{p(p-1)}{2}(\mathscr{X}^{n}_{x,t})^{p-2}(\Psi_{x}^{n,m}(s))^{2}ds + \\[1em]
&&\!\!\!\!\!\!\!\!\!\!\!\!\!\!\!\!\!\!\!\!\!\!\!\!\!\!\!\!\!\!\!\!\!\!\!\!\!\!\!\!\!\!\!\!\!\!\!\!\!\!\!\!\!\!\!\!\!\! + \int_{0}^{t}p(\mathscr{X}^{n}_{x,t})^{p-1}\Psi_{x}^{n,m}(s)dW_{x}(s). \label{ExEq4}
\end{alignedat}
\end{equation}
Therefore, from inequalities (\ref{QQ1}) - (\ref{QQ2}) we can see that for all $x\in\gamma$ and all $t\in\timeint$ we have
\begin{align}
(\xi^{n}_{x,t} - \eta_{x,t})^{p-1}\Phi_{x}^{n}(t) \leq (b + 1)(\xi^{n}_{x,t} - \eta_{x,t})^{p} + \bar{a}^{2}n_{x}^{3}\sum_{y\in B_{x}}(\xi^{n}_{y,t} - \xi_{y,t})^{p}, \label{ExEq5}
\end{align}
and 
\begin{align}
(\xi^{n}_{x,t} - \eta_{x,t})^{p-2}\bigg{(}\Psi_{x}^{n}(t)\bigg{)}^{2} \leq 2M_{1}^{2}(\xi^{n}_{x,t} - \eta_{x,t})^{p} + 2M_{2}^{2}n_{x}^{4}\sum_{y\in B_{x}}(\xi^{n}_{y,t} - \xi_{y,t})^{p}. \label{ExEq6}
\end{align}

Now, because $B_{x}$ is finite for all $x\in\gamma$ it is clear from equation (\ref{ExEq2}) that 
\begin{align}
\mathbb{E}\bigg{[}\sum_{y\in B_{x}}(\xi^{n}_{y,t} - \xi_{y,t})^{p}\bigg{]},
\end{align}
can be made arbitrary small uniformly on $\timeint$ by taking $n\in\naturals$ sufficiently large. Therefore from inequalities (\ref{ExEq5}) and (\ref{ExEq6}) above we see that we can establish the following inequlities for all $t\in\timeint$
\begin{align}
\mathbb{E}\bigg{[}(\xi^{n}_{x,t} - \eta_{x,t})^{p-1}\Phi_{x}^{n}(t)\bigg{]}  \leq (b + 1)\mathbb{E}\bigg{[}(\xi^{n}_{x,t} - \eta_{x,t})^{p}\bigg{]}  + A_{x}^{n}, \label{ExEq7}
\end{align}
and 
\begin{align}
\mathbb{E}\bigg{[}(\xi^{n}_{x,t} - \eta_{x,t})^{p-2}\bigg{(}\Psi_{x}^{n}(t)\bigg{)}^{2}\bigg{]} \leq 2M_{1}^{2}\mathbb{E}\bigg{[}(\xi^{n}_{x,t} - \eta_{x,t})^{p}\bigg{]} + A_{x}^{n}. \label{ExEq8}
\end{align}
Where for all $x\in\gamma$ we have 
\begin{align}
A_{x}^{n} \coloneqq \max\{ \bar{a}^{2}n_{x}^{3}, \ 2M_{2}^{2}n_{x}^{4}\}\mathbb{E}\bigg{[}\sum_{y\in B_{x}}(\xi^{n}_{y,t} - \xi_{y,t})^{p}\bigg{]}.
\end{align}
Moreover $A_{x}^{n}\to 0$ uniformly on $\timeint$ as $n\to\infty$. Therefore using inequalities (\ref{ExEq7}) and (\ref{ExEq8}) above we can conclude from equation (\ref{ExEq4}) that  for all $x\in\gamma$ and all $t\in\timeint$ we have 
\begin{align}
\mathbb{E}\bigg{[}|\xi^{n}_{x,t} - \eta_{x,t}|^{p}\bigg{]} \leq C\int_{0}^{t}\mathbb{E}\bigg{[}|\xi^{n}_{x,s} - \eta_{x,s}|^{p}\bigg{]}ds +   \bar{A}_{x}^{n}, \label{ExEq9}
\end{align}
where 
\begin{align}
&C \coloneqq p^{2}(b + 1 + 2M_{1}^{2}), \\[1em]
&\bar{A}_{x}^{n} \coloneqq 2p^{2}TA_{x}^{n}.
\end{align}
Finally using Gronwall inequality \ref{ClassicGronTheorem} we see that for all $x\in\gamma$ and all $t\in\timeint$ we have
\begin{align}
\mathbb{E}\bigg{[}|\xi^{n}_{x,t} - \eta_{x,t}|^{p}\bigg{]} \leq A_{x}^{n}e^{CT},
\end{align}
which shows that for all $x\in\gamma$ and uniformly on $\timeint$
\begin{align}
\lim_{n\to\infty}\mathbb{E}\bigg{[}|\xi^{n}_{x,t} - \eta_{x,t}|^{p}\bigg{]} = 0.
\end{align}
Equation (\ref{ExEq3}) now follows via application of Therorem \ref{LPInclusion} hence the proof is complete. 
\end{proof}

In the following theorem we now address uniqueness. 

\begin{theorem} \label{Uniqueness}
Suppose $\zeta\in l^{p}_{\underline{\inta}}$ and $\underline{\inta}<\alpha\in\indexint$. Then stochastic system (\ref{MainSystem}) admits a unique strong solution $\Xi$ in $Z^{p}_{\alpha}$.	
\end{theorem} 

\begin{proof}
For contradiction, using Theorem \ref{Existence}, suppose that $\Xi^{1}$ and $\Xi^{2}$ are distinct strong solutions of the system (\ref{MainSystem}) in $Z^{p}_{\alpha}$. Now let us define a map $\bar{\Xi}\in Z^{p}_{\alpha}$ via the following formula 
\begin{align}
\bar{\Xi}_{t} \coloneqq \Xi^{1}_{t}-\Xi^{2}_{t}.
\end{align}
We see that almost surely we have 
\begin{align}
\bar{\xi}_{x,t} = \int_{0}^{t}\Phi_{x}(\xi^{1}_{x,s},\Xi^{1}_{s})-\Phi_{x}(\xi^{2}_{x,s},\Xi^{2}_{s})ds + \int_{0}^{t}\Psi_{x}(\xi^{1}_{x,s},\Xi^{1}_{s})-\Psi_{x}(\xi^{2}_{x,s},\Xi^{2}_{s})dW_{x}(s). 
\end{align}
Now as in the proof of Theorem \ref{CauchySequenceTheorem} we deduce, using Ito Lemma, that 
\begin{equation}
\begin{alignedat}{2}
|\bar{\xi}_{x,t}|^{p} = \int_{0}^{t}p(\bar{\xi}_{x,s})^{p-1}\Phi_{x}^{1,2}(s)ds &+ \int_{0}^{t}\frac{p(p-1)}{2}(\bar{\xi}_{x,s})^{p-2}(\Psi_{x}^{1,2}(s))^{2}ds \ + \\  
& \quad\quad\quad\quad\quad\quad\quad\quad\quad + \int_{0}^{t}p(\bar{\xi}_{x,s})^{p-1}\Psi_{x}^{1,2}(s)dW_{x}(s), 
\end{alignedat}
\end{equation}
where we have chosen for all $t\in\timeint$ to let
\begin{align}
\Phi_{x}^{1,2}(t)&\coloneqq\Phi_{x}(\xi^{1}_{x,t},\Xi^{1}_{t}) - \Phi_{x}(\xi^{2}_{x,t},\Xi^{2}_{t}), \\
\Psi_{x}^{1,2}(t)&\coloneqq\Psi_{x}(\xi^{1}_{x,t},\Xi^{1}_{t}) - \Psi_{x}(\xi^{2}_{x,t},\Xi^{2}_{t}).
\end{align}
Therefore we see that
\begin{align}
\mathbb{E}\bigg{[}|\bar{\xi}_{x,t}|^{p}\bigg{]} \leq B_{1}(p,b,c,M_{1})\int_{0}^{t}\mathbb{E}\bigg{[}|\bar{\xi}_{x,s}|^{p}\bigg{]}ds+B_{2}(x,p,M_{2})\sum_{y\in B_{x}}\int_{0}^{t}\mathbb{E}\bigg{[}|\bar{\xi}_{y,s}|^{p}\bigg{]}ds,  \label{Theorem7EqnA}
\end{align}
where
\begin{align}
&B_{1}(p,b,c,M_{1}) \coloneqq pb + \frac{p}{2} + M_{1}^{2}p(p-1), \\
&B_{2}(x,p,M_{2}) \coloneqq p\Tilde{a}^{2}_{x}+n_{x}^{4}(p+M_{2}^{2}p(p-1)).
\end{align}

Let us now use inequality (\ref{Theorem7EqnA}) to define a map $\kappa:\timeint\to\reals^{\gamma}$ via the following formula
\begin{align}
\kappa_{x}(t)\coloneqq\mathbb{E}\bigg{[}|\bar{\xi}_{x,t}|^{p}\bigg{]},
\end{align}
and deduce from inequality (\ref{Theorem7EqnA}) that 
\begin{align}
\kappa_{x}(t) \leq \sum_{y\in\gamma}Q_{x,y}\int_{0}^{t}\kappa_{y}(s)ds,
\end{align}
where for all $x,y\in\gamma$ we have
\begin{align}
Q_{x,y} = \begin{cases} 
B_{1}(p,b,c,M_{1}) + B_{2}(x,p,M_{2}), &x=y,\\ 
B_{2}(x,p,M_{2}), &0<|x-y|\leq p,\\
0, &|x-y|> p.
\end{cases}
\end{align}
Fixing $\underline{\inta}<\tilde{\alpha}\leq\alpha\in\indexint$ we can now deduce the following facts.
\begin{enumerate}
	\item Using Theorem \ref{Continuity Lemma}, we see that $\kappa\in C([0,T], l^{1}_{\tilde{\alpha}})$,
	\item From equation (\ref{Matrix1}) we see that there exists a constant $C$ such that $|Q_{x,y}|\leq Cn_{x}$. Therefore $Q$ is the Ovsiannikov operator on $\mathcal{L}^{1}$. 
\end{enumerate}
Therefore we can now use Theorem \ref{GronTheorem} and Corollary \ref{GronCorollary} to conclude that
\begin{align}
\sum_{x\in\gamma}e^{-\alpha|x|}\sup_{t\in\timeint}\kappa_{x}(t) \leq K(\tilde{\alpha}, \alpha)\sum_{x\in\gamma}e^{-\tilde{\alpha}|x|}|A_{x}|,
\end{align}
where $A_{x}$ is the zero sequence in $ l^{1}_{\tilde{\alpha}}$. Therefore we can establish that
\begin{align}
\sup_{t\in\timeint}\mathbb{E}\bigg{[}\sum_{x\in\gamma}e^{-\alpha|x|}|\bar{\xi}_{x,t}|^{p}\bigg{]} = 0. \label{Theorem7EqnD}
\end{align}
Hence
\begin{align}
||F^{1} - F^{2}||_{Z_{\alpha}^{p}} = 0,
\end{align}
and the proof is complete.
\end{proof}

\newpage 
\section{Complementary Theory}\label{CompSection}
Let us begin this section by recalling a couple of earlier statements and definitions. \\[1em]
Throughout this section assume the following. 
\begin{enumerate}
	\item We fix $T\in\reals^{+}$, define $\mathcal{T}\coloneqq [0,T]$ and work on a complete filtered probability space
	\begin{align}
	\mathbf{P}\coloneqq (\Omega,\mathcal{F},\mathbb{P},\mathbb{F}).
	\end{align}
	\begin{enumerate}
		\item Completeness of $\mathbf{P}$ implies that for all $t\in\timeint$, $\mathbf{P}_{t}\coloneqq (\Omega,\mathcal{F}_{t},\mathbb{P})$ is complete.
		\item Filtration $\mathbb{F} \coloneqq \{\mathcal{F}_{t}\}_{t\in\mathcal{T}}$ is assumed to be right continuous. That is for all $t\in\timeint$,
		\begin{align}
		\mathcal{F}_{t} = \bigcap_{n\in\naturals}\mathcal{F}_{t+\frac{1}{n}}.
		\end{align}
	\end{enumerate}
	\item We fix a measure space $\mathbf{M}\coloneqq (\mathcal{T},\mathscr{B}(\mathcal{T}),\mu)$, where $\mu$ is a Lebesgue measure and $\mathscr{B}(\mathcal{T})$ is a Borel $\sigma-$algebra.
	\item We now agree to work on a fixed product measure space
	\begin{align}
	\mathbf{M}\mathbf{P}\coloneqq (\overline{\Omega}\coloneqq\mathcal{T}\times\Omega,\overline{\mathcal{F}}\coloneqq\mathscr{B}(\mathcal{T})\times\mathcal{F},\overline{\mathbb{P}}\coloneqq\mu\times\mathbb{P}).
	\end{align}
	\item Given two measurable spaces $\mathbf{A}$ and $\mathbf{B}$ we denote by $\mathcal{M}(\mathbf{A}, \mathbf{B})$ the space of all measurable maps from $\mathbf{A}$ to $\mathbf{B}$. In particular, the following spaces will be frequently mentioned
	\begin{enumerate}
		\item $\mathbf{M}^{p}_{\inta}\coloneqq (l_{\inta}^{p},\mathscr{B}(l_{\inta}^{p}))$, where $l_{\inta}^{p}$ was introduced by Definition \ref{DefSequenceSpaceScale}. 
		\item $\mathbf{M}^{\reals}\coloneqq (\reals,\mathscr{B}(\reals))$,
	\end{enumerate}
\end{enumerate}

Here is how we will understand and denote stochastic processes in this section.
\begin{definition} \label{StochProcess1}
	Let $Y$ be a normed linear space and $\mathbf{Y}\coloneqq(Y,\mathcal{B})$ be a measurable space. \textbf{Stochastic process} is an element of $\mathcal{M}(\mathbf{M}\mathbf{P}, \mathbf{Y})$. In particular for all $t\in\timeint$ and all $\omega\in\Omega$ 
	\[\mathcal{M}(\mathbf{P}, \mathbf{Y}) \ni \xi_{t}(\cdot):\Omega\to Y,\]
	\[\mathcal{M}(\mathbf{M}, \mathbf{Y}) \ni \xi_{\cdot}(\omega):\timeint\to Y.\]
	For brevity we shall denote by $\mathcal{S}(\mathbf{Y})$ the set of all stochastic processes.
\end{definition}

Following Bnach spaces will be frequently used. 
\begin{definition} \label{LPSpaces1}
	Let $\mathscr{X}\coloneqq(X,\mathcal{A},\eta)$ be a measure space, $Y$ be a normed linear space, with norm denoted by $\|\cdot\|_{Y}$, and $\mathscr{Y}\coloneqq(Y,\mathcal{B})$ be a measurable space. For all $p\in\reals^{++}$ we define the following Banach spaces.
	\begin{align}
	\mathcal{L}^{p}(\mathscr{X},\mathscr{Y}) \coloneqq \left\{f:X\to Y \  \begin{tabular}{|l}
	\ $\|f\|_{\mathcal{L}^{p}(\mathscr{X},\mathscr{Y})}\coloneqq\left( \bigint_{\!\!\!\!X}\|f\|^{p}_{Y}d\eta \right)^{\frac{1}{p}} < \infty$. \\
	\ $f\in\mathcal{M}(\mathscr{X},\mathscr{Y})$.
	\end{tabular}\right\}
	\end{align}
\end{definition}
Let us also make the following definitions
\begin{align}
&S_{1}\coloneqq \bigg{\{}K\subset\mathcal{T}\times\Omega \ \bigg{|}\ K= (s,t]\times A\text{ where } s<t\in\timeint\land A\in\mathcal{F}_{s} \bigg{\}}, \\[1em]
&S_{2}\coloneqq\bigg{\{}K\subset\mathcal{T}\times\Omega \ \bigg{|}\ K= \{0\}\times A\text{ where } A\in\mathcal{F}_{0} \bigg{\}}, \\[1em]
&\mathcal{P}\coloneqq\sigma(S_{1}\cup S_{2}), \\[1em]
&\mathbb{L}\coloneqq \left\{\xi\in\mathcal{S}(\mathbf{M}^{\reals})  \  \begin{tabular}{|l}
\ $\text{trajectories of }\xi\text{ are left continuous, almost surely}$. \\
\ $\xi \ \text{is adapted to} \  \mathbb{F}$.
\end{tabular}\right\}. 
\end{align}
We note that $\mathcal{P}$ above is the smallest $\sigma-$algebra with respect to which all elements of $\mathbb{L}$ are measurable. 
\begin{enumerate}
\setcounter{enumi}{4}
	\item We now fix the following product measure space
	\begin{align}
	\overline{\mathbf{M}\mathbf{P}}\coloneqq (\overline{\Omega},\mathcal{P},\overline{\mathbb{P}}).
	\end{align}
\end{enumerate}

\newpage
\subsection{Expectation, Measurability and Related Inequalities}\label{ES}
Unless stated otherwise, information in this subsection is based on \cite{SB,RLS}. For brevity and convenience in this subsection we will be working with the following definition. 
\begin{definition}\label{NormalLPSpaces} \quad \\
Suppose $\mathbf{X}\coloneqq(X,\mathcal{A},\mu)$ is a measure spaces. For all $p\in\reals^{++}$ we make the following definition.
\begin{align}
&\mathcal{L}^{p}\coloneqq \mathcal{L}^{p}(\mathbf{P}, \mathbf{M}^{\reals}),\\[1em]
&\mathcal{L}^{p}(\mathbf{X})\coloneqq \mathcal{L}^{p}(\mathbf{X}, \mathbf{M}^{\reals}),\\[1em]	
&\mathcal{L}^{p}_{+}\coloneqq \{f\in\mathcal{L}^{p}|f\geq 0\ \text{almost surely}\},\\[1em]	
&\mathcal{L}^{p}_{+}(\mathbf{X})\coloneqq \{f\in\mathcal{L}^{p}(\mathbf{X})|f\geq 0\ \text{almost surely}\}.
\end{align}
\end{definition}

\begin{theorem}[Borel–Cantelli Theorem] \label{BorelCantelli} \quad \\
Let $\{A_{i}\}_{i\in\naturals}$ be a sequence of measurable subsets of $\Omega$. Then
\begin{align}
\sum_{i=0}^{\infty}\mathbb{P}(A_{i}) < \infty \implies \mathbb{P}\bigg{(}\bigcap_{j=0}^{\infty}\bigcup_{i=j}^{\infty}A_{i}\bigg{)} = 0.
\end{align}
\end{theorem}

\begin{corollary} \label{BorelCantelliCor1} \quad \\
Let $\{X_{i}\}_{i\in\naturals}$ be a sequence of real valued random variables and let $X$ be another real valued random variable. For all $\epsilon > 0$ and all $i\in\naturals$ define also the followig measurable sets
\[A_{i}(\epsilon) = \{ \omega\in\Omega \ | \ |X_{i}(\omega)-X(\omega)| \geq \epsilon\}.\] 
\end{corollary}
If for all $\epsilon > 0$
\begin{align}
\sum_{i=0}^{\infty}\mathbb{P}(A_{i}(\epsilon)) < \infty,
\end{align}
then $X_{i} \overset{a.s.}{\to} X$.

\begin{corollary} \label{BorelCantelliCor2} \quad \\
Let $\{X_{i}\}_{i\in\naturals}$ be a sequaence of real valued random variables. For all $\epsilon > 0$ and all $i\in\naturals$ define also the followig measurable sets
\[A_{i}(\epsilon) = \{ \omega\in\Omega \ | \ |X_{i+1}(\omega)-X_{i}(\omega)| \geq \epsilon\}.\] 
\end{corollary}
If for all $\epsilon > 0$
\begin{align}
\sum_{i=0}^{\infty}\mathbb{P}(A_{i}(\epsilon)) < \infty,
\end{align}
then almost surely $\{X_{i}\}_{i\in\naturals}$ is a Cauchy sequence.

\begin{theorem}[Minkowski’s inequality] \label{MinkowskiIneq} \quad \\
Let $p\in[1,\infty)$ and also let $f,g\in\mathcal{L}^{p}$. Then $f + g\in\mathcal{L}^{p}$ and 
\begin{align}
\|f + g\|_{\mathcal{L}^{p}} \leq \|f\|_{\mathcal{L}^{p}} + \|g\|_{\mathcal{L}^{p}}.
\end{align}
\end{theorem}

\begin{theorem} \label{RieszTheorem1} \quad \\
Let $p\in[1,\infty)$ and also let $\{f_{n}\}_{n\in\naturals} \in\mathcal{L}^{p}$. In addition, suppose that there exists $g\in\mathcal{L}^{p}$ such that $|f_{n}|<g$ for all $n\in\naturals$ and almost surely $\lim_{n\to\infty}f_{n}(\omega) = f(\omega)$. Then
\begin{align}
f\in\mathcal{L}^{p} \quad \text{and} \quad \lim_{n\to\infty}\|f_{n} - f\|_{\mathcal{L}^{p}} \ \to \ 0. 
\end{align}
\end{theorem}

\begin{theorem} \label{RieszTheorem2} \quad \\
 Let $p\in[1,\infty)$ and also let $\{f_{n}\}_{n\in\naturals}\in\mathcal{L}^{p}$. In addition, let $f\in\mathcal{L}^{p}$ and suppose that almost surely $\lim_{n\to\infty}f_{n}(\omega) = f(\omega)$. Then
\begin{align}
\lim_{n\to\infty}\|f_{n} - f\|_{\mathcal{L}^{p}} \ \to \ 0 \quad \iff \quad \lim_{n\to\infty}\|f_{n}\|_{\mathcal{L}^{p}} \ \to \ \|f\|_{\mathcal{L}^{p}}.
\end{align}
\end{theorem}

\begin{theorem} \label{ASsubsequence} \quad \\
	Let $\{f_{n}\}_{n\in\naturals}\in\mathcal{L}^{p}$ and $f\in\mathcal{L}^{p}$. Suppose that
	\begin{align}
	\lim_{n\to\infty}\|f_{n} - f\|_{\mathcal{L}^{p}} \ \to \ 0.
	\end{align}
	Then there exists a subsequence $\{f_{\sigma(n)}\}_{n\in\naturals}$ such that almost surely $\lim_{n\to\infty}f_{\sigma(n)}(\omega) = f(\omega)$. 
\end{theorem}
\begin{mdframed}
	\begin{remark} 
		Suppose that $\mathbf{X}\coloneqq(X,\mathcal{A},\mu)$  is any finite measure space. Then Theorem \ref{ASsubsequence} remains true if $\mathcal{L}^{p}$ is replaced by $\mathcal{L}^{p}(\mathbf{X})$.
	\end{remark}
\end{mdframed}
\begin{theorem}[Egoroff Theorem] \label{EgoroffT} \quad \\
Let $\{f_{n}\}_{n\in\naturals}$ and $f$ be measurable functions on a finite measure space $\mathbf{X}\coloneqq(X,\mathcal{A},\mu)$. Suppose that $\lim_{n\to\infty}f_{n}(x) = f(x)$ almost surely. Then given any $\delta>0$ there exists a measurable set $F$ such that $\mu(F)\leq\delta$ and $\lim_{n\to\infty}f_{n}(x) = f(x)$ uniformly on $X-F$.
\end{theorem}

\begin{theorem}\label{measurablelimit} \quad \\
	Let $\{f_{n}\}_{n\in\naturals}$ be a sequence of measurable functions on a measure space $\mathbf{X}\coloneqq(X,\mathcal{A},\mu)$. Suppose that we have a function $f:X\to\reals$ such that  $\lim_{n\to\infty}f_{n}(x) = f(x)$ almost surely. Then $f$ is measurable.
\end{theorem}

\begin{theorem} \label{Continuity Lemma} \quad \\
Let $p\in\reals^{++}$, $\widetilde{\Omega}\subset\Omega$ such that $\mathbb{P}(\widetilde{\Omega})=1$  and also let $f:\overline{\Omega} \to \reals$. Assume that $f$ has the following properties
\begin{enumerate}
	\item $f(t,\cdot)\in\mathcal{L}^{p}$ for all $t\in\timeint,$
	\item $f(\cdot,\omega)$ is continuous almost surely,
	\item$|f(t,\omega)| < g(\omega)$ for all $(t,\omega)\in\timeint\times\widetilde{\Omega}$ and some positive $g\in\mathcal{L}^{p}.$
\end{enumerate}
If we now define for all $t\in\timeint$ the following function
\begin{align}
h(t)\coloneqq\int_{\Omega}|f(t)|^{p}d\mathbb{P}.
\end{align}
Then $h$ is continuous.
\end{theorem}

\begin{theorem} \label{LPInclusion} \quad \\
Let $1\leq q \leq p$ be some real numbers and also let $f\in\mathcal{L}^{p}$. Then
\begin{align}
\mathcal{L}^{p} &\subset \mathcal{L}^{q}, \\
\|f\|_{\mathcal{L}^{q}} &\leq \|f\|_{\mathcal{L}^{p}}.
\end{align}
\end{theorem}
\begin{mdframed}
	\begin{remark} 
		Note that Theorem \ref{LPInclusion} remains true if $\mathcal{L}^{p}$ is replaced by $\mathcal{L}^{p}(\mathbf{MP},\mathbf{M}^{\reals})$. 
	\end{remark}
\end{mdframed}
\begin{theorem}  \label{Convergence in measure} \quad \\
Let $p\in[1,\infty)$ and also let $\{f_{n}\}_{n\in\naturals} \in\mathcal{L}^{p}$. In addition, let $f\in\mathcal{L}^{p}$. Then
\begin{enumerate}
	\item $\|f_{n} - f\|_{\mathcal{L}^{p}} \ \to \ 0 \ \text{as} \ n\to\infty \quad \implies \quad f_{n} \xrightarrow{\mathbb{P}} f \ \text{as} \ n\to\infty$,
	\item $\{f_{n}\}_{n\in\naturals} \ \text{is Cauchy in} \ \mathcal{L}^{p}  \quad \quad \implies \quad \{f_{n}\}_{n\in\naturals} \ \text{is Cauchy in} \ \mathbb{P}$.
\end{enumerate}
\end{theorem}

\begin{theorem}  \label{MeasurabilityCriteria} \quad \\
Let $\{f_{n}\}_{n\in\naturals}$ be a sequence of measurable functions from $\Omega$ to $\reals$ such that $\{f_{n}\}_{n\in\naturals}$ is Cauchy in $\mathbb{P}$. Then there exists a measurable function $f:\Omega\to\reals$ such that 
\begin{align}
f_{n} \xrightarrow{\mathbb{P}} f \ \text{as} \ n\to\infty,
\end{align}
and almost surely $f$ is unique. 
\end{theorem}

\begin{theorem}  \label{LpComparison} \quad \\
Let $f:\Omega\to\reals$ be a measurable function. If there exist $g\in\mathcal{L}^{p}$ such that $|f|\leq g$ then $f\in\mathcal{L}^{p}$.
\end{theorem}
\begin{mdframed}
	\begin{remark} 
		Note that Theorem \ref{LpComparison} remains true if $\mathcal{L}^{p}$ is replaced by $\mathcal{L}^{p}(\mathbf{MP},\mathbf{M}^{\reals})$. 
	\end{remark}
\end{mdframed}

\begin{theorem}[Grönwall Inequality] \label{ClassicGronTheorem} \quad \\
Suppose that $\alpha, \beta\in\reals$ are constants and $f\in \mathcal{L}^{1}(\mathbf{M},\mathbf{M}^{\reals})$ satisfies the following inequality
\begin{align}
f(t) \leq \alpha + \beta\int_{0}^{t}f(s)ds, \quad\forall t\in\timeint.
\end{align}
Then 
\begin{align}
f(t) \leq \alpha e^{\beta t}, \quad\forall t\in\timeint.
\end{align}
\end{theorem}

\begin{theorem}[Jensen Inequality]\label{JensenInequality} \quad \\
Let $\Lambda:\reals^{+}\to\reals^{+}$ and $V:\reals^{+}\to\reals^{+}$ be a concave and a convex function respectively. Suppose that $w,u\in\mathcal{L}^{1}_{+}$ and  $uw\in\mathcal{L}^{1}$. Then $\Lambda(u)w\in\mathcal{L}^{1}$ and
\begin{align}
\frac{\int_{\Omega}\Lambda(u)wd\mathbb{P}}{\int_{\Omega}wd\mathbb{P}} &\leq \Lambda\bigg{(}\frac{\int_{\Omega}uwd\mathbb{P}}{\int_{\Omega}wd\mathbb{P}} \bigg{)}, \\[1em]
V\bigg{(}\frac{\int_{\Omega}uwd\mathbb{P}}{\int_{\Omega}wd\mathbb{P}} \bigg{)} &\leq \frac{\int_{\Omega}V(u)wd\mathbb{P}}{\int_{\Omega}wd\mathbb{P}}.
\end{align} \\
\end{theorem}

\begin{theorem}[Fubini Theorem]\label{FubiniTheorem}  \quad \\
Let $\mathbf{X}\coloneqq(X,\mathcal{A},\mu)$ and $\mathbf{Y}\coloneqq(Y,\mathcal{B},\eta)$ be two $\sigma-\text{finite}$ measure spaces. Let
\begin{align}
\mathbf{XY}\coloneqq(X\times Y,\mathcal{A}\times \mathcal{B},\mu\times\eta)
\end{align}
be a product measure space and let $u:X\times Y\to\reals$ be $\mathcal{A}\times \mathcal{B}$ measurable. If at least one of the following integrals is finite
\begin{align}
\int_{X\times Y}|u|d(\mu\times\eta),\quad \int_{X}\int_{Y}|u|d\mu d\eta, \quad \int_{Y}\int_{X}|u|d\eta d\mu, 
\end{align}
then all three integrals are finite, $u\in\mathcal{L}^{1}(\mathbf{XY})$ and
\begin{enumerate}
	\item $x\to u(x,y)\in \mathcal{L}^{1}(\mathbf{X})$, $\eta-\text{almost everywhere}$,
	\item $y\to u(x,y)\in \mathcal{L}^{1}(\mathbf{Y})$, $\mu-\text{almost everywhere}$,
	\item $y\to \int_{X}u(x,y)d\mu(x) \in \mathcal{L}^{1}(\mathbf{Y})$, 
	\item $x\to \int_{Y}u(x,y)d\eta(y) \in \mathcal{L}^{1}(\mathbf{X})$,
	\item $\int_{X\times Y}|u|d(\mu\times\eta) =  \int_{X}\int_{Y}|u|d\mu d\eta = \int_{Y}\int_{X}|u|d\eta d\mu$.
\end{enumerate}
\end{theorem}
\begin{mdframed}
	\begin{remark} 
		Note that Theorem \ref{FubiniTheorem} above remains true for Banach space valued maps. For details see \cite{Lang}.
	\end{remark}
\end{mdframed}
\begin{mdframed}
	\begin{remark}
		It follows that if $f\in\mathcal{L}^{1}(\mathbf{MP},\mathbf{M}^{\reals})$ then by Theorem \ref{FubiniTheorem} function
		\begin{align}
		t\to\int_{\Omega}f(t)d\mathbb{P},
		\end{align}
		is $\mathscr{B}(\timeint)$ measurable. 
	\end{remark}
\end{mdframed}
\begin{theorem} \label{SupremumDence} \quad \\
	Let $X$ be a Banach space and let $Y$ be a dense subset of $X$. Moreover suppose that $X$ is compact and $f:X\to\reals$ is continuous. Then
	\begin{align}
	\sup\{f(y) \ | \ y \in Y\} = \sup\{f(x) \ | \ x \in X\}. 
	\end{align}
\end{theorem}
\begin{theorem}\label{powerinequalitytheorem}  \quad \\
For some $n\in\naturals$, suppose that $x_{k}\geq0$ for all $1\leq k \leq n$ and $p\geq 1$. Then
\begin{align}
\bigg{(}\sum_{k=1}^{n}x_{k}\bigg{)}^{p} \leq n^{p-1}\sum_{k=1}^{n}x_{k}^{p}. 
\end{align}
\end{theorem}
\begin{theorem}[Young Inequality]\label{Yinequality}  \quad \\
Suppose that $p,q\in(1,\infty)$ are such that $\frac{1}{p} + \frac{1}{q}=1$ and $x,y\in\reals^{+}$. Then 
\begin{align}
xy\leq\frac{x^{p}}{p}+\frac{y^{q}}{q}. \label{Yinequality1}
\end{align}
Moreover equality in (\ref{Yinequality1}) above occures if and only if $y=x^{p-1}$.
\end{theorem}

\newpage 
\subsection{Martingales and Wiener Process in $\reals$}\label{WienerProcessinR}
In this section we work with a real valued Wiener process $W$ defined on $\mathbf{M}\mathbf{P}$  and assume that a filtration $\mathbb{F} \coloneqq \{\mathcal{F}_{t}\}_{t\in\mathcal{T}}$ is suitably chosen so that the following properties are satisfied;
\begin{enumerate} 
	\item For all $t\in\mathcal{T}$, $W(t)$ is $\mathcal{F}_{t}$ measurable, 
	\item  For all $s \leq t\in\mathcal{T}$, $W(t) - W(s)$ is independent of $\mathcal{F}_{s}$. 
\end{enumerate}
Unless stated otherwise, information in this subsection is based on \cite{AR}. 
\begin{definition}\label{ItoSpaces}
	For all $p\in\reals^{++}$ we introduce the folowing spaces of stochastic processes. 
	\begin{align}
	L^{p}_{ad}\coloneqq \{ \xi\in \mathcal{L}^{p}(\mathbf{MP},\mathbf{M}^{\reals}) \ | \ \xi \ \text{is adapted to} \  \mathbb{F}.\},
	\end{align}
	and a space
	\begin{align}
	\mathcal{M}_{\mathbb{F}} \coloneqq \left\{\xi\in\mathcal{S}(\mathbf{M}^{\reals}) \  \begin{tabular}{|l}
	\ $\xi_{t}\in \mathcal{L}(\mathbf{P},\mathbf{M}^{\reals}) ,\ \forall t\in\timeint$. \\
	\ $\xi \ \text{is adapted to} \  \mathbb{F}$.
	\end{tabular}\right\}
	\end{align}
\end{definition}
\begin{definition}\label{MartingaleSpaces} \quad
\begin{enumerate}
	\item $\xi\in\mathcal{M}_{\mathbb{F}} $ is called a martingale with respect to $\mathbb{F}$ if almost surely for all $s \leq t\in\mathcal{T}$
			 \begin{align}
			 \mathbb{E}[\xi_{t}|\mathcal{F}_{s}] = \xi_{s}. 
			 \end{align}
	\item $\xi\in\mathcal{S}(\mathbf{M}^{\reals})$ is called square integrable if $\xi_{t}\in \mathcal{L}^{2}(\mathbf{P},\mathbf{M}^{\reals}) ,\ \forall t\in\timeint$.
	\item $\xi\in\mathcal{S}(\mathbf{M}^{\reals})$ is called predictable if $\xi\in\mathcal{M}(\overline{\mathbf{MP}}, \mathbf{M}^{\reals})$.
\end{enumerate}	
\end{definition}
\begin{theorem}\label{MartingaleDecomposition}
Let $\xi$ be a right continuous, square integrable martingale with left-hand limits. Then there is a unique decomposition	
\begin{align}
\xi^{2}_{t} = L_{t}+ A_{t},\ \forall t\in\timeint,
\end{align}
where $L$ is a right continuous martingale with left-hand limits and $A$ is a predictable, right continuous, and increasing process such that $A(0)=0$ and $A_{t}\in \mathcal{L}(\mathbf{P},\mathbf{M}^{\reals}) ,\ \forall t\in\timeint$.
\end{theorem}
\begin{mdframed}
	\begin{remark}
		Process $A$ found by Theorem \ref{MartingaleDecomposition} will be called a Meyer process in this text and the following abbreviation will be used
		\begin{align}
		\langle \xi \rangle_{t} = A_{t},\ \forall t\in\timeint. \label{Meyerprocess}
		\end{align}
		Moreover, one can show that
		\begin{align}
		\langle W \rangle_{t} = t,\ \forall t\in\timeint. \label{MeyerprocessWiener}
		\end{align}
	\end{remark}
\end{mdframed}
\begin{theorem}
Suppose that $\xi\in L^{2}_{ad}$ and define a stochastic process $X$ in the following way
\begin{align}
X_{t}\coloneqq\int_{0}^{t}\xi(s)dW(s).
\end{align}
Then 
\begin{enumerate}
	\item[(A)] $X$ is a martingale with respect to $\mathbb{F}$ and trajectories of $X$ are  almost surely continuous. 
\end{enumerate}
\begin{enumerate}
	\item[(B)] For all $t\in\timeint$
	\begin{enumerate}
		\item[(1)] $\mathbb{E}\bigg{[} \mathop{\mathlarger{\int}}_{\!\!\!0}^{t}\xi(s)dW(s) \bigg{]} = 0$,
		\item[(2)] $\mathbb{E}\bigg{[} \bigg{|}\mathop{\mathlarger{\int}}_{\!\!\!0}^{t}\xi(s)dW(s)\bigg{|}^{2} \bigg{]} = \mathop{\mathlarger{\int}}_{\!\!\!0}^{t}\mathbb{E}\bigg{[}|\xi(s)|^{2}\bigg{]}ds$,
		\item[(3)] $\langle X \rangle_{t} = \mathop{\mathlarger{\int}}_{\!\!\!0}^{t}|\xi(s)|^{2}d \langle W \rangle_{s}$.
	\end{enumerate} 
\end{enumerate}
\end{theorem}
Following theorem is a usefull result from \cite{MR}. 
\begin{theorem}[Burkholder, Davis and Gundy Inequality]\label{BDGInequality} \quad \\
Let $X$ be a continuous martingale. Then for all $t\in\timeint$ and all $p\in(0,\infty)$
\begin{align}
\mathbb{E}\bigg{[}\sup\bigg{\{}|X_{s}|^{p}\ \bigg{|}\ 0\leq s\leq t\bigg{\}}\bigg{]} = \mathbb{E}\bigg{[} \bigg{(}\langle X \rangle_{t}\bigg{)}^{\frac{p}{2}} \bigg{]}.
\end{align}
\end{theorem}
\begin{definition} \label{ItoProcess1}
	Suppose that $f\in L^{2}_{ad}$, $g\in L^{1}_{ad}$ and let $\xi_{0}$ be a $\mathcal{F}_{0}$ measurable random variable. An Itô process is a real valued stochastic process $\xi$ satisfying
	\begin{align}
	\xi_{t} &= \xi_{0} + \int_{0}^{t}g(s)ds + \int_{0}^{t}f(s)dW(s),\forall t\in\timeint. \label{ItoProcessEqn1}
	\end{align}
\end{definition}
\begin{theorem}[Itô Lemma] \label{ItoLemma1}
	Let $\xi$ be an Itô process satisfying equation (\ref{ItoProcessEqn1}) above and suppose that $\theta:\reals^{2}\to\reals$ is a continuous function such that all $\frac{\partial\theta}{\partial t}$, $\frac{\partial\theta}{\partial x}$ and $\frac{\partial^{2}\theta}{\partial x^{2}}$ are continuous functions from $\reals^{2}$ to $\reals$. Then $\theta\circ\xi$ is an Itô process satisfying
	\begin{align}
	\theta(t,\xi_{t}) = \theta(0,\xi_{0}) + \int_{0}^{t}\mathcal{K}(s,\xi_{s})ds + \int_{0}^{t}\frac{\partial\theta}{\partial x}(s,\xi_{s})f(s)dW(s),\ \forall t\in\timeint, \label{ItoLemmaEqn1}
	\end{align}
	where 
	\begin{align}
	\mathcal{K}(t,\xi_{t})\coloneqq \frac{\partial\theta}{\partial t}(t,\xi_{t}) + \frac{\partial\theta}{\partial x}(t,\xi_{t})g(t) + \frac{1}{2}\frac{\partial^{2}\theta}{\partial x^{2}}(t,\xi_{t})f^{2}(t),\ \forall t\in\timeint.
	\end{align}
\end{theorem}

\newpage 
\subsection{Deterministic Ovsjannikov Equation}\label{DeterministicEqn}
Unless stated otherwise, information in this subsection is based on \cite{AD,AD&DF,LVO,LVO2}. \\[1em]
In this subsection we would like to address the problem of finding a unique continuous infinite time solution f satisfying the following integral equation
\begin{align}
f(t) = x_{\underline{\inta}} + \int_{0}^{t}F(f(s))ds, \label{ABGoal}
\end{align}
where we let $\mathbf{X}\coloneqq\{X_{\inta}\}_{\inta\in\indexint}$ be a suitable scale of Banach spaces, $x_{\underline{\inta}}\in X_{\underline{\inta}}$ and $F\in\mathcal{O}(\mathbf{X},L,q)$ be an Ovsjannikov map on $\mathbf{X}$. The main result of this appendix, that is existence and uniqueness of $f$, is summarised in the Theorem \ref{ABEUtheorem} bellow.  \\ \\
We will now show how the proof of Theorem \ref{ABEUtheorem} can be obtained. We start with a result that will be needed later on.
\begin{lemma} \label{ABSeriesLemma}
	Supose that $A,B\in\reals^{+}$, $p\in\naturals$ and $q\in[0,\frac{1}{p})$. Then
	\begin{align}
	\sum_{n=0}^{\infty} \frac{\sqrt[\leftroot{-2}\uproot{2}p]{A^{n}\ }}{B^{qn}} \frac{n^{qn}}{\sqrt[\leftroot{-2}\uproot{2}p]{n!\ }} <\infty. \label{ABMainSeries}
	\end{align}
\end{lemma}

\begin{proof}
	We will consider the following two cases separately.
	\begin{enumerate}
		\item $q\in (0,\frac{1}{2})$. \\
		Then by analyzing ratio of terms of series (\ref{ABMainSeries}) we get
		\begin{align}
		\frac{\sqrt[\leftroot{-2}\uproot{2}p]{A^{n+1}}}{B^{q(n+1)}} \frac{(n+1)^{q(n+1)}}{\sqrt[\leftroot{-2}\uproot{2}p]{(n+1)!}} \bigg{/} \frac{\sqrt[\leftroot{-2}\uproot{2}p]{A^{n}}}{B^{qn}} \frac{n^{qn}}{\sqrt[\leftroot{-2}\uproot{2}p]{n!}} &=\frac{\sqrt[\leftroot{-2}\uproot{2}p]{A}}{B^{q}}(n+1)^{qn+q-\frac{1}{p}}\frac{1}{n^{qn}}, \\[1em]
		&=\frac{\sqrt[\leftroot{-2}\uproot{2}p]{A}}{B^{q}}\frac{1}{(n+1)^{\frac{1}{p}-q}}\bigg(1+\frac{1}{n}\bigg)^{qn}_{.}
		\end{align}
	Now since 
		\begin{align}
		\lim_{n\to \infty}\frac{\sqrt[\leftroot{-2}\uproot{2}p]{A}}{B^{q}}\frac{1}{(n+1)^{\frac{1}{p}-q}}\bigg(1+\frac{1}{n}\bigg)^{qn} &= \frac{\sqrt[\leftroot{-2}\uproot{2}p]{A}}{B^{q}}\bigg{(}\lim_{n\to \infty}\frac{1}{(n+1)^{\frac{1}{p}-q}}\bigg{)}\bigg{(}\lim_{n\to \infty}\bigg(1+\frac{1}{n}\bigg)^{qn}\bigg{)}, \nonumber \\[1em]
		&=\frac{\sqrt[\leftroot{-2}\uproot{2}p]{A}}{B^{q}}(0)(e^{q}), \nonumber \\[1em]
		&=0.
		\end{align}
	Hence we conclude by ratio test that when $q\in (0,\frac{1}{2})$ series (\ref{ABMainSeries}) converges. 
	
	\item $q=0$. \\
	Then series (\ref{ABMainSeries}) reduces to $\sum_{n=0}^{\infty}\sqrt[\leftroot{-2}\uproot{2}p]{\frac{A^{n}}{n!}\ }$. By ratio test, as shown above,  it is clear that when $q=0$ series (\ref{ABMainSeries}) also converges hence the proof is complete.
\end{enumerate}
\end{proof}
Continuing, we fix some $T\in\reals^{+}$ and introduce a family $\mathbf{Z}\coloneqq\{Z_{\inta}\}_{\inta\in\indexint}$ where
$Z_{\inta}$ is the classical space of continuous $X_{\inta}$ valued maps. That is for all $\inta\in\indexint$ we define
\begin{align}
Z_{\inta} \coloneqq \mathcal{C}([0,T],X_{\inta}).
\end{align} 
Now, for all $\alpha<\beta\in\indexint$ and $f\in Z_{\alpha}$ the following consequences are immediate. 
\begin{enumerate}
	\item $\mathbf{Z}$ is a family of Banach spaces,
	\item $Z_{\alpha} \prec Z_{\beta}$, \hfill\refstepcounter{equation}(\theequation) \label{listZ}
	\item $||f||_{Z_{\beta}} \leq ||f||_{Z_{\alpha}}$. \label{2ndNormInequality}
\end{enumerate}
Therefore, from the list (\ref{listZ}) above we can conclude, using the Definition \ref{Defscale}, that $\mathbf{Z}$ is the scale. Continuing, we let
\begin{align}
\overline{\mathbf{Z}}\coloneqq\bigcup_{\inta\in(\underline{\inta},\overline{\inta})}Z_{\inta},
\end{align}
and define a map $\mathcal{I}:\overline{\mathbf{Z}}\to Z_{\overline{\inta}}$ by letting for all $t\in [0,T]$ and all $f\in\overline{\mathbf{Z}}$
\begin{align}
\mathcal{I}(f)(t) \coloneqq x_{\underline{\inta}}+ \int_{0}^{t}F(f(s))ds. \label{IABntegralMap}
\end{align}
The following result can now be proved. 
\begin{theorem} \label{ABIntmap}
$\mathcal{I}\in\mathcal{O}(\mathbf{X},LT,q)$. 
\end{theorem}
\begin{proof}
Fix $\alpha<\beta\in\indexint$, $f,g\in Z_{\alpha}$ and $t\in [0,T]$. We now check that the integral map $\mathcal{I}$ satisfies the Definition \ref{Defovsop}. We begin by using the definition of Bochner integral and the fact that $F\in\mathcal{O}(\mathbf{X},L,q)$ to conclude that $\mathcal{I}|_{Z_{\alpha}}:Z_{\alpha}\to Z_{\beta}$. Moreover we see that
\begin{align}
||\mathcal{I}(f)(t) - \mathcal{I}(g)(t)||_{X_{\beta}} &\leq \int_{0}^{t}||F(f(s)) - F(g(s))||_{X_{\beta}}ds, \\
&\leq \frac{L}{(\beta - \alpha)^{q}} \int_{0}^{t}||f(s) - g(s)||_{X_{\alpha}}ds, \label{AB2ndLCondition1} \\  
&\leq \frac{L}{(\beta - \alpha)^{q}} \int_{0}^{t}||f - g||_{Z_{\alpha}}ds. \label{AB2ndLCondition}
\end{align}
Therefore we see that 
\begin{align}
||\mathcal{I}(f) - \mathcal{I}(g)||_{Z_{\beta}} &\leq \frac{L}{(\beta - \alpha)^{q}} \int_{0}^{T}||f - g||_{Z_{\alpha}}ds, \\
&\leq \frac{LT}{(\beta - \alpha)^{q}} ||f - g||_{Z_{\alpha}}, \label{AB3rdLCondition}
\end{align}	
hence the proof is complete.
\end{proof}
We now would like to define something called an itterated or a composite map. That is for all $n\in\naturals$ we define
\begin{align}
 \mathcal{I}^{n} \coloneqq \overbrace{\mathcal{I}\circ\mathcal{I}\circ\cdots\circ\mathcal{I}}^{n\text{\ times}}, \label{ABcomposite}
\end{align}
and let $\mathcal{T}^{0}$ be the identity map from $Z_{\underline{\inta}}$ to $Z_{\underline{\inta}}$. Letting 
\begin{align}
\underline{\mathbf{Z}}\coloneqq\bigcap_{\inta\in(\underline{\inta},\overline{\inta})}Z_{\inta},
\end{align}
our next result shows that for all $n\in\naturals$ the composite map $\mathcal{I}^{n}$ is well defined. 
\begin{theorem} \label{ABcompositetheorem0}
For all $n\in\naturals^{0}$
\begin{align}
\mathcal{I}^{n}:Z_{\underline{\inta}}\to\underline{\mathbf{Z}}. \label{ABcompositetheorem}
\end{align}
\end{theorem}
\begin{proof}
We prove this statement by induction. For $n=0$ the statement (\ref{ABcompositetheorem}) is trivially true because $Z_{\underline{\inta}}\subset\underline{\mathbf{Z}}$. Now suppose that induction hypothesis holds for some $n\geq 0$. Fix arbitrary $\inta\in(\underline{\inta},\overline{\inta})$ and $p\in(\underline{\inta},\inta)$. Observe that induction hypothesis implies that $\mathcal{I}^{n}:Z_{\underline{\inta}} \to Z_{p}$. However because $\mathcal{I}\in\mathcal{O}(\mathbf{X},LT,q)$ we know that $\mathcal{I}|_{Z_{p}}:Z_{p} \to Z_{\inta}$ hence by composition $\mathcal{I}\circ\mathcal{I}^{n}$ it follows that $\mathcal{I}^{n+1}:Z_{\underline{\inta}} \to Z_{\inta}$ and since $\inta\in(\underline{\inta},\overline{\inta})$ is arbitrary the proof is complete. 
\end{proof}
\begin{mdframed}
	\begin{remark}
	Observe that Theorem \ref{ABcompositetheorem0} shows that if $f\in Z_{\underline{\inta}}$ then the sequence $\{\mathcal{I}^{n}(f)\}_{n=0}^{\infty}$ belogs to $Z_{\inta}$ for all $\inta\in(\underline{\inta},\overline{\inta})$. 
	\end{remark}
\end{mdframed}
Let us now, for a moment, consider some fixed $t_{0}\in[0,T]$, $\alpha<\beta\in(\underline{\inta},\overline{\inta})$ and $f\in Z_{\underline{\inta}}$. Moreover let us consider arbitrary $n\in\naturals$ and a partition $\{\psi_{i}\}_{i=0}^{n}$ of $[\alpha,\beta]$ into $n$ intervals of equal length. That is $\psi_{0}=\alpha$, $\psi_{n}=\beta$ and $\psi_{i+1}-\psi_{i} = \frac{b-a}{n}$ for all $0\leq i \leq n$. Letting
\begin{align}
K_{n}^{n+1}(t) = \mathcal{I}^{n}(f)(t_{0}) - \mathcal{I}^{n+1}(f)(t_{0}),\ \forall t\in [0,t_{0}],
\end{align}
we see from Theorem \ref{ABIntmap} and \ref{ABcompositetheorem0} that
\begin{align}
||K_{n}^{n+1}(t_{0})||_{X_{\psi_{n}}} &\leq \frac{L}{(\psi_{n} - \psi_{n-1})^{q}}\int_{0}^{t_{0}}||K_{n-1}^{n}(t_{1})||_{X_{\psi_{n-1}}}dt_{1}, \nonumber \\[1em] 
&\leq \frac{L}{(\psi_{n} - \psi_{n-1})^{q}}\frac{L}{(\psi_{n-1} - \psi_{n-2})^{q}}\int_{0}^{t_{0}}\int_{0}^{t_{1}}||K_{n-2}^{n-1}(t_{2})||_{X_{\psi_{n-2}}}dt_{2}dt_{1}, \nonumber \\[1em] 
& \leq L^{n}\bigg{(}\frac{\beta-\alpha}{n}\bigg{)}^{-qn} \int_{0}^{t_{0}}\int_{0}^{t_{1}}\cdots\int_{0}^{t_{n-1}}||K_{0}^{1}(t_{n})||_{X_{\psi_{0}}}dt_{n}dt_{n-1}\cdots dt_{1}, \label{ABEquationForLL} \\[1em] 
&\leq \frac{L^{n}}{(\beta-\alpha)^{qn}}n^{qn}||K_{0}^{1}||_{Z_{\psi_{0}}}\int_{0}^{t_{0}}\int_{0}^{t_{1}}\cdots\int_{0}^{t_{n-1}}\ dt_{n}dt_{n-1}\cdots dt_{1}, \nonumber \\[1em]
&\leq \frac{L^{n}t_{0}^{n}}{(\beta-\alpha)^{qn}}\frac{n^{qn}}{n!}||K_{0}^{1}||_{Z_{\psi_{0}}}. \nonumber
\end{align}
Hence, defining recursively $\mathcal{H}^{n}:\mathcal{C}([0,T],\reals)\to\mathcal{C}([0,T],\reals)$ for all $n\in\naturals^{0}$ via formula
\begin{align} \label{ABHmap}
\mathcal{H}^{n}(t,f) \coloneqq \begin{cases} 
\begin{tabular}{l|l}
$f(t)$ \ & \ $t\in [0,T]\land n=0$, \\
$\int_{0}^{t}f(s)ds$ \ & \ $t\in [0,T]\land n=1$, \\
$\int_{0}^{t}\mathcal{H}^{n-1}(s,f)ds$ \ & \ $t\in [0,T]\land n>1$.
\end{tabular}
\end{cases}
\end{align}
we see from inequalities (\ref{ABEquationForLL}) that the following result can be formulateed and proved. 
\begin{theorem}\label{ABprecauchytheorem}
Suppose $\alpha<\beta\in(\underline{\inta},\overline{\inta})$ and $f,g\in Z_{\underline{\inta}}$. Then for all $n\in\naturals$
\begin{align}
||\mathcal{I}^{n}(f) - \mathcal{I}^{n+1}(g)||_{Z_{\beta}} \leq \frac{L^{n}T^{n}}{(\beta-\alpha)^{qn}}\frac{n^{qn}}{n!}||f-\mathcal{I}(g)||_{Z_{\alpha}}. \label{ABprecauchy1}
\end{align}
\end{theorem}
\begin{proof}
Fixing $t\in [0,T]$ we prove by induction that
\begin{align}
||\mathcal{I}^{n}(f)(t) - \mathcal{I}^{n+1}(g)(t)||_{X_{\beta}} \leq \frac{L^{n}}{(\beta-\alpha)^{qn}}n^{qn}\mathcal{H}^{n}(t,||f-\mathcal{I}(g)||_{X_{\alpha}}), \label{ABprecauchy2}
\end{align}
from where inequality (\ref{ABprecauchy1}) follows directly. Clearly case $n=1$ follows immediately from the Theorem \ref{ABIntmap}. Precisely spaking inequality (\ref{AB2ndLCondition1}) shows that the induction hypothesis holds for $n=1$. Now, suppose that the induction hypothesis holds for some $n\geq 1$. Chosing $\psi\in(\alpha,\beta)$ such that $\beta-\psi= \frac{\beta-\alpha}{n+1}$ we see, using Theorem \ref{ABIntmap}, that
\begin{align}
||\mathcal{I}^{n+1}(f)(t) - \mathcal{I}^{n+2}(g)(t)||_{X_{\beta}} \leq \frac{L}{(\beta-\psi)}\int_{0}^{t}||\mathcal{I}^{n}(f)(s) - \mathcal{I}^{n+1}(g)(s)||_{X_{\psi}}ds.
\end{align}
Hence letting
\begin{align}
\mathbf{A} \coloneqq ||f-\mathcal{I}(g)||_{X_{\alpha}},
\end{align}
and applying the induction hypothesis we get
\begin{align}
||\mathcal{I}^{n+1}(f)(t) - \mathcal{I}^{n+2}(g)(t)||_{X_{\beta}} &\leq \frac{L}{(\beta-\psi)^{q}} \frac{L^{n}}{(\psi-\alpha)^{qn}}n^{qn}\int_{0}^{t}\mathcal{H}^{n}(s,\mathbf{A})ds, \nonumber \\[1em]
&\leq \frac{L^{n+1}}{(\beta-\psi)^{q}(\psi-\alpha)^{qn}}n^{qn}\mathcal{H}^{n+1}(t,\mathbf{A}), \nonumber \\[1em]
&\leq L^{n+1}\bigg{(\frac{\beta-\alpha}{n+1}}\bigg{)}^{-q}\bigg{(\frac{n(\beta-\alpha)}{n+1}}\bigg{)}^{-qn}n^{qn}\mathcal{H}^{n+1}(t,\mathbf{A}), \\[1em]
&\leq \frac{L^{n+1}}{(\beta-\alpha)^{q(n+1)}}\frac{(n+1)^{q(n+1)}}{n^{qn}}n^{qn}\mathcal{H}^{n+1}(t,\mathbf{A}), \nonumber \\[1em]
&\leq \frac{L^{n+1}}{(\beta-\alpha)^{q(n+1)}}(n+1)^{q(n+1)}\mathcal{H}^{n+1}(t,\mathbf{A}). \nonumber
\end{align}
Hence
\begin{align}
||\mathcal{I}^{n+1}(f)(t) - \mathcal{I}^{n+2}(g)(t)||_{X_{\beta}} \leq \frac{L^{n+1}}{(\beta-\alpha)^{q(n+1)}}(n+1)^{q(n+1)}\mathcal{H}^{n+1}(t,||f-\mathcal{I}(g)||_{X_{\alpha}}),
\end{align}
and the proof is complete.
\end{proof}
\begin{mdframed}
	\begin{remark}
		It is clear from the definition of the composite map $\mathcal{I}^{n}$ that the Theorem \ref{ABprecauchytheorem} is trivially true for $n=0$. Moreover it is essential that $\alpha\in(\underline{\inta},\overline{\inta})$ because it is possible that $\mathcal{I}(f)$ does not belogn to $Z_{\underline{\inta}}$.
	\end{remark}
\end{mdframed}

Theorem \ref{ABprecauchytheorem} puts us in a position to prove the following. 
\begin{theorem} \label{ABfixpointtheorem}
Suppose that $q<1$ and $F\in\mathcal{O}(\mathbf{X},L,q)$. Then there exists a unique element $\phi\in\underline{Z}$ such that $\mathcal{I}(\phi)=\phi$. Moreover if $\inta\in(\underline{\inta},\overline{\inta})$ and $f\in Z_{\underline{\inta}}$ then
\begin{align}
\overbrace{\ \lim_{n\to\infty}\mathcal{I}^{n}(f)\ }^{\text{in} \ Z_{\inta}} = \phi.
\end{align}
\end{theorem}
\begin{proof}
Fix $f\in Z_{\underline{\inta}}$ and $\inta\in(\underline{\inta},\overline{\inta})$. Fix also an arbitrary $\gamma\in(\underline{\inta},\inta)$ and using theorem \ref{ABprecauchytheorem} observe that for all $m\geq n\in\naturals$ we have
\begin{align}
||\mathcal{I}^{n}(f) - \mathcal{I}^{m}(f)||_{Z_{\inta}} &\leq\sum_{k=n}^{m-1}||\mathcal{I}^{k}(f) - \mathcal{I}^{k+1}(f)||_{Z_{\gamma}}, \\[1em]
&\leq\sum_{k=n}^{m-1}\frac{L^{k}T^{k}}{(\inta-\gamma)^{qk}}\frac{n^{qk}}{k!}\ ||f-\mathcal{I}(f)||_{Z_{\gamma}}, \\[1em]
&\leq \sum_{k=n}^{\infty}\frac{L^{k}T^{k}}{(\inta-\gamma)^{qk}}\frac{n^{qk}}{k!}\ ||f-\mathcal{I}(f)||_{Z_{\gamma}}. \label{ABfixpointtheorem1}
\end{align}
According to Theorem \ref{ABSeriesLemma} the right hadn side of inequality (\ref{ABfixpointtheorem1}) above is a remainder of a convergent series. Therefore we conclude that sequence $\{\mathcal{I}^{n}(f)\}n\in\naturals$ is Cauchy in $Z_{\inta}$. Since $\inta$ is arbitrary, let us now consider $\alpha<\beta\in(\underline{\inta},\overline{\inta})$ and 
\begin{align}
&\overbrace{\ \lim_{n\to\infty}\mathcal{I}^{n}(f)\ }^{\text{in} \ Z_{\alpha}} = \phi_{\alpha}. \\
&\overbrace{\ \lim_{n\to\infty}\mathcal{I}^{n}(f)\ }^{\text{in} \ Z_{\beta}} = \phi_{\beta}.
\end{align}
Because $Z_{\alpha}\prec Z_{\beta}$ we see that
\begin{align}
\|\phi_{\beta}-\phi_{\alpha}\|_{Z_{\beta}} &\leq \|\phi_{\beta} -\mathcal{I}^{n}(f)\|_{Z_{\beta}} + \|\mathcal{I}^{n}(f)-\phi_{\alpha}\|_{Z_{\beta}}, \\
&\leq \|\phi_{\beta} -\mathcal{I}^{n}(f)\|_{Z_{\beta}} + \|\mathcal{I}^{n}(f)-\phi_{\alpha}\|_{Z_{\alpha}},
\end{align}
which shows that $\phi_{\beta} = \phi_{\alpha}$. Therefore defining
\begin{align}
\phi_{\alpha} \eqqcolon \phi \coloneqq \phi_{\beta},
\end{align}
we see that $\phi\in \underline{Z}$ and 
\begin{align}
\overbrace{\ \lim_{n\to\infty}\mathcal{I}^{n}(f)\ }^{\text{in} \ Z_{\inta}} = \phi.
\end{align}
Now, from Theorem \ref{ABIntmap} it follows that $\mathcal{I}$ is a continuous map from $Z_{\inta}$ to $Z_{\overline{\inta}}$. Hence we see that
\begin{align}
\mathcal{I}^{n+1}(f) &\to \phi\ \text{as}\ n\to\infty, \\
\mathcal{I}^{n+1}(f) = \mathcal{I}(\mathcal{I}^{n}(f)) &\to \mathcal{I}(\phi)\ \text{as}\ n\to\infty,
\end{align}
which shows that $\mathcal{I}(\phi)=\phi$. Finally suppose that there exists $\psi\in \underline{Z}$ such that $\psi \not= \phi$ and $\mathcal{I}(\psi)=\psi$. In this case it is clear that
\begin{align}
||\mathcal{I}^{n}(\phi) - \mathcal{I}^{n+1}(\psi)||_{Z_{\inta}}  = ||\phi - \psi||_{Z_{\inta}}. 
\end{align}
However from Theorem \ref{ABprecauchytheorem} we can infer that
\begin{align}
||\mathcal{I}^{n}(\phi) - \mathcal{I}^{n+1}(\psi)||_{Z_{\inta}}  &\leq \frac{L^{n}T^{n}}{(\beta-\alpha)^{qn}}\frac{n^{qn}}{n!}||\phi-\mathcal{I}(\psi)||_{Z_{\alpha}}, \\
&=\frac{L^{n}T^{n}}{(\beta-\alpha)^{qn}}\frac{n^{qn}}{n!}||\phi-\psi||_{Z_{\alpha}}. \label{ABfixpointtheorem2}
\end{align}
Since, by Theorem \ref{ABSeriesLemma},  the right hand side of inequality (\ref{ABfixpointtheorem2}) tends to zero we conclude that $||\phi - \psi||_{Z_{\inta}}=0$. Therefore $\phi$ is unique and the proof is complete.
\end{proof}
We now formulate and prove the main result of this appendix. 
\begin{theorem}\label{ABEUtheorem}
Suppose $x_{\underline{\inta}} \in X_{\underline{\inta}}$, $q<1$ and $F\in\mathcal{O}(\mathbf{X},L,q)$ are fixed. Then there exist a unique map $f:[0,T]\to\underline{X}$ such that if $\inta\in(\underline{\inta}, \overline{\inta})$ then $f:[0,T]\to X_{\inta}$ is continuous and 
\[
f(t) = x_{\underline{\inta}}+ \int_{0}^{t}F(f(s))ds,\ t\in [0,T].
\]
\end{theorem}
\begin{proof}
This result folows directly from Theorem \ref{ABfixpointtheorem} above by letting $f\coloneqq\phi$. 
\end{proof}
\begin{mdframed}
	\begin{remark}
	For $q=1$ current method can be used to prove Theorem \ref{ABEUtheorem}  by introducing a suitable upper bound on $T$. 
	\end{remark}
\end{mdframed}
The final result of this appendix is a usefull norm estimate. To prove this final result we now make two preliminary observations.\\ \\ First, suppose that $\alpha<\beta\in\indexint$ and $x\in X_{\alpha}$. Then we can see that
\begin{align}
\|F(x)\|_{X_{\beta}} &= \|F(x)+F(0)-F(0)\|_{X_{\beta}},  \\
&\leq \|F(x)-F(0)\|_{X_{\beta}} + \|F(0)\|_{X_{\beta}}, \\
&\leq \frac{L}{(\beta-\alpha)^{q}}\|x\|_{X_{\alpha}} + \|F(0)\|_{X_{\beta}}, \\
&\leq \frac{L}{(\beta-\alpha)^{q}}\bigg{(}P+ \|x\|_{X_{\alpha}}\bigg{)},
\end{align}
where
\begin{align}
P\coloneqq\frac{\|F(0)\|_{X_{\beta}}(\beta-\alpha)^{q}}{L}. \label{ABdefP}
\end{align}
Second, suppose $\inta\in(\underline{\inta},\overline{\inta})$ and $x_{\underline{\inta}}\in X_{\underline{\inta}}$. Moreover consider a partition $\{\psi_{i}\}_{i=0}^{n+1}$ of $[\underline{\inta},\alpha]$ into $n+1$ intervals of equal length. That is $\psi_{0}=\underline{\inta}$, $\psi_{n+1}=\alpha$ and $\psi_{i+1}-\psi_{i} = \frac{\alpha-\underline{\inta}}{n-1}$ for all $0\leq i \leq n$. Now, from Theorem \ref{ABprecauchytheorem} we see that for all $n\in\naturals^{0}$ we have
\begin{align}
||\mathcal{I}^{n}(x_{\underline{\inta}})(t) - \mathcal{I}^{n+1}(x_{\underline{\inta}})(t)||_{X_{\inta}} &\leq \frac{L^{n}}{(\inta-\psi_{1})^{qn}}n^{qn}\mathcal{H}^{n}(t,||x_{\underline{\inta}}-\mathcal{I}(x_{\underline{\inta}})||_{X_{\psi_{1}}}), \nonumber \\[1em]
&\leq \frac{L^{n}}{(\inta-\psi_{1})^{qn}}n^{qn}\mathcal{H}^{n+1}(t,\|F(x_{\underline{\inta}})\|_{X_{\psi_{1}}}), \nonumber \\[1em]
&\leq \frac{L^{n}}{(\inta-\psi_{1})^{qn}}\frac{L}{(\psi_{1}-\underline{\inta})}n^{qn}\mathcal{H}^{n+1}(t,P+ \|x_{\underline{\inta}}\|_{X_{\underline{\inta}}}), \label{ABnormest1} \nonumber \\[1em]
&\leq \frac{L^{n}T^{n+1}}{(\inta-\psi_{1})^{qn}}\frac{L}{(\psi_{1}-\underline{\inta})}\frac{n^{qn}}{(n+1)!}\bigg{(}P+ \|x_{\underline{\inta}}\|_{X_{\underline{\inta}}}\bigg{)},  \nonumber \\[1em]
&\leq \frac{L^{n+1}T^{n+1}}{(\inta-\underline{\inta})^{q(n+1)}}\frac{(n+1)^{q(n+1)}}{(n+1)!}\bigg{(}P+ \|x_{\underline{\inta}}\|_{X_{\underline{\inta}}}\bigg{)}.  \nonumber
\end{align}
We now obtain the norm estimate. 
\begin{theorem}\label{ABnormesttheorem}
Let $f$ be defined by Theorem \ref{ABEUtheorem} and suppose that $\inta\in(\underline{\inta},\overline{\inta})$. Then for all $t\in [0,T]$
\begin{align}
||f(t)||_{X_{\inta}} \leq \sum_{n=0}^{\infty}\frac{L^{n}T^{n}}{(\inta-\underline{\inta})^{qn}}\frac{n^{qn}}{n!}\bigg{(}P+ \|x_{\underline{\inta}}\|_{X_{\underline{\inta}}}\bigg{)}.
\end{align}
\end{theorem}
\begin{proof}
From Theorem \ref{ABfixpointtheorem} it is clear by continuity of $f$ that for all $t\in [0,T]$ we have
\begin{align}
\overbrace{\ \lim_{n\to\infty}\|\mathcal{I}^{n}(x_{\underline{\inta}})(t)\|_{X_{\inta}}\ }^{\text{in} \ X_{\inta}} = \|f(t)\|_{X_{\inta}}.
\end{align}
Hence we now use estimate (\ref{ABnormest1}) to see that for all $n\in\naturals$ and all $t\in [0,T]$ we have
\begin{align}
\|\mathcal{I}^{n}(x_{\underline{\inta}})(t)\|_{X_{\inta}} - \|\mathcal{I}^{0}(x_{\underline{\inta}})(t)\|_{X_{\inta}} &= \sum_{k=1}^{n}\|\mathcal{I}^{k}(x_{\underline{\inta}})(t)\|_{X_{\inta}} - \|\mathcal{I}^{k-1}(x_{\underline{\inta}})(t)\|_{X_{\inta}}, \nonumber \\[1em]
&\leq \sum_{k=1}^{n}\|\mathcal{I}^{k-1}(x_{\underline{\inta}})(t) - \mathcal{I}^{k}(x_{\underline{\inta}})(t) \|_{X_{\inta}}, \nonumber \\[1em]
&\leq \sum_{k=1}^{n}\frac{L^{k}T^{k}}{(\inta-\underline{\inta})^{qn}}\frac{k^{qk}}{k!}\bigg{(}P+ \|x_{\underline{\inta}}\|_{X_{\underline{\inta}}}\bigg{)}.
\end{align}
Therefore for all $n\in\naturals$ and all $t\in [0,T]$ we have
\begin{align}
\|\mathcal{I}^{n}(x_{\underline{\inta}})(t)\|_{X_{\inta}}  &\leq \|\mathcal{I}^{0}(x_{\underline{\inta}})(t)\|_{X_{\inta}} + \sum_{k=1}^{n}\frac{L^{k}T^{k}}{(\inta-\underline{\inta})^{qn}}\frac{k^{qk}}{k!}\bigg{(}P+ \|x_{\underline{\inta}}\|_{X_{\underline{\inta}}}\bigg{)}, \nonumber \\[1em]
&\leq P + \|x_{\underline{\inta}}\|_{X_{\underline{\inta}}} + \sum_{k=1}^{n}\frac{L^{k}T^{k}}{(\inta-\underline{\inta})^{qn}}\frac{k^{qk}}{k!}\bigg{(}P+ \|x_{\underline{\inta}}\|_{X_{\underline{\inta}}}\bigg{)}, \nonumber \\[1em]
&\leq \bigg{(}1+ \sum_{k=1}^{n}\frac{L^{k}T^{k}}{(\inta-\underline{\inta})^{qn}}\frac{k^{qk}}{k!}\bigg{)}\bigg{(}P+ \|x_{\underline{\inta}}\|_{X_{\underline{\inta}}}\bigg{)}, \nonumber \\[1em]
&\leq \sum_{k=0}^{n}\frac{L^{k}T^{k}}{(\inta-\underline{\inta})^{qn}}\frac{k^{qk}}{k!}\bigg{(}P+ \|x_{\underline{\inta}}\|_{X_{\underline{\inta}}}\bigg{)}. \label{ABnormesttheorem1} \nonumber
\end{align}
Finally taking the limit on both sides of inequality (\ref{ABnormesttheorem1}) we see that for all $t\in [0,T]$ we have
\begin{align}
||f(t)||_{X_{\inta}} \leq \sum_{n=0}^{\infty}\frac{L^{n}T^{n}}{(\inta-\underline{\inta})^{qn}}\frac{n^{qn}}{n!}\bigg{(}P+ \|x_{\underline{\inta}}\|_{X_{\underline{\inta}}}\bigg{)}, \nonumber
\end{align}
hence the proof is complete. 
\end{proof}
\begin{mdframed}
	\begin{remark}\label{ABnormesttheorem123}
	It is clear from the definition (\ref{ABdefP}) that if $F$ is a linear map then $P\equiv 0$ hence in this case from Theorem \ref{ABnormesttheorem} we see that for all $\inta\in(\underline{\inta},\overline{\inta})$.
	\begin{align} 
	||f(t)||_{X_{\inta}} \leq \sum_{n=0}^{\infty}\frac{L^{n}T^{n}}{(\inta-\underline{\inta})^{qn}}\frac{n^{qn}}{n!} \|x_{\underline{\inta}}\|_{X_{\underline{\inta}}}.
	\end{align}
	\end{remark}
\end{mdframed}
		\begin{mdframed}
	\begin{remark}
		It is aso clear from Theorem \ref{ABnormesttheorem} that 
		\begin{align}
		||f||_{Z_{\inta}} \leq \sum_{n=0}^{\infty}\frac{L^{n}T^{n}}{(\inta-\underline{\inta})^{qn}}\frac{n^{qn}}{n!}\bigg{(}P+ \|x_{\underline{\inta}}\|_{X_{\underline{\inta}}}\bigg{)}.
		\end{align}
	\end{remark}
\end{mdframed}

\end{document}